\documentclass[11pt,reqno]{amsart}
\usepackage{amssymb}
\usepackage{amsmath}
\usepackage{amsthm}
\usepackage{graphicx}
\usepackage{caption}
\usepackage{bm}
\usepackage[colorlinks=true,urlcolor=blue,
citecolor=red,linkcolor=blue,linktocpage,pdfpagelabels,
bookmarksnumbered,bookmarksopen]{hyperref}
\usepackage[titletoc,title]{appendix}
\numberwithin{equation}{section} 
\newcounter{cont}[section] 

\newtheorem{thm}[cont]{Theorem}
\newtheorem{prop}[cont]{Proposition}
\newtheorem{lem}[cont]{Lemma}

\theoremstyle{definition}

 \theoremstyle{remark}
 \newtheorem{rem}[cont]{Remark}

\newcommand{\R}{\mathbb{R}}
\newcommand{\e}{\varepsilon}
\newcommand{\T}{T_{\max}}

\begin{document}
\baselineskip=16pt

\title[Motion of interfaces for a damped hyperbolic Allen--Cahn equation]{Motion of interfaces for a damped hyperbolic Allen--Cahn equation}

\author[R. Folino]{Raffaele Folino}
\address[Raffaele Folino]{Dipartimento di Ingegneria e Scienze dell'Informazione e Matematica, Universit\`a degli Studi dell'Aquila (Italy)}
\email{raffaele.folino@univaq.it}
	
\author[C. Lattanzio]{Corrado Lattanzio}
\address[Corrado Lattanzio]{Dipartimento di Ingegneria e Scienze dell'Informazione e Matematica, Universit\`a degli Studi dell'Aquila (Italy)}
\email{corrado@univaq.it} 
 
\author[C. Mascia]{Corrado Mascia}
\address[Corrado Mascia]{Dipartimento di Matematica,
Sapienza Universit\`a di Roma (Italy)}
\email{mascia@mat.uniroma1.it}

\keywords{Allen--Cahn equation; motion by mean curvature; energy estimates}

\maketitle


\begin{abstract} 
Consider the Allen--Cahn equation $u_t=\e^2\Delta u-F'(u)$, where $F$ is a double well potential with wells of equal depth, located at $\pm1$.
There are a lot of papers devoted to the study of the limiting behavior of the solutions as the diffusion coefficient $\e\to0^+$,
and it is well known that, if the initial datum $u(\cdot,0)$ takes the values $+1$ and $-1$ in the regions $\Omega_+$ and $\Omega_-$, 
then the \emph{interface} connecting $\Omega_+$ and $\Omega_-$ moves with normal velocity equal to the sum of its principal curvatures,
i.e. the interface moves by mean curvature flow.

This paper concerns with the motion of the inteface for a damped hyperbolic Allen--Cahn equation, in a bounded domain of $\R^n$, for $n=2$ or $n=3$.
In particular, we focus the attention on radially simmetric solutions, studying in detail the differences with the classic parabolic case,
and we prove that, under appropriate assumptions on the initial data $u(\cdot,0)$ and $u_t(\cdot,0)$, 
the interface moves by mean curvature as $\e\to0^+$ also in the hyperbolic framework.
\end{abstract}

\section{Introduction}\label{sec:intro}
The aim of this paper is to analyze the behavior of the solutions to the \emph{nonlinear damped hyperbolic Allen--Cahn equation}
\begin{equation}\label{eq:hypalca-intro}
	\tau u_{tt}+g(u)u_t=\e^2\Delta u-F'(u), \qquad \quad x\in\Omega, t>0,
\end{equation}
in a bounded domain $\Omega\subset\R^n$, $n=2$ or $3$, which has a $C^1$ boundary, 
with appropriate boundary conditions and initial data $u(\cdot,0)=u_0$ and $u_t(\cdot,0)=u_1$ in $\Omega$.
We will specify later the precise assumptions on the functions $F,g$;
from now, we say that $g$ is a (smooth) strictly positive function and $F$ is a double well potential with wells of equal depth.
The main example we have in mind is $F(u)=\frac14(u^2-1)^2$ and so 
the reaction term in the equation \eqref{eq:hypalca-intro} is equal to $u-u^3$.
The \emph{relaxation parameter} $\tau$ and the diffusion coefficient $\e$ are strictly positive
and we consider the case when $\e$ is small.
Indeed, our interest is in the limiting behavior of the solutions to \eqref{eq:hypalca-intro} as $\e\to0$.

Equation \eqref{eq:hypalca-intro} is a hyperbolic variant of the classic Allen--Cahn equation
\begin{equation}\label{eq:allencahn-multiD}
	u_t=\e^2\Delta u-F'(u),
\end{equation}
which is obtained from \eqref{eq:hypalca-intro} in the (formal) limit $\tau\to0$ when $g\equiv1$.
The latter equation is a classic reaction-diffusion equation with a reaction term of bistable type,
and it has been proposed in \cite{Allen-Cahn} to describe the motion of antiphase boundaries in iron alloys.
The reaction function $F$ has two global minimal points, that correspond to two stable stationary solutions of the equation \eqref{eq:allencahn-multiD}.
In this paper, we assume that the only global minimal points of $F$ are $-1$ and $+1$.

In general, reaction-diffusion equations are widely used to describe a variety of phenomena 
such as pattern formation and front propagation in biological, chemical and physical systems.
However, such equations undergo the same criticisms of the linear diffusion equation,
mainly concerning the infinite propagation speed of disturbances and lack of inertia.
There are many ways to overcome these unphysical properties;
one of them is to consider hyperbolic reaction-diffusion equations like \eqref{eq:hypalca-intro}.
In particular, substituting the classic Fick law with a relaxation relation of Maxwell--Cattaneo type, 
one obtains the hyperbolic reaction-diffusion equation \eqref{eq:hypalca-intro} with $g=1+\tau F''$.
For a complete discussion on the derivation of the model \eqref{eq:hypalca-intro} 
and on the physical or biological details see \cite{Cat,DunbOthm86,HerPav,Hillen,Holmes,JP89a,JP89b,MenFedHor}.  

As it was previously mentioned, we are interested in the limiting behavior of the solutions
as the diffusion coefficient $\e\to0^+$.
In the one dimensional case, it is well known that equation \eqref{eq:allencahn-multiD} exhibits the phenomenon of \emph{metastability}.
If we consider equation \eqref{eq:allencahn-multiD} in a bounded domain with appropriate boundary conditions,
then we have the persistence of unsteady structure for a very long time.
Indeed, the only stable states are the constant solutions $-1$ or $+1$ (the global minimal points of the potential $F$),
but it has been proved that if the initial profile has a $N$-transition layer structure, i.e. it is approximately constant to $-1$ or $+1$
except close to $N$ transition points, then the solution maintains that structure for an exponentially long time, 
namely a time proportional to $\exp\left(Al/\e\right)$, where $A$ is a positive constant depending only on $F$ 
and $l$ is the minimum distance between the transition points.
There are many papers devoted to the study of the metastability for the Allen--Cahn equation;
here we recall the fundamental contributions \cite{Bron-Kohn,Carr-Pego,Carr-Pego2,Chen2}.
In particular, in \cite{Carr-Pego} the authors studied in details the motion of the $N$ transition points and derived a system of ODE
describing their dynamics; the transition layers move with an exponentially small velocity and so we have the persistence of the 
transition layer structure for an exponentially long time.

Similar results are also valid for the one dimensional version of \eqref{eq:hypalca-intro},
and then we have the phenomenon of the metastability also in the hyperbolic framework \eqref{eq:hypalca-intro}.
The study of the metastable properties of the solutions and the differences with the classic parabolic case \eqref{eq:allencahn-multiD}
are performed in \cite{FolinoJHDE,FolinoDIE,FolinoHyp2016,FLM-CMS}.
In particular, in \cite{FLM-CMS} using a similar approach of \cite{Carr-Pego} it has been derived a system of ODE describing the motion of the transition points
and a comparison with the classic case is perfomed.
In conclusion, both equation \eqref{eq:hypalca-intro} and equation \eqref{eq:allencahn-multiD} exhibit the phenomenon of metastability in the one dimensional case:
in both cases we have persistence of a transition layer structure for an exponentially long time and
the dynamics of such solutions is described by a finite dynamical system.

This paper concerns with the multidimensional case, where the situation is rather different.
Indeed, in this case we have to study the motion of ``transition surfaces'' instead of transition points.
There is vast literature of works about motion of interfaces in several space dimensions for the Allen--Cahn equation \eqref{eq:allencahn-multiD}, 
where the effect of the curvature of the interfaces turns out to be relevant for the dynamics, 
and it has been shown that steep interfaces are generated in a short time with subsequent motion governed by mean curvature flow.
It is impossible to quote all the contributions; 
without claiming to be complete, we recall the papers \cite{Bron-Kohn2, Chen2, demot-sch,Ev-So-Sou}.
The behavior of the solutions to equation \eqref{eq:allencahn-multiD} for $\e$ small can be described as follows:
for a short time the solution $u^\e$ behaves as if there were no diffusion, i.e. $\e=0$, 
and so, $u^\e\approx\pm1$ according to the sign of the initial datum.  
Therefore, we can divide the domain where we are considering the equation in three different regions:
two regions $\Omega_+$, $\Omega_-$ where $u^\e\approx+1$ and $u^\e\approx-1$, respectively,
and a ``thin'' region $\Omega_0$ which connects $\Omega_+$ and $\Omega_-$.
The region $\Omega_0$ is usually referred as \emph{interface} and the process described above is called \emph{generation of interface}.
After this phase of the dynamics, if $x$ is away of the interface, the diffusion term $\e^2\Delta u$ can still be neglected,
and $u^\e$ takes the values $\pm1$ in $\Omega_\pm$.
On the other hand, close to the interface, when the gradient of $u^\e$ is large enough,
the diffusion term plays a crucial role: it balances the reaction term $-F'$ and we have the \emph{propagation of the interface}.
In this phase, the mean curvature $K$ of the interface plays a fundamental role, 
indeed the interface propagates with normal velocity proportional to the mean curvature $K$, namely 
\begin{equation}\label{eq:meancurvature}
	V=\e^2 K,
\end{equation}
where $V$ is the normal velocity of the interface, and the mean curvature $K$ is the sum of its principal curvatures.
The link between the equation \eqref{eq:allencahn-multiD} and the motion by mean curvature was firstly observed
by Allen and Cahn in \cite{Allen-Cahn} on the basis of a formal analysis.
Another formal asymptotic expansion is performed in \cite{RubSteKel}.
In \cite{Bron-Kohn2, Chen2, demot-sch,Ev-So-Sou} the authors studied in details the process described above 
and proved rigorously that the formal analysis is correct.
In particular, in \cite{Bron-Kohn2} the authors consider a rescaled version of \eqref{eq:allencahn-multiD} with $F(u)=\frac14(u^2-1)^2$,
namely
\begin{equation}\label{eq:allencahn-rescaled}
	u_t=\Delta u+\e^{-2}(u-u^3),
\end{equation}
with appropriate boundary conditions and initial data, and they present two rigorous results.
Firstly, they prove a compactness theorem: as $\e\to0$, the solution $u^\e$ is in a certain sense compact as function of space-time and 
the limit is a function assuming only the values $\pm1$.
Secondly, they focused the attention on radially symmetric solutions, and proved that 
if $\Omega$ is a ball, the initial datum is radial with one transition sphere between $-1$ and $+1$ at $r=\rho_0$,
and the boundary conditions are of Dirichlet type, then the transition at time $t$ is $r=\rho(t)$, where $\rho$ satisfies
\begin{equation}\label{eq:meancurvature_sphere}
	\rho'=-\frac{n-1}{\rho}, \qquad \rho(0)=\rho_0.
\end{equation}
Therefore, they show that the motion of the interface is governed by mean curvature flow in the case of radial solution.
Indeed, it is well known that the evolution by mean curvature for general spheres in $\mathbb R^n$ is governed by the law \eqref{eq:meancurvature_sphere}
and the sphere shrinks into a point in finite time.
The scaling of the equation \eqref{eq:allencahn-multiD} has been chosen so that the associated motion by mean curvature takes place on a time scale of order one,
and so the sphere shrinks into a point in finite time which does not depend on $\e$.
This implies that the solution of \eqref{eq:allencahn-multiD} has one transition between $-1$ and $+1$ for a time proportional to $\e^{-2}$,
and then we have a fundamental difference with respect to the one dimensional case, where the solution maintains the transition layer structure
for an exponentially long time.
We remark that, in the case of the rescaled version \eqref{eq:allencahn-rescaled}, the law for the normal velocity \eqref{eq:meancurvature} becomes
\begin{equation}\label{eq:meancurvature-rescaled}
	V=K,
\end{equation}
where $K$ is again the mean curvature of the interface.
From now on, \emph{faster time scale} is referred to the rescaled version \eqref{eq:allencahn-rescaled},
and  \emph{slower time scale} to \eqref{eq:allencahn-multiD}.
Therefore, in the faster time scale the interface propagates with normal velocity equal to \eqref{eq:meancurvature-rescaled},
whereas in the slower time scale with normal velocity equal to \eqref{eq:meancurvature}.

The contributions \cite{Chen2, demot-sch,Ev-So-Sou} deal with the equation \eqref{eq:allencahn-multiD} in the whole space, without the assumption of radial symmetry.
Chen \cite{Chen2} studies generation and propagation of the interface, showing that in the faster time scale, 
the interface develops in a short time $O(\e^2|\ln\e|)$ and disappears in a finite time.
De Mottoni and Schatzman \cite{demot-sch} obtain similar results by means of completely different techniques;
they consider the slower time scale with an initial data which  has an interface, 
and study the motion of the interface giving an asymptotic expansion of arbitrarily high order and error estimates valid up to time $O(\e^{-2})$. 
At lowest order, the interface evolves normally, with a velocity proportional to the mean curvature.
All the previous papers treat the dynamics of the solutions before the appearance of geometric singularities;
the main accomplishment of \cite{Ev-So-Sou} is the verification of the fact that the interface evolves according to mean curvature motion for all positive time, 
and so even beyond the time of appearance of singularities.
In the latter paper, the motion is interpreted in the generalized sense of Evans--Spruck \cite{Ev-Sp} 
and Chen--Giga--Goto \cite{Ch-Gi-Go} after the onset of geometric singularities.
Let us stress that the proofs of \cite{Chen2, demot-sch,Ev-So-Sou} rely heavily on the maximum principle for parabolic equation. 

The aforementioned bibliography is confined to the parabolic case \eqref{eq:allencahn-multiD}.
To the best of our  knowledge, the only paper devoted to the study to the same problem for hyperbolic variations of \eqref{eq:allencahn-multiD}
is \cite{Hil-Nara}, where the authors study the singular limit of \eqref{eq:hypalca-intro} when $g$ is constant, in the whole space $\mathbb{R}^n$ for $n=2$ or $n=3$. 
The authors derive estimates for the generation and the propagation of interfaces and prove that the motion is governed by mean curvature flow
in the limit $\e\to0$ under the assumption that the damping coefficient is sufficiently strong.
Their proofs use a comparison principle for a damped wave equation and a construction of suitable subsolutions and supersolutions.
The comparison principle is obtained by expressing the solutions by Kirchhoff's formula and estimating them.

In this paper, we study the propagation of the interface of \eqref{eq:hypalca-intro} in a bounded domain, by following the approach introduced in \cite{Bron-Kohn2}.
Therefore, after rescaling the equation to study the motion of the interface on a time scale of order one, we first prove a compactness theorem, Theorem \ref{thm:compactness}, 
valid for any sufficiently regular domain $\Omega$, any positive function $g$ and for appropriate boundary conditions of Dirichlet or Neumann type.
Next, we focus the attention on the radial case with $g\equiv1$ and Dirichlet boundary conditions as in \cite{Bron-Kohn2}.
As an intermediate result, we will prove that for some radially symmetric solutions with one transition sphere at time $t=0$, 
the motion of the transition sphere can be described by the ODE
\begin{equation}\label{eq:rho-intro}
	\e^2\tau\rho''+\rho'=-\frac{n-1}{\rho}.
\end{equation}
As we will see in Section \ref{sec:radial}, equation \eqref{eq:rho-intro} allows us to prove that the interface moves by mean curvature as $\e\to0$.
Thus, in the hyperbolic framework \eqref{eq:hypalca-intro},
we have to take into account the inertial term $\e^2\tau\rho''$, involving also the small parameter $\e$.

Let us now show some numerical solutions of the equation \eqref{eq:hypalca-intro}
where $F(u)=\frac14(u^2-1)^2$ and $\Omega=\left\{x\in \R^2 : |x|\leq1\right\}$,
with Dirichlet boundary conditions $u(x,t)=1$ for all $t\geq0$ on $\partial\Omega$.
The initial datum is as in Figure \ref{fig:t=0}, it is smooth and has the transition at $\rho_0=0.6$.
Precisely, the initial datum $u_0(r)$ is equal to $+1$ when $r>0.6$ (red region), and it is equal to $-1$ in the blue region.
\begin{figure}[htbp]
\centering
\includegraphics[scale=.2]{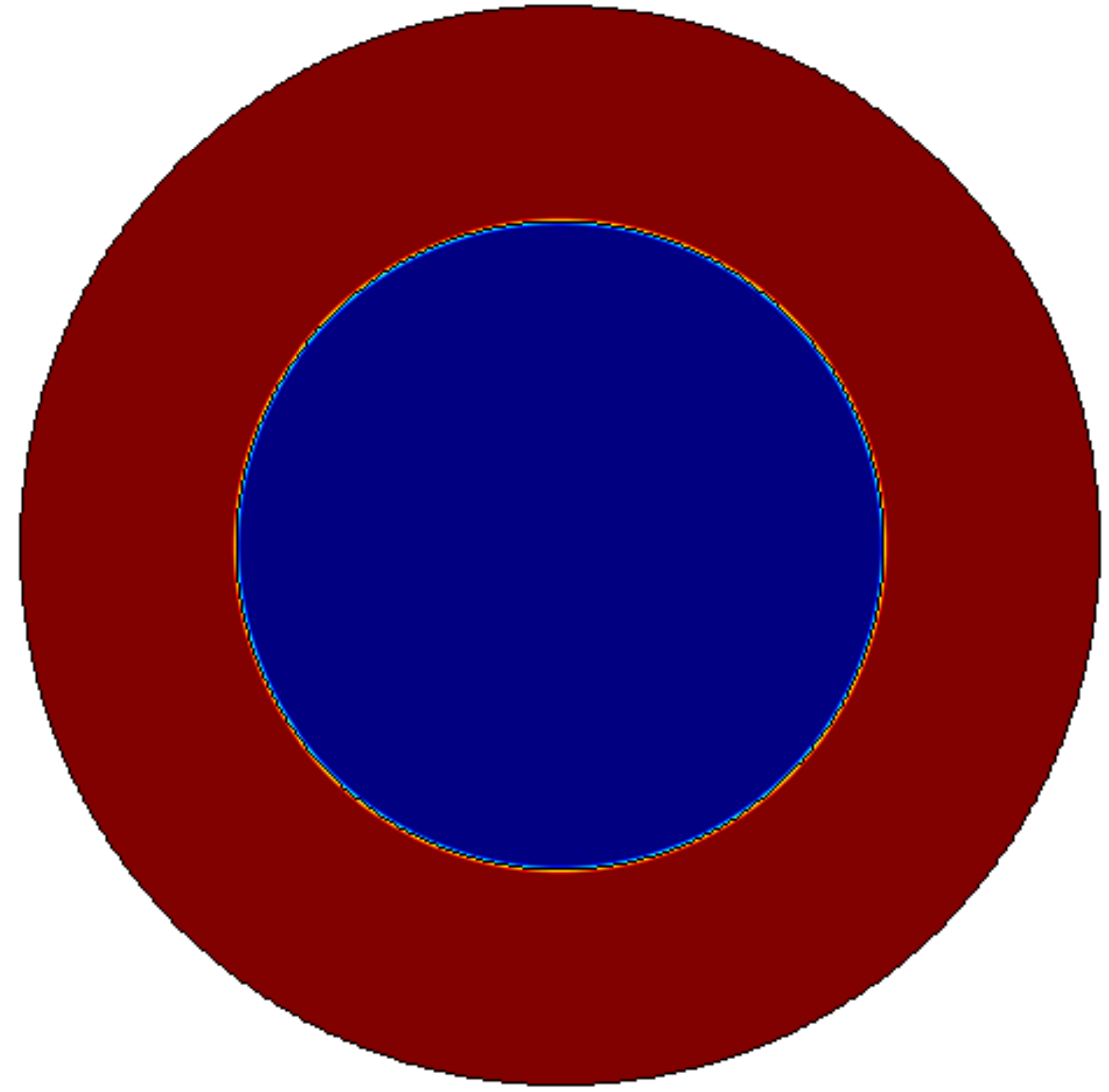}
\caption{Initial datum $u_0$ with transition at $\rho_0=0.6$.}
\label{fig:t=0}
\end{figure}

In the following examples, we choose the parameter $\tau=1$ and the initial velocity $u_1\equiv0$.
Firstly, we consider $\e=0.02$ and show the solution for different values of $t$ in Figure \ref{fig:eps=0.02}.
We see that the solution maintains the interface until the time $t=450$.
\begin{figure}[htbp]
\centering
\includegraphics[scale=.25]{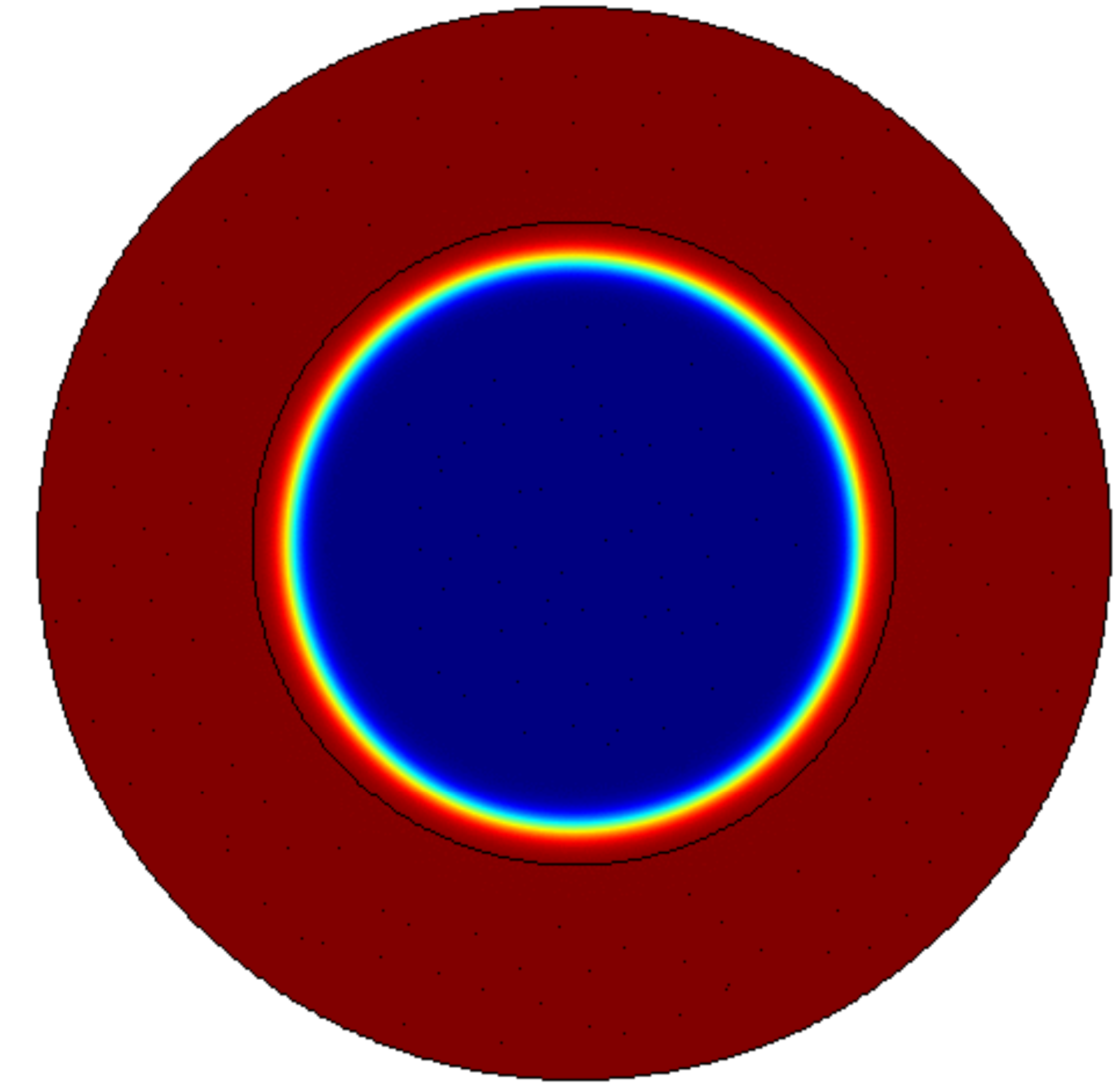}\qquad \quad
\includegraphics[scale=.25]{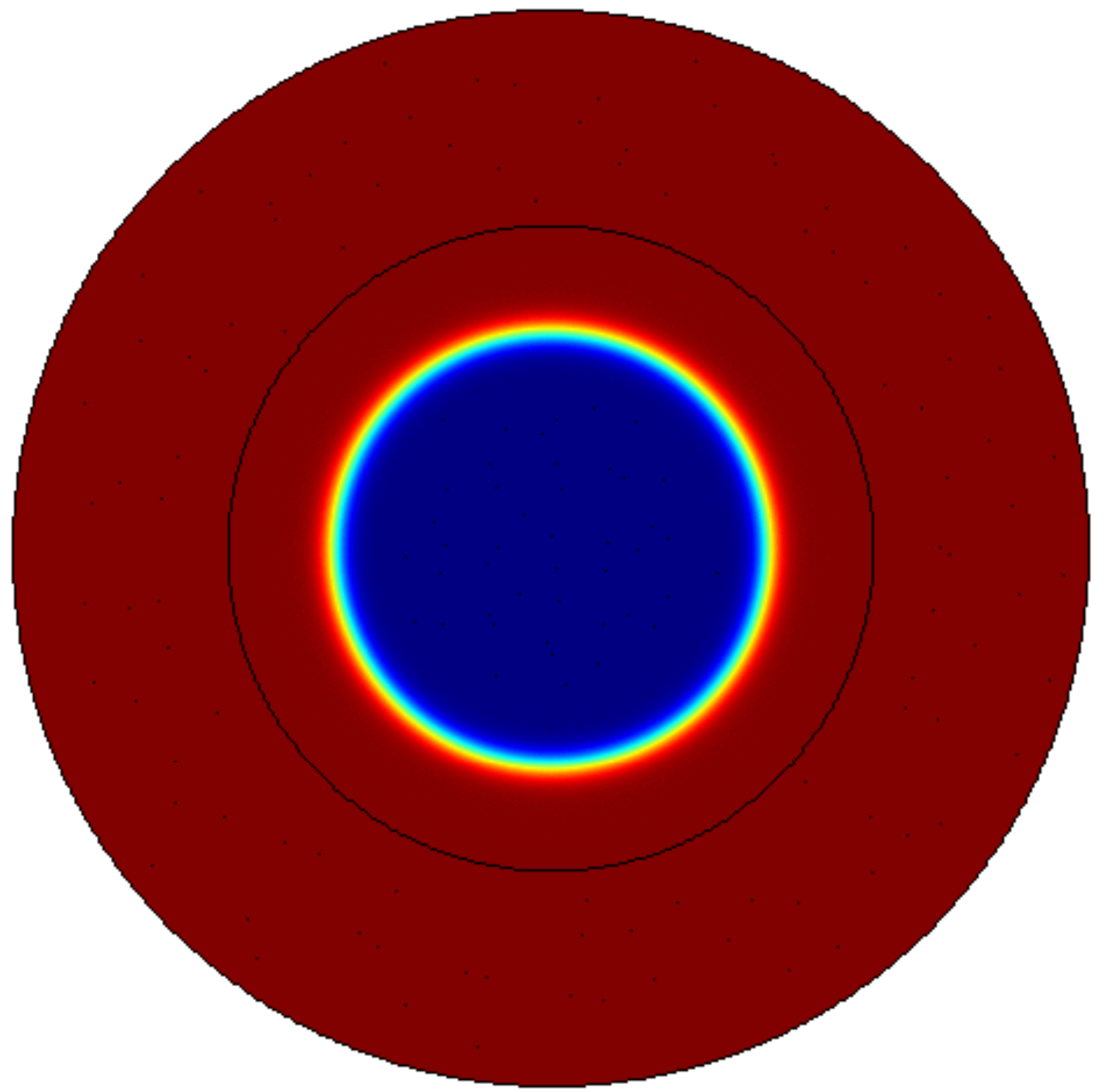}\\ \vspace{1cm}
\includegraphics[scale=.25]{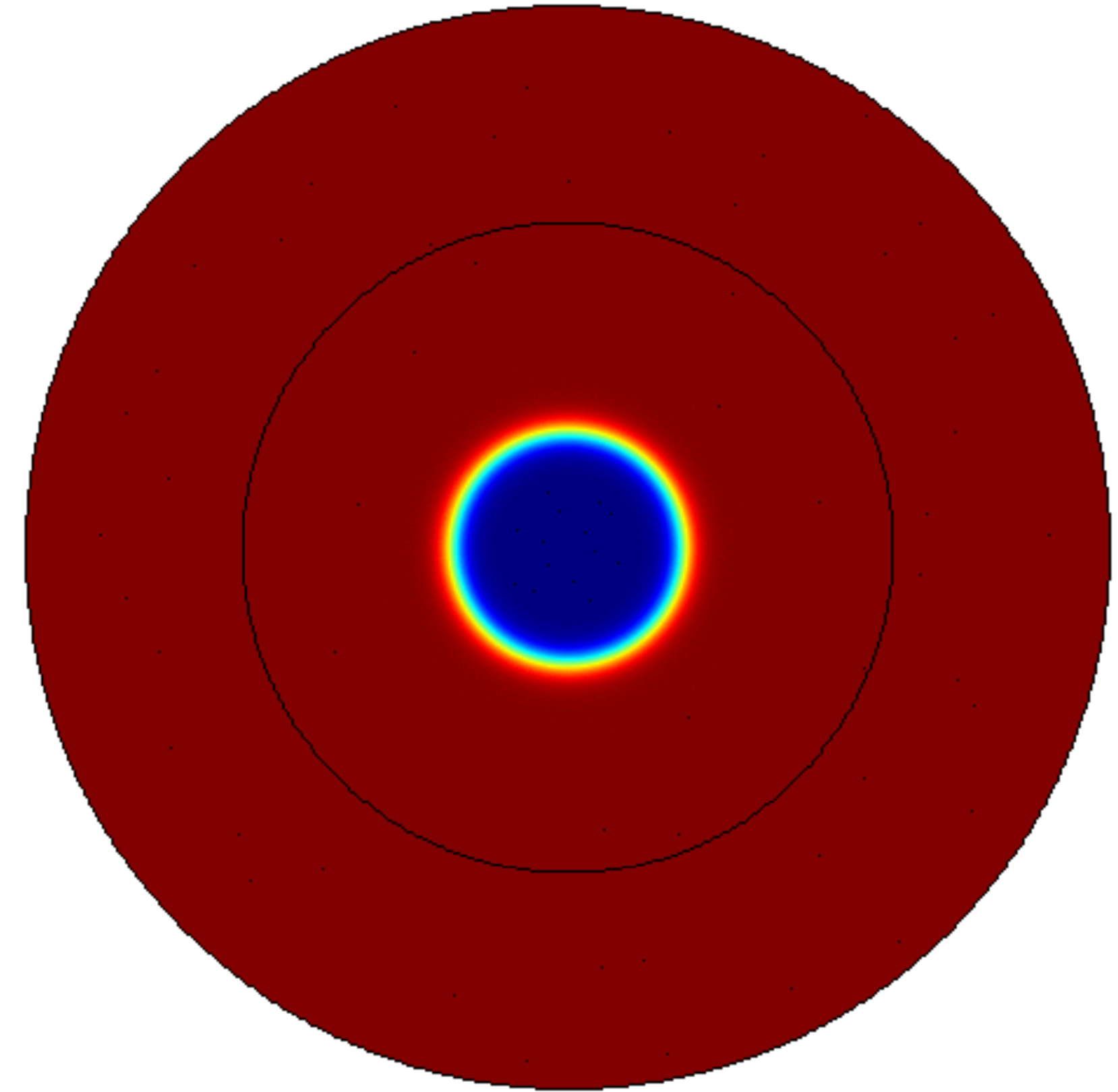}\qquad \quad
\includegraphics[scale=.25]{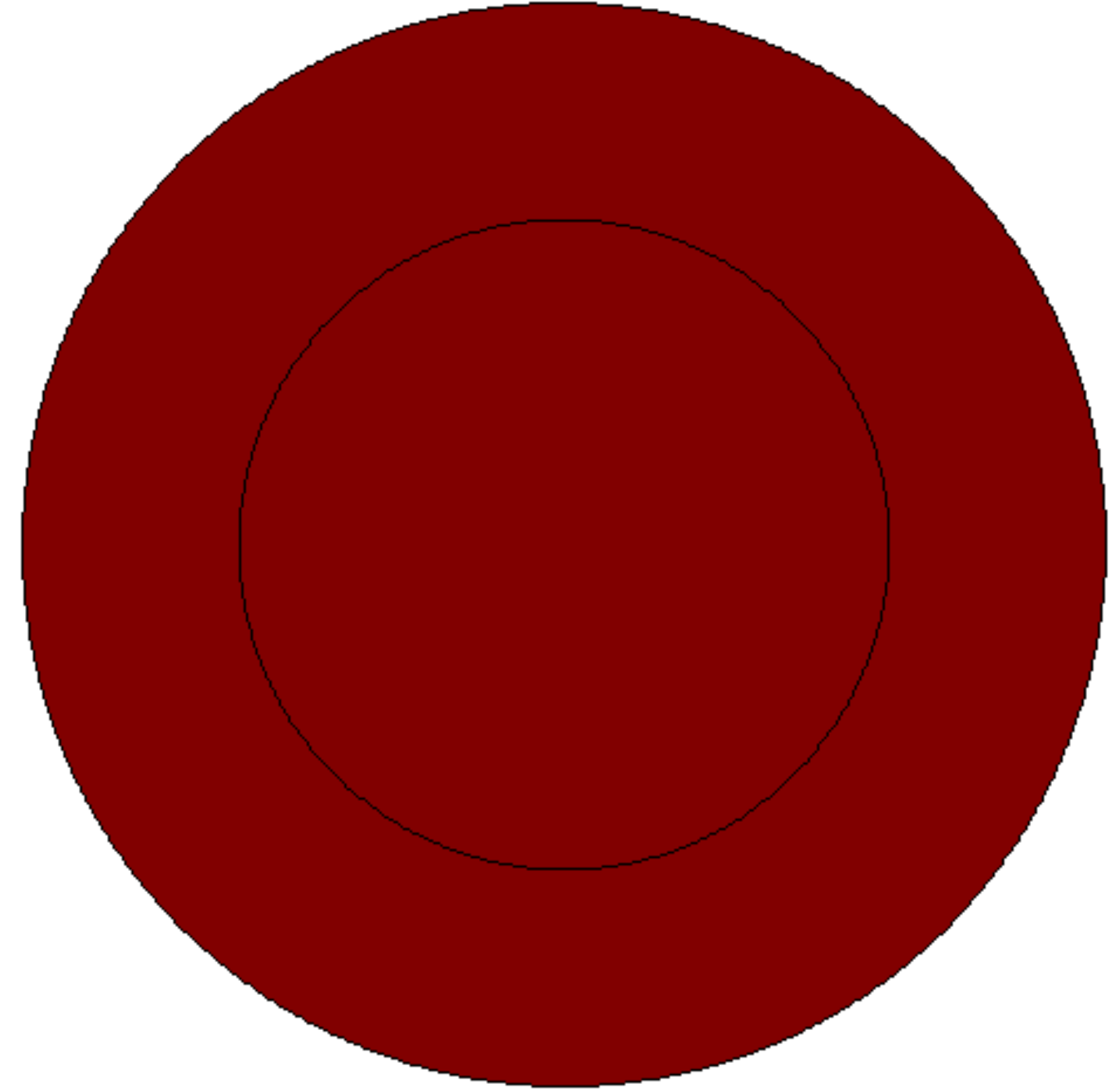}
\caption{Solution for $\tau=1$, $\e=0.02$ and different values of $t$. Top left: $t=100$, top right: $t=250$, bottom left: $t=400$, bottom right: $t=450$. The initial datum is as in Figure \ref{fig:t=0}.}
\label{fig:eps=0.02}
\end{figure}

Next, we take $\e=0.01$ and, since we are considering equation \eqref{eq:hypalca-intro} and then the slower time scale,
the evolution of the solution is slower than the previous case (see Figure \ref{fig:eps=0.01_t=250}).
\begin{figure}[htbp]
\centering
\includegraphics[scale=.25]{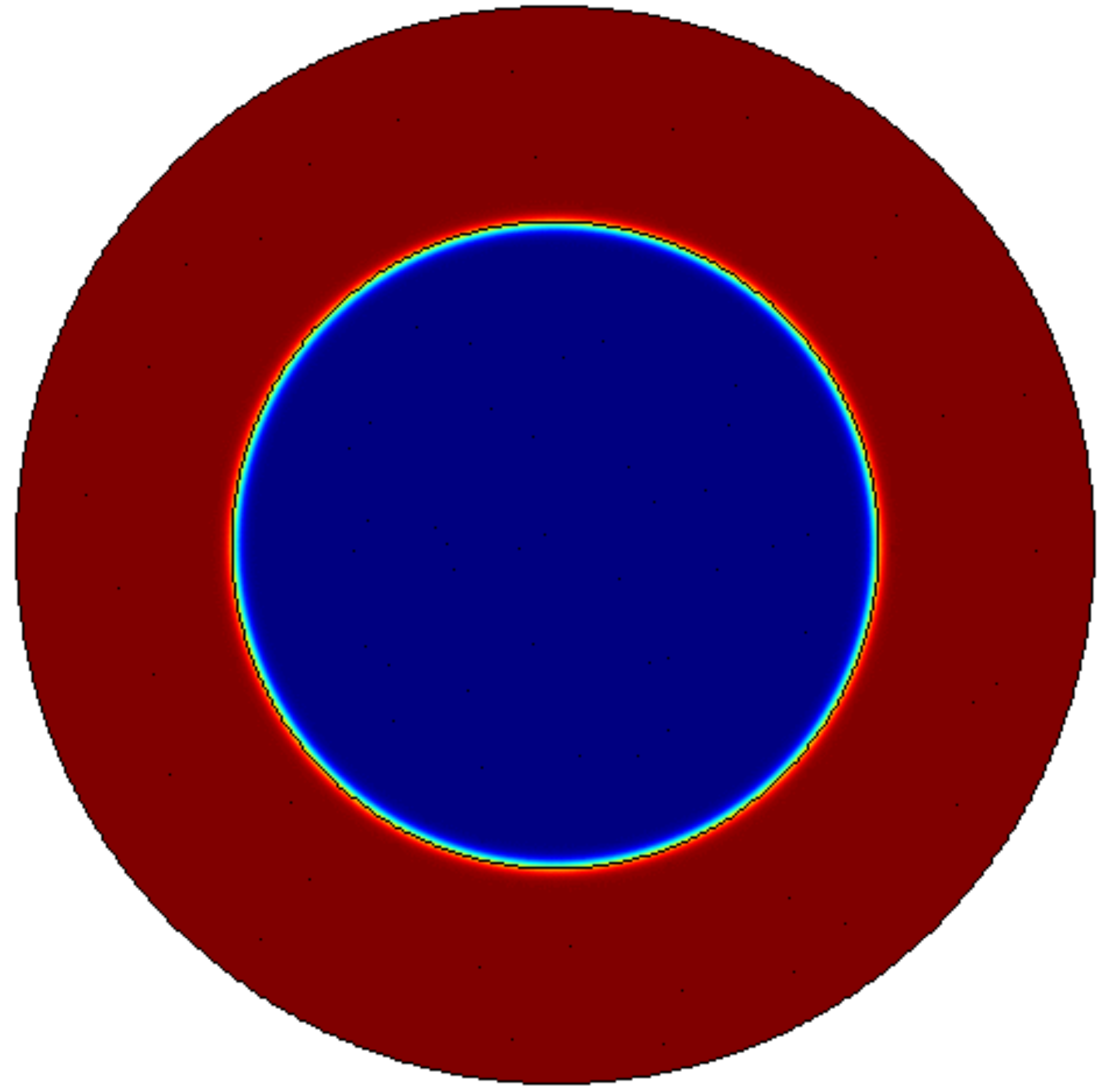}\qquad\quad
\includegraphics[scale=.25]{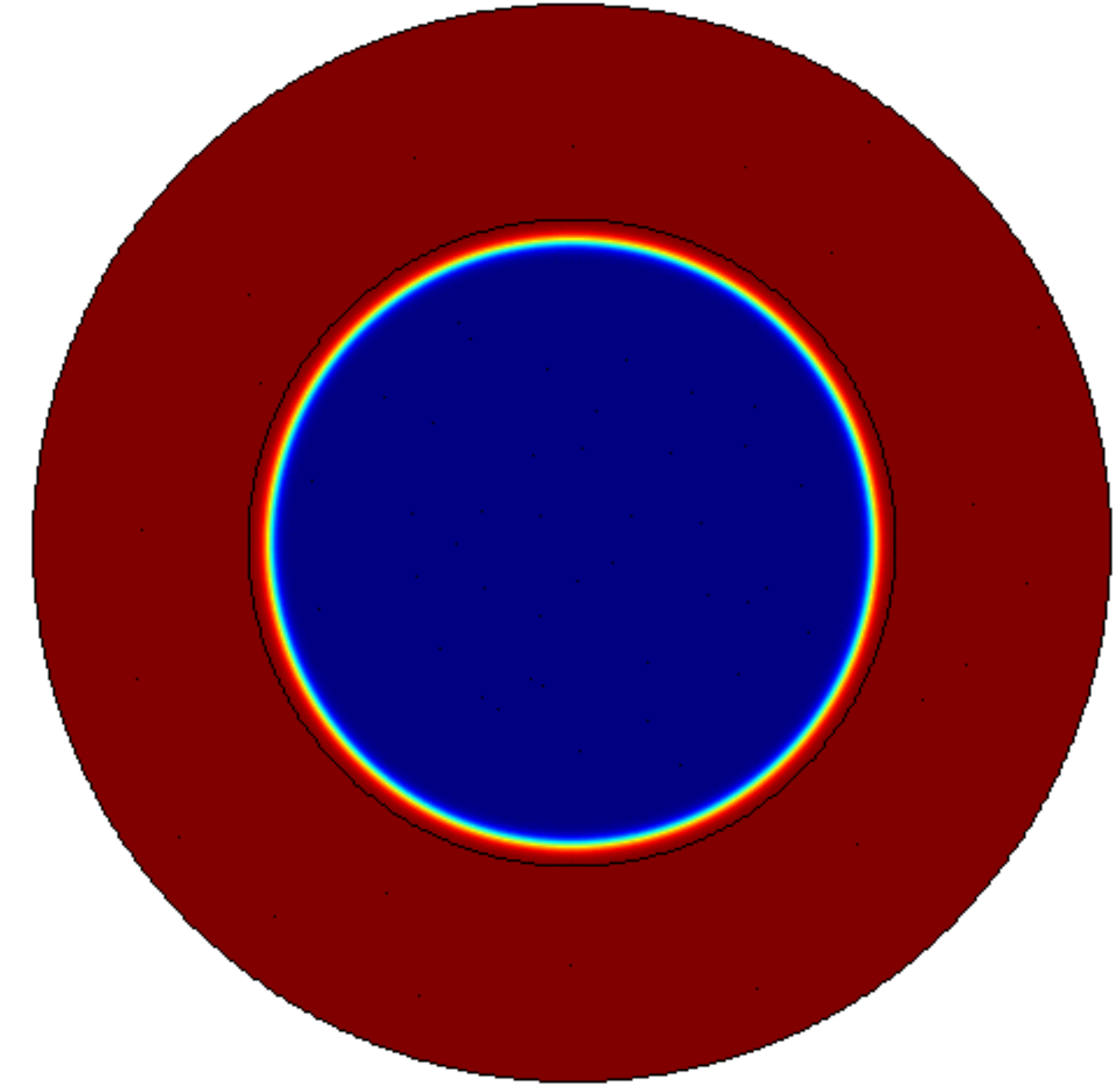}
\caption{Solution for $\tau=1$, $\e=0.01$ and different values of $t$. Left: $t=50$, right: $t=250$. The initial datum is as in Figure \ref{fig:t=0}.}
\label{fig:eps=0.01_t=250}
\end{figure}

The goal of this paper is to rigorously describe the behavior of the solutions shown in Figures \ref{fig:eps=0.02} and \ref{fig:eps=0.01_t=250}; 
we consider the faster time scale, and so a rescaled version of \eqref{eq:hypalca-intro} as in \cite{Bron-Kohn2}, 
to study the motion of the interface on a time scale of order one.
The interface will disappear in a finite time (independent on $\e$) and we study the motion of the interface before it disappears.
To this aim, we consider well-prepared initial data: an initial datum $u_0$ which makes the transition from $-1$ to $+1$ in an ``energetical efficient'' way, 
and an initial velocity $u_1$ sufficiently small in an appropriate sense; for details, see assumptions in Section \ref{sec:radial} and Remark \ref{rem:initial}.

Our results confirm that the motion of the interface is governed by mean curvature flow 
as $\e\to0$ in the radial case and for $g\equiv1$, see Theorem \ref{cor:main}.
In general, a formal computation shows that the interfaces move by mean curvature for any strictly positive function $g$ without the assumption of radially symmetric solutions.
In particular, we will (formally) show that the normal velocity satisfies 
\begin{equation}\label{eq:meancurvature-g}
	\overline{g}\,V=K,
\end{equation} 
in the limit $\e\to0$, where $K$ is the mean curvature of the interface and
\begin{equation*}
	\overline{g}:=\frac{1}{\|\sqrt{F}\|_{{}_{L^1}}}\int_{-1}^{1} \sqrt{F(s)}\,g(s)\,ds.
\end{equation*}
Thus, in the case $g\equiv1$, the asymptotic limit \eqref{eq:meancurvature-g} is equal to \eqref{eq:meancurvature-rescaled} in the faster time scale.

The rest of the paper is organized as follows.
In Section \ref{sec:lim} we consider the IBVP for equation \eqref{eq:hypalca-intro} with a generic strictly positive function $g$, 
in a generic domain $\Omega$ and with boundary conditions of Dirichlet or Neumann type.
The main result of the section is Theorem \ref{thm:compactness}, that is the compactness theorem we discussed above.
Moreover, in Section \ref{sec:lim} we deduce an estimate on the $H^1$--norm of the time derivative $u_t$ of the solution,
that we will use in the study of the radial case (see Proposition \ref{prop:higher}).
Finally, Section \ref{sec:lim} contains the formal computation suggesting that the motion of the interface is governed by mean curvature flow as $\e\to0$,
and that the normal velocity satisfies \eqref{eq:meancurvature-g} in the asymptotic limit $\e\to0$.

In Section \ref{sec:radial}, we focus the attention on the radially symmetric solutions
in the case of damping coefficient $g\equiv1$ with boundary conditions of Dirichlet type,
and prove that the interface moves by mean curvature flow in the singular limit $\e\to0$,
see Theorems \ref{thm:main} and \ref{cor:main}.

\section{Limiting behavior as $\e\to0$ in the general case}\label{sec:lim}
Rescale equation \eqref{eq:hypalca-intro} and consider the \emph{hyperbolic reaction-diffusion equation}
\begin{equation}\label{eq:hypalca-multiD}
	\e^2\tau u_{tt}+g(u)u_t=\Delta u+\e^{-2}f(u), \qquad \quad x \in\Omega, t>0,
\end{equation}
in a bounded domain $\Omega\subset\R^n$, $n=2$ or $3$, with a $C^1$ boundary, 
where $f,g:\R\rightarrow\R$ are regular functions satisfying appropriate assumptions, that will be specified later.
Equation \eqref{eq:hypalca-multiD} is complemented with initial data
\begin{equation}\label{eq:initialdata}
	u(x,0)=u_0(x), \qquad u_t(x,0)=u_1(x), \qquad \quad x\in\Omega,
\end{equation}
and appropriate boundary conditions.
Precisely, we consider either Neumann 
\begin{equation}\label{eq:neumann}
	\frac{\partial u}{\partial n}(x,t)=0, \qquad \, x\in\partial\Omega, \, t>0,
\end{equation}
or Dirichlet type boundary conditions
\begin{equation}\label{eq:dirichlet}
	u(x,t)=\pm1, \qquad \, x\in\partial\Omega, \, t>0.
\end{equation}
In the latter case, we assume that at the boundary $u$ takes values in $\{-1,+1\}$ in a way such that the solution is sufficiently regular.
In this section, we collect some results on the behavior of the solutions to \eqref{eq:hypalca-multiD} as $\e\to0^+$,
valid for any regular bounded domain $\Omega$ and any strictly positive function $g$.
Let us start with some comments on the well-posedness of the IBVPs introduced above.

\subsection{Existence framework}
Let us introduce the energy functional
\begin{equation}\label{eq:energy}
	E_\e[u,u_t](t):=\int_\Omega\left[\frac{\e^3\tau}{2} u^2_t(x,t)+\frac{\varepsilon}2 |\nabla u(x,t)|^2+\e^{-1}F(u(x,t))\right]dx,
\end{equation}
where $F'=-f$ and denote by $E_\e[u_0,u_1]:=E_\e[u,u_t](0)$.
Thanks to the boundary conditions \eqref{eq:neumann} or \eqref{eq:dirichlet}, 
we can state that the energy functional \eqref{eq:energy} is a non-increasing function of time along (sufficiently regular) solutions to \eqref{eq:hypalca-multiD}.
Precisely, we have the following result.

\begin{lem}
Let $(u,u_t)\in C\left([0,T],H^2(\Omega)\times H^1(\Omega)\right)$ be a solution to \eqref{eq:hypalca-multiD} with $f,g:\R\to\R$, 
$f=-F'$ for some $F:\R\to\R$ and either Neumann \eqref{eq:neumann} or Dirichlet \eqref{eq:dirichlet} boundary conditions.
Then, for any $0\leq t_1<t_2\leq T$ 
\begin{equation}\label{variazione-energia}
	\e\int_{t_1}^{t_2}\int_\Omega g(u)u^2_t dxdt=E_\e[u,u_t](t_1)-E_\e[u,u_t](t_2).
\end{equation}
\end{lem}
\begin{proof}
Multiplying \eqref{eq:hypalca-multiD} by $u_t$ and integrate on $\Omega\times[t_1,t_2]$, we infer
\begin{equation*}
	\int_{t_1}^{t_2}\int_\Omega\left(\e^2\tau u_tu_{tt}+g(u)u^2_t\right)dxdt=\int_{t_1}^{t_2}\int_\Omega\left(u_t\Delta u-\e^{-2}F'(u)u_t\right)dxdt.
\end{equation*}
Integrating by parts and using the boundary conditions \eqref{eq:neumann} or \eqref{eq:dirichlet} we deduce
\begin{equation*}
	\int_\Omega u_t\Delta u\,dx=\int_{\partial\Omega}u_t\frac{\partial u}{\partial n}\,d\sigma-\int_\Omega\nabla u\cdot\nabla u_t\,dx=-\frac d{dt}\int_\Omega\frac12|\nabla u|^2\,dx.
\end{equation*}
Since $u_tu_{tt}=\partial_tu_t^2/2$ and $F'(u)u_t=\partial_t F(u)$, we have
\begin{align*}
	\int_{t_1}^{t_2}\int_\Omega g(u)u^2_t dxdt &=\int_\Omega\left[\frac{\e^2\tau}2 u^2_t(x,t_1)-\frac{\e^2\tau}2u^2_t(x,t_2)\right]dx\\
	&+\int_\Omega\left[\frac12 |\nabla u(x,t_1)|^2-\frac12 |\nabla u(x,t_2)|^2\right]dx\\
	&+\int_\Omega\left[\e^{-2}F(u(x,t_1))-\e^{-2}F(u(x,t_2))\right]dx. 
\end{align*}
Multiplying by $\e$ and using the definition \eqref{eq:energy}, we obtain \eqref{variazione-energia}.
\end{proof}
In the rest of the paper, we shall consider a framework where the equality \eqref{variazione-energia} is satisfied,
and we assume that $(u,u_t)\in C\left([0,\infty),H^2(\Omega)\times H^1(\Omega)\right)$ is the solution to
\eqref{eq:hypalca-multiD}-\eqref{eq:initialdata} with Neumann boundary conditions \eqref{eq:neumann}, 
and $(u,u_t)\in C\left([0,\infty),H^2(\Omega)\times H^1_0(\Omega)\right)$ is the one with Dirichlet boundary conditions \eqref{eq:dirichlet}.

The complete discussion of the well-posedness of the IBVPs for equation \eqref{eq:hypalca-multiD} is beyond the scope of this paper.
However, we show how to obtain the existence of a unique solution $(u,u_t)\in C\left([0,\infty),H^2(\Omega)\times H^1(\Omega)\right)$
in the case $g\equiv1$ and with homogeneous Neumann boundary conditions 
by means of the classical semigroup theory for solutions of differential equations on Hilbert spaces (see \cite{Pazy} or \cite{Cazenave}).
Setting $y=(u,v)=(u,u_t)$, we rewrite \eqref{eq:hypalca-multiD} as  
\begin{equation}\label{eq:primo-ordine}
	y_t=Ay+\Phi(y),
\end{equation}
where 
\begin{equation}\label{eq:A-Fi}
	Ay:=\left(\begin{array}{cc} 0 & \,1 \\ \varepsilon^{-2}\tau^{-1}\Delta & \, 0 \end{array} \right)y-y
	\quad \textrm{and}\quad \Phi(y):=y+\dfrac1{\e^2\tau}\left(\begin{array}{c} 0 \\ f(u)-v \end{array}\right),
\end{equation}
and the operator $A:D(A)\subset X\rightarrow X$, with $X:=\left\{(u,v)\in H^1(\Omega)\times L^2(\Omega)\right\}$, is defined in the domain 
\begin{equation}\label{eq:dom-neu}
	D(A):=\left\{(u,v)\in H^2(\Omega)\times H^1(\Omega) : \frac{\partial u}{\partial n}=0 \; \mbox{ in }\,\partial\Omega \right\}.
\end{equation}
It can be shown (following, for example, \cite{FolinoJHDE} where the one-dimensional version of \eqref{eq:hypalca-multiD} is studied) that the linear operator $A$
defined by \eqref{eq:A-Fi}-\eqref{eq:dom-neu} is m-dissipative with dense domain, and so from the Lumer--Phillips Theorem, it follows that 
it is the generator of a contraction semigroup $(S(t))_{t\geq0}$ in $X$.
Hence, there exists a unique {\it mild solution} on $[0,T]$ of \eqref{eq:primo-ordine} with initial condition $y(0)=x\in X$,
that is a function $y\in C([0,T],X)$ solving the problem 
\begin{equation*}
	y(t)=S(t)x+\int_0^t S(t-s)\Phi(y(s))ds, \qquad \forall\,t\in[0,T],
\end{equation*}
if the function $\Phi$ defined by \eqref{eq:A-Fi} is a Lipschitz continuous function on bounded subsets of $X$,
namely
\begin{equation}\label{eq:Phi-Lip}
	\|\Phi(y_1)-\Phi(y_2)\|_X\leq C(M)\|y_1-y_2\|_X,
\end{equation}
for all $y_i\in X$ with $\|y_i\|_X\leq M$, $i=1,2$.
The latter condition holds if we assume that there exists a positive constant $C>0$ such that the function $f$ satisfies
\begin{equation}\label{eq:hyp-f}
	|f(x)-f(y)|\leq C(1+|x|^\alpha+|y|^\alpha)|x-y|, \qquad \forall\, x,y\in\mathbb{R},
\end{equation}
for some $\alpha>0$ in the case $n=2$ and for some $\alpha\in[0,2]$ in the case $n=3$.
Indeed, let $y_1=(u_1,v_1)$ and $y_2=(u_2,v_2)$, with $u_i\in H^1(\Omega)$ and $v_i\in L^2(\Omega)$, for $i=1,2$.
From the definition of $\Phi$ \eqref{eq:A-Fi}, we have
\begin{equation*}
	\|\Phi(y_1)-\Phi(y_2)\|_X  \leq \|y_1-y_2\|_X +\e^{-1}\tau^{-1/2}\left(\|f(u_1)-f(u_2)\|_{L^2}+\|v_1-v_2\|_{L^2}\right).
\end{equation*}
From \eqref{eq:hyp-f} it follows that 
\begin{align*}
	\|f(u_1)-f(u_2)\|_{L^2} & \leq C\|u_1-u_2\|_{L^p}\|1+|u_1|^\alpha+|u_2|^\alpha\|_{L^q} \\
	& \leq C\|u_1-u_2\|_{L^p}(1+\|u_1\|_{L^{\alpha q}}^\alpha+\|u_2\|_{L^{\alpha q}}^\alpha), 
\end{align*}
where $\frac{1}{p}+\frac1q=\frac12$.
Now, we want to use the fact that $H^1(\Omega)$ is continuously embedded in $L^p(\Omega)$ for any $p\in[1,\infty)$ if $n=2$ and $p\in[1,6]$ if $n=3$.
Consider the case $n=3$ (the case $n=2$ is simpler); by choosing $p=6$ ($q=3$), we obtain
\begin{equation*}
	\|f(u_1)-f(u_2)\|_{L^2}\leq C\|u_1-u_2\|_{H^1}(1+\|u_1\|_{H^1}^\alpha+\|u_2\|_{H^1}^\alpha)\leq C(M)\|u_1-u_2\|_{H^1},
\end{equation*}
if $\alpha\in[0,2]$, where the positive constant $C$ depends on $M:=\max\{\|y_1\|_X, \|y_2\|_X\}$,
and we can conclude that \eqref{eq:Phi-Lip} holds.
Thus, we can apply a classical theory (cfr. \cite[Chapter 4]{Cazenave} or \cite[Chapter 6]{Pazy}) to state that for all $x\in X$ 
there exists a unique mild solution $y\in C([0,T(x)),X)$, and that if $x\in D(A)$, then $y$ is a classical solution. 
Finally, the solution depends continuously on the initial data $x\in X$, uniformly for all $t\in[0,T]$.

The global existence of the solution $y=(u,u_t)$ is guaranteed if we have an {\it a priori} estimate of $\|(u,u_t)\|_X$ on $[0,T(x))$. 
To obtain such estimate, we can use the energy functional defined by \eqref{eq:energy},
that is a nonincreasing function of $t$ along the solutions of \eqref{eq:hypalca-multiD} with boundary conditions \eqref{eq:neumann}. 
This allows us to obtain an estimate (depending only on the initial data $u_0,u_1$) on the $X$--norm of the solutions 
and thus to prove the global existence of the solutions, provided an extra assumptions on the nonlinearity $f$ (see, among others, \cite[Theorem A.7]{FolinoJHDE}).
\subsection{The compactness theorem}
Now, consider the equation \eqref{eq:hypalca-multiD}, and assume that $f=-F'$, where $F\in C^3(\R)$ satisfies
\begin{equation}\label{eq:ass-F}
	F(\pm1)=F'(\pm1)=0, \quad F(s)>0 \; \mbox{ for } s\neq\pm1,
\end{equation}
and there exist positive constants $c_1,C_1$, $K\geq1$ and $\gamma\geq2$ such that
\begin{equation}\label{eq:ass-F-inf}
	c_1|s|^{\gamma/2+1}\leq F(s)\leq C_1|s|^\gamma, \qquad \quad \mbox{ for } |s|\geq K.
\end{equation}
Moreover, $g\in C^1(\R)$ is required to be strictly positive, namely
\begin{equation}\label{eq:ass-g}
	g(s)\geq\kappa>0, \qquad \quad \textrm{ for any }\, s\in\R.
\end{equation}
\begin{rem}
Observe that if $f=-F'$ and the condition \eqref{eq:hyp-f} holds, then there exists $C>0$ such that
\begin{equation*}
	|F(s)|	\leq C |s|^{\alpha+2}, \qquad \quad \mbox{ for } |s|\geq 1,
\end{equation*}
for some $\alpha>0$ in the case $n=2$ and for some $\alpha\in[0,2]$ in the case $n=3$.
Hence, if we want both \eqref{eq:hyp-f} and \eqref{eq:ass-F-inf} to be satisfied, we have to choose $\gamma\leq6$ in the case $n=3$.
As we previously mentioned in Section \ref{sec:intro}, the main example we have in mind is $F(u)=\frac14(u^2-1)^2$, 
and this potential satisfies all the assumptions discussed above.
\end{rem}
The aim of this subsection is to prove a compactness theorem for the solutions to \eqref{eq:hypalca-multiD} as $\e\to0$, 
when the potential $F$ satisfies the assumptions discussed above and $g$ is strictly positive.
To do this, we use the approach introduced by Bronsard and Kohn \cite{Bron-Kohn2} in the case of the classic Allen--Cahn equation \eqref{eq:allencahn-rescaled}.
Regarding the initial data, let us assume that $u_0$, $u_1$ depend on $\e$ and 
\begin{equation}\label{eq:u0-v0}
	\lim_{\e\to0}\|u_0^\e-v_0\|_{L^1(\Omega)}=0,
\end{equation} 
where $v_0$ is a fixed function taking only the values $\pm1$, and that there exists a positive constant $M$ such that 
\begin{equation}\label{eq:energialimitata}
	E_\e[u^\e_0,u^\e_1]\leq M,
\end{equation}
where the energy $E_\e$ is defined in \eqref{eq:energy}.
Since $g$ is strictly positive, from \eqref{variazione-energia} and \eqref{eq:energialimitata} it follows that $(u^\e,u^\e_t)$ satisfies
\begin{align}
	\sup_{t\geq0} E_\e[u^\e,u^\e_t](t)&\leq M,  \label{eq:E<M}\\
	\sup_{t\geq0} \int_\Omega F(u^\e(x,t))\,dx&\leq \e M. \label{eq:int-F}
\end{align}
Moreover, for \eqref{eq:ass-g} we deduce
\begin{equation}\label{eq:ut-stima}
	\e\kappa\int_{t_1}^{t_2}\int_\Omega u^\e_t(x,t)^2\,dxdt\leq E_\e[u,u_t](t_1)-E_\e[u,u_t](t_2)\leq M,
\end{equation}
for any $0\leq t_1<t_2$.
Introducing the function
\begin{equation}\label{eq:Psi}
	\Psi(x):=\int_{-1}^x\sqrt{2F(s)}\,ds,
\end{equation}
we can also prove the following result.
\begin{prop}
Let $(u^\e,u^\e_t)\in C\left([0,\infty),H^2(\Omega)\times H^1(\Omega)\right)$ be the solution to \eqref{eq:hypalca-multiD}, 
where $f=-F'$ with $F$ satisfying \eqref{eq:ass-F} and $g$ satisfying \eqref{eq:ass-g}, 
with either Neumann \eqref{eq:neumann} or Dirichlet \eqref{eq:dirichlet} boundary conditions.
In addition, assume that the initial data \eqref{eq:initialdata} satisfy \eqref{eq:energialimitata}.
Then, 
\begin{equation}\label{eq:grad-Psi}
	\sup_{t\geq0}\int_\Omega |\nabla \Psi(u^\e(x,t))|\,dx\leq M,
\end{equation}
and, for $0\leq t_1<t_2$, 
\begin{equation}\label{eq:dt-Psi}	
	\int_{t_1}^{t_2}\int_\Omega |\partial_t\Psi(u^\e(x,t))|\,dxdt\leq\sqrt\frac{2}\kappa M(t_2-t_1)^{1/2}.
\end{equation}
\end{prop}

\begin{proof}
Let us start with \eqref{eq:grad-Psi}.
Since $\nabla\Psi(u^\e)=\sqrt{2F(u^\e)}\nabla u^\e$ a.e. in $\Omega$, using Young inequality we get
\begin{align*}
	\int_\Omega|\nabla\Psi(u^\e(x,t))|\,dx&=\int_\Omega\sqrt{2F(u^\e(x,t))}|\nabla u^\e(x,t)|\,dx\\
	&\leq\int_\Omega\left[\frac\e2|\nabla u^\e(x,t)|^2+\e^{-1}F(u^\e(x,t))\right]\,dx\leq E_\e[u^\e,u^\e_t](t),
\end{align*}
for any $t\geq0$.
Hence, using \eqref{eq:E<M} we obtain \eqref{eq:grad-Psi}.
The proof of \eqref{eq:dt-Psi} is very similar.
From the Cauchy--Schwarz inequality, it follows that 
\begin{align*}
	\int_{t_1}^{t_2}\int_\Omega |\partial_t&\Psi(u^\e(x,t))|\,dxdt= \int_{t_1}^{t_2}\int_\Omega\sqrt{2F(u^\e(x,t))}|u_t^\e(x,t)|\,dxdt\\
	&\leq \left(\int_{t_1}^{t_2}\int_\Omega2F(u^\e(x,t))dxdt\right)^{1/2}\left(\int_{t_1}^{t_2}\int_\Omega u_t^\e(x,t)^2dxdt\right)^{1/2},	
\end{align*}
for any $0\leq t_1<t_2$.
Using \eqref{eq:int-F} and \eqref{eq:ut-stima}, we obtain
\begin{equation*}
	\int_{t_1}^{t_2}\int_\Omega |\partial_t\Psi(u^\e(x,t))|\,dxdt\leq\sqrt\frac{2}\kappa M(t_2-t_1)^{1/2},
\end{equation*}
and the proof is complete.
\end{proof}
The previous properties of the solution $(u^\e,u^\e_t)$ allow us to prove the following compactness theorem,
that is the main result of this section.
\begin{thm}\label{thm:compactness}
Let $(u^\e,u^\e_t)\in C\left([0,\infty),H^2(\Omega)\times H^1(\Omega)\right)$ be the solution to \eqref{eq:hypalca-multiD}-\eqref{eq:initialdata} 
with either Neumann \eqref{eq:neumann} or Dirichlet \eqref{eq:dirichlet} boundary conditions 
and $f=-F'$, with $F,g$ satisfying \eqref{eq:ass-F}, \eqref{eq:ass-F-inf}, \eqref{eq:ass-g}.
Assume that the initial data $u_0^\e$, $u_1^\e$ satisfy \eqref{eq:u0-v0} and \eqref{eq:energialimitata}.
Then, for any sequence of $\e$'s approaching to zero, there exists a subsequence $\e_j$ such that 
\begin{equation}\label{eq:lim-main}
	\lim_{\e_j\to0} u^{\e_j}(x,t)=v(x,t) \qquad \quad \mbox{ for a.e. }\, (x,t)\in\Omega\times(0,\infty),
\end{equation}
where the function $v$ takes only the values $\pm1$ and satisfies
\begin{align}
	\int_\Omega|v(x,t_2)-v(x,t_1)|\,dx&\leq C|t_2-t_1|^{1/2}, \label{eq:v-t2t1}\\
	\sup_{t\geq0}\|v(\cdot,t)\|_{BV(\Omega)}&\leq C, \label{eq:v-BV}
\end{align}
for some $C>0$, and
\begin{equation}\label{eq:v-v0}
	\lim_{t\to0}\|v(\cdot,t)-v_0\|_{L^1(\Omega)}=0.
\end{equation}
\end{thm}
\begin{proof}
Firstly, let us fix $T>0$ and prove the existence of a subsequence which converges a.e. on $\Omega_T:=\Omega\times(0,T)$.
To this aim, we use that the Banach space $BV(\Omega_T)$ is compactly embedded in $L^1(\Omega_T)$ (among others, see \cite[Theorem 1.19]{Giusti}).
We recall that, given an open set $A\subset\R^n$ and a function $f\in L^1(A)$, 
\begin{align*}
	\int_A|Df|:=\sup\biggl\{\int_A f\mathrm{div}\phi\,dx\, : \, \phi=&(\phi_1,\dots,\phi_n)\in C^1_0(A;\R^n) \\
	&  \qquad\qquad\mbox{ and } \, |\phi(x)|\leq1, \; x\in A\biggr\}.
\end{align*}
The space $BV(A)$ of all the functions $f\in L^1(A)$ such that $\displaystyle\int_A|Df|<\infty$ is a Banach space with the norm
\begin{equation*}
	\|f\|_{BV(A)}:=\|f\|_{L^1(A)}+\int_A|Df|.
\end{equation*}
Now, we have that the functions $\Psi(u^\e)$ are uniformly bounded in $BV(\Omega_T)$.
Indeed, from \eqref{eq:grad-Psi} and \eqref{eq:dt-Psi}, it follows that
\begin{equation}\label{eq:BV-1}
	\int_{\Omega_T}|D\Psi(u^\e)|=\int_0^T\int_\Omega|\nabla \Psi(u^\e(x,t))|\,dxdt+\int_0^T\int_\Omega|\partial_t\Psi(u^\e(x,t))|\,dxdt\leq C, 
\end{equation}
for some constant $C>0$.
Moreover, we claim that
\begin{equation}\label{eq:BV-2}
	\int_0^T\int_\Omega|\Psi(u^\e(x,t))|\,dxdt\leq C, 
\end{equation}
for some constant $C>0$ (independent on $\e$).
In order to prove \eqref{eq:BV-2} let us use the assumption on $F$ \eqref{eq:ass-F-inf}.
If $|u^\e|\leq K$ a.e on $\Omega_T$, then \eqref{eq:BV-2} trivially holds.
Otherwise, we split the integral 
\begin{equation*}
	\int_0^T\int_\Omega|\Psi(u^\e)|\,dxdt=\int_{\{|u^\e|\leq K\}}|\Psi(u^\e)|\,dxdt+\int_{\{|u^\e|\geq K\}}|\Psi(u^\e)|\,dxdt.
\end{equation*}
The first integral is uniformly bounded, whereas for the second one we use \eqref{eq:ass-F-inf} and
\begin{align*}
	|\Psi(u^\e)|&\leq \int_{-K}^K\sqrt{2F(s)}\,ds+\int_{-|u^\e|}^{-K}\sqrt{2F(s)}\,ds+\int_{K}^{|u^\e|}\sqrt{2F(s)}\,ds\\
	&\leq C+2\sqrt{2C_1}\int_K^{|u^\e|}|s|^{\gamma/2}\,ds
	\leq C\left(1+|u^\e|^{\gamma/2+1}\right)\leq C\left(1+F(u^\e)\right).
\end{align*}
Therefore, 
\begin{equation*}
	\int_{\{|u^\e|\geq K\}}|\Psi(u^\e(x,t))|\,dxdt\leq C+C\int_0^T\int_\Omega F(u^\e(x,t))\,dxdt,
\end{equation*}
and using \eqref{eq:int-F} we obtain the claim \eqref{eq:BV-2}.

Thanks to \eqref{eq:BV-1}-\eqref{eq:BV-2} and a standard compactness result (among others, see \cite[Theorem 1.19]{Giusti}), we can state that 
there exists a subsequence $\Psi(u^{\e_j})$ which converges in $L^1(\Omega_T)$ to a function $\Psi^*$,
namely
\begin{equation}\label{eq:Psi*}
	\lim_{\e_j\to0}\|\Psi(u^{\e_j})-\Psi^*\|_{L^1(\Omega_T)}=0.
\end{equation}
Passing to a further subsequence if necessary, we obtain
\begin{equation*}
	\lim_{\e_j\to0}\Psi(u^{\e_j}(x,t))=\Psi^*(x,t), \qquad \qquad  \mbox{ a.e. on }\, \Omega\times(0,T).
\end{equation*}
Since $\Psi'=\sqrt{2F}$ is strictly positive except at $\pm1$, the function $\Psi$ is monotone and there is
a unique function $v$ such that $\Psi(v(x,t))=\Psi^*(x,t)$, and so 
\begin{equation*}
	\lim_{\e_j\to0} u^{\e_j}(x,t)=v(x,t) \qquad \quad \mbox{ a.e. on }\, \Omega\times(0,T).
\end{equation*}
Using the Fatou's Lemma and \eqref{eq:int-F}, we get
\begin{equation*}
	\int_0^T\int_\Omega F(v(x,t))\,dxdt\leq\liminf_{\e_j\to0}\int_0^T\int_\Omega F(u^{\e_j}(x,t))\,dxdt=0,
\end{equation*}
and so, $v$ takes only the values $\pm1$.
Now, let us prove \eqref{eq:v-t2t1}.
For any fixed $x\in\Omega$ one has
\begin{equation*}
	\left|\Psi(u^\e(x,t_2))-\Psi(u^\e(x,t_1))\right|\leq\int_{t_1}^{t_2}\left|\partial_t\Psi(u^\e(x,t))\right|\,dt,
\end{equation*}
for any $0\leq t_1<t_2$.
Integrating and using \eqref{eq:dt-Psi} we end up with
\begin{equation}\label{eq:Psi-t1t2}
	\int_\Omega\left|\Psi(u^\e(x,t_2))-\Psi(u^\e(x,t_1))\right|\,dx\leq\sqrt\frac{2}\kappa M(t_2-t_1)^{1/2}.
\end{equation}
Since $\Psi(u^{\e_j}(\cdot,t))\to\Psi^*(\cdot,t)$ in $L^1(\Omega)$ for almost every $t\in(0,T)$ by \eqref{eq:Psi*} and
\begin{equation}\label{eq:Psi-u0v0}
	\int_\Omega\left|\Psi(u^\e_0(x))-\Psi(v_0(x))\right|\,dx=0,
\end{equation}
because of \eqref{eq:u0-v0} and \eqref{eq:grad-Psi}, passing to the limit as $\e_j\to0$ in \eqref{eq:Psi-t1t2} we conclude that
\begin{equation}\label{eq:Psi*-t1t2}
	\int_\Omega\left|\Psi^*(x,t_2)-\Psi^*(x,t_1)\right|\,dx\leq\sqrt\frac{2}\kappa M(t_2-t_1)^{1/2},
\end{equation}
for almost every $0\leq t_1<t_2<T$.
However, $\Psi^*(x,t)=\Psi(v(x,t))$ with $v$ taking only the values $\pm1$ and as a consequence
\begin{align}
	\left|\Psi^*(x,t_2)-\Psi^*(x,t_1)\right|&=\left|\Psi(v(x,t_2))-\Psi(v(x,t_1))\right|\notag\\
	&=\frac{\Psi(1)}2|v(x,t_2)-v(x,t_1)|, \label{eq:Psi*-c0}
\end{align}
where we used that $\Psi(-1)=0$. 
Therefore, substituting \eqref{eq:Psi*-c0} in \eqref{eq:Psi*-t1t2}, we obtain \eqref{eq:v-t2t1} for almost every $t_1, t_2\in(0,T)$.
It is possible to redefine $v$ at the exceptional times to make it continuous as a map from $[0,T]$ to $L^1(\Omega)$,
and then \eqref{eq:v-t2t1} holds for every $t_1, t_2\in(0,T)$.

By reasoning in the same way, we obtain \eqref{eq:v-v0}.
Taking $t_1=0$ in \eqref{eq:Psi-t1t2} and passing to the limit as $\e_j\to0$ making use of \eqref{eq:Psi-u0v0}, we deduce
\begin{equation*}
	\int_\Omega\left|\Psi(v(x,t_2))-\Psi(v_0(x))\right|\,dx\leq\sqrt\frac{2}\kappa M(t_2)^{1/2},
\end{equation*}
and using \eqref{eq:Psi*-c0}, we get \eqref{eq:v-v0}.
In conclusion, we proved the properties \eqref{eq:lim-main}-\eqref{eq:v-BV} on arbitrary finite time intervals $(0,T)$.
It is possible to extend the results on the infinite interval $(0,\infty)$ by taking a sequence of times $T_j\to\infty$ and 
a diagonal subsequence of $\{u^\e\}$ in the usual manner.
\end{proof}
\begin{rem}
Consider the slower time scale of order $\e^{-2}$, i.e. the new variable $s=\e^{-2}t$.
Then the function $\tilde{u}^\e(x,s)=u^\e(x,\e^{2}s)$ satisfies the equation
\begin{equation}\label{eq:rescale}
	\tau \tilde u^\e_{ss}+g(\tilde u^\e)\tilde u^\e_s=\e^2\Delta\tilde u^\e+f(\tilde u^\e), \qquad \quad x \in\Omega, t>0.
\end{equation}
Using \eqref{eq:Psi-t1t2} with $t_2=\e^{2}s$ and $t_1=0$, we obtain
\begin{equation*}
	\int_\Omega\left|\Psi(\tilde u^\e(x,s))-\Psi(\tilde u^\e(x,0))\right|\,dx\leq\sqrt\frac{2}\kappa M\e s^{1/2},
\end{equation*}
for any $s>0$.
This shows that the evolution of the solutions to \eqref{eq:rescale} is very slow (for $\e$ small) until $s\sim\e^{-2}$.
\end{rem}
\begin{rem}\label{rem:ass-F-inf}
In all this section, we used the assumption \eqref{eq:ass-F-inf} only to prove \eqref{eq:BV-2}.
Indeed, assumption \eqref{eq:ass-F-inf} implies the uniformly boundedness of the term if the initial data satisfy \eqref{eq:energialimitata}.
Observe that, if we assume that the solution $u^\e$ is uniformly (with respect to $\e$) bounded for any $t$,
then \eqref{eq:BV-2} trivially holds and  we can remove the assumption \eqref{eq:ass-F-inf} from Theorem \ref{thm:compactness}.
\end{rem}

\subsection{Higher order estimates}\label{sec:higher}
By using the energy functional \eqref{eq:energy}, it is possible to obtain a control for the $H^1\times L^2$--norm of the solutions $(u^\e,u^\e_t)$.
The goal of this subsection is to obtain higher order estimates, in particular to control the behavior of the $H^1$--norm of $u^\e_t$ as $\e\to 0$.
\begin{rem}\label{rem:regularity}
Since $\Omega$ is a bounded domain of $\R^n$, $n=2$ or $3$, with a $C^1$ boundary, thanks to the general Sobolev inequalities, 
we can say that $H^2(\Omega)$ is continuously embedded in $C^{0,\gamma}(\Omega)$, 
with $\gamma$ any positive number strictly less than $1$ if $n=2$ and $\gamma=1/2$ if $n=3$.
Furthermore, $H^1(\Omega)$ is continuously embedded in $L^p(\Omega)$ for any $p\in[1,\infty)$ if $n=2$ and $p\in[1,6]$ if $n=3$.
Therefore, we can say that $(u^\e,u^\e_t)\in C\left([0,\infty),C^{0,\gamma}(\Omega)\times L^p(\Omega)\right)$ and the functions
\begin{equation*}
	s_1(t):=\sup_{x\in\Omega}|u^\e(x,t)|, \quad \quad s_2(t):= \|u^\e_t(\cdot,t)\|_{L^p(\Omega)}
\end{equation*}
are continuous function on $[0,\infty)$. 
In the following we assume that the function $s_1$ defined above is uniformly bounded in $\e$. 
\end{rem}
Consider the case $g=1$, that is the case we will study in the next section, where we will use the following result.
\begin{prop}\label{prop:higher}
Let 
$$(u^\e,u^\e_t)\in C\left([0,T],H^2(\Omega)\times H^1(\Omega)\right)\cap C^1\left([0,T],H^1(\Omega)\times L^2(\Omega)\right)$$
be the solution to \eqref{eq:hypalca-multiD}, where $f=-F'$ with $F$ satisfying \eqref{eq:ass-F} and $g\equiv1$, 
with either Neumann \eqref{eq:neumann} or Dirichlet \eqref{eq:dirichlet} boundary conditions.
Regarding the initial data \eqref{eq:initialdata}, we assume that they satisfy \eqref{eq:energialimitata}, that $u^\e$ is uniformly bounded, namely
\begin{equation*}
	\sup_{x\in \Omega}|u^\e(x,t)|\leq C, \qquad \quad \forall\, t\in[0,T],
\end{equation*}
and that there exists a positive constant $C$ (independent on $\e$ and $\tau$) such that 
\begin{equation}\label{eq:R[u0,u1]}
	R[u^\e_0,u^\e_1]:=\e^{-2}\tau^{-1}\int_\Omega \left(\Delta u^\e_0-\e^{-2}F'(u^\e_0)-u^\e_1\right)^2\,dx+\int_\Omega|\nabla u^\e_1|^2\,dx\leq C\e^{-5}\tau^{-1}.
\end{equation}
Then, there exists $C>0$ (independent on $\e$ and $\tau$) such that
\begin{equation}\label{eq:grad-ut}
	\int_0^T\|  u^\e_t(\cdot,t)\|_{H^1(\Omega)}^2\leq C\e^{-5}\left(1+\tau^{-1}\right).
\end{equation}
\end{prop}
\begin{proof}
Denote by $w^\e=u^\e_t$.
From the assumptions on the regularity of the solution, $w^\e\in C([0,T], H^1(\Omega))\cap C^1([0,T], L^2(\Omega))$ 
and by differentiating the equation \eqref{eq:hypalca-multiD} with respect to $t$, we end up with
\begin{equation*}
	\e^2\tau w^\e_{tt}+w^\e_t=\Delta w^\e-\e^{-2}F''(u^\e)w^\e, \qquad \quad x \in\Omega,\, t>0.
\end{equation*}
The initial data for $w^\e$ are
\begin{equation*}
	w^\e(\cdot,0)=u_1, \qquad \e^{2}\tau w^\e_t(\cdot,0)=\Delta u^\e_0-\e^{-2}F'(u^\e_0)-u^\e_1, \qquad \quad x\in\Omega,
\end{equation*}
and the boundary conditions are 
\begin{equation*}
	\frac{\partial w^\e}{\partial n}(x,t)=0, \qquad \, x\in\partial\Omega, \, t>0,
\end{equation*}
in the case of homogeneous Neumann boundary conditions \eqref{eq:neumann} and
\begin{equation*}
	w^\e(x,t)=0, \qquad \, x\in\partial\Omega, \, t>0.
\end{equation*}
in the case of Dirichlet boundary conditions \eqref{eq:dirichlet}.
Multiplying the equation by $w^\e_t$ and integrating in $\Omega$, we obtain
\begin{align*}
	\frac{d}{dt}\int_\Omega\frac{\e^2\tau}{2}\left(w^\e_t\right)^2\,dx+\int_\Omega\left(w^\e_t\right)^2\,dx= & \int_\Omega \mathrm{div}(\nabla w^\e w^\e_t)\,dx\\
	&-\int_\Omega\nabla w^\e\nabla w^\e_t\,dx-\e^{-2}\int_\Omega F''(u^\e)w^\e w^\e_t\,dx.
\end{align*}
Using the divergence theorem and the boundary conditions, we deduce 
\begin{equation*}
	\frac12\frac{d}{dt}\left[\int_\Omega \e^2\tau \left(w^\e_t\right)^2\,dx+\int_\Omega|\nabla w^\e|^2\,dx\right]+\int_\Omega\left(w^\e_t\right)^2\,dx
	=-\e^{-2}\int_\Omega F''(u^\e)w^\e w^\e_t\,dx.
\end{equation*}
Since $|F''(u^\e)|\leq C$ for $(x,t)\in\Omega\times(0,T)$ (with $C$ independent on $\e$ for the assumption on the boundedness of the solution and the regularity of $F$), we infer
\begin{equation*}
	\frac12\frac{d}{dt}\left[\int_\Omega \e^2\tau \left(w^\e_t\right)^2\,dx+\int_\Omega|\nabla w^\e|^2\,dx\right]
	+\frac12\int_\Omega \left(w^\e_t\right)^2\,dx\leq \frac{C^2}2\e^{-4}\int_\Omega \left(w^\e\right)^2\,dx.
\end{equation*}
By integrating on $(0,T)$, we end up with
\begin{align*}
	\|w^\e_t(\cdot,t)\|^2_{L^2(\Omega)}+\|\nabla w^\e(\cdot,t)\|^2_{L^2(\Omega)}+\int_0^T\int_\Omega \left(w^\e_t\right)^2\,dxdt&\leq 
	C^2\e^{-4}\int_0^T\int_\Omega \left(w^\e\right)^2\,dxdt\\
	&\quad+R[u^\e_0,u^\e_1],
\end{align*}
for any $t\in(0,T)$.
In particular, we proved that
\begin{equation*}
	\|\nabla u^\e_t(\cdot,t)\|_{L^2(\Omega)}^2\leq C\e^{-4}\int_0^T\int_\Omega \left(u^\e_t\right)^2\,dxdt+R[u^\e_0,u^\e_1],
\end{equation*}
for any $t\in(0,T)$.
Recalling that from the assumptions \eqref{eq:energialimitata} (see \eqref{eq:ut-stima}) it follows that
\begin{equation*}
	\int_0^T\int_\Omega\left(u^\e_t\right)^2\,dxdt\leq C\e^{-1},
\end{equation*}
and using the assumption on $R[u^\e_0,u^\e_1]\leq C\e^{-5}\tau^{-1}$, we obtain \eqref{eq:grad-ut}.
\end{proof}

\subsection{Formal derivation of the interface motion equation}\label{sec:formal}
Theorem \ref{thm:compactness} asserts that some solutions $u^\e$ to the IBVP for the nonlinear damped hyperbolic Allen--Cahn equation \eqref{eq:hypalca-multiD}
take only the values $\pm1$ as $\e\to0$.
As we already mentioned, the main aim of the paper is to study the motion of the interface where
the solution $u^\e$ makes its transitions from $-1$ to $+1$.
The interface motion equation can be formally derived by means of asymptotic expansions and coincides with the mean curvature flow equation
(see \cite{Alf-Hil-Mat} or \cite{Hil-Nara}). 
In this subsection, we present this formal computation in the case of the nonlinear damped hyperbolic Allen--Cahn equation \eqref{eq:hypalca-multiD},
showing that the motion is governed by mean curvature flow for general damping coefficients $g$.
We shall assume that the steep interfaces are already developed. 

Let $u^\varepsilon$ be a solution to \eqref{eq:hypalca-multiD} where $f=-F'$, with $F$ satisfying \eqref{eq:ass-F}. 
Define
\begin{equation*}
	\Gamma^\varepsilon(t):=\left\{\bm x\in\Omega : u^\varepsilon(\bm x,t)=0\right\}, \qquad \Omega^\varepsilon_{\pm}(t):=\left\{\bm x\in\Omega: \pm u^\varepsilon(\bm x,t)>0\right\},
\end{equation*}
and the signed distance function
\begin{equation*}
	d^\varepsilon(\bm x,t):=
	\begin{cases}
		\textrm{dist}(\bm x,\Gamma^\varepsilon(t)), \qquad &\bm x\in\Omega^\varepsilon_+(t), \\
		0, \qquad &\bm x\in\Gamma^\varepsilon(t),\\
		-\textrm{dist}(\bm x,\Gamma^\varepsilon(t)), \qquad &\bm x\in\Omega^\varepsilon_-(t).
	\end{cases}
\end{equation*}
We assume that the function $d^\varepsilon$ has the following expansion
\begin{equation*}
	d^\varepsilon(\bm x,t)=\sum_{k=0}^\infty\varepsilon^kd_k(\bm x,t)=d_0(\bm x,t)+\varepsilon d_1(\bm x,t)+\varepsilon^2d_2(\bm x,t)+\cdots.
\end{equation*}
Observe that $|\nabla d^\varepsilon|=1$ in a neighborhood of $\Gamma^\varepsilon(t)$.
Here and in what follows $|\cdot|$ and $\cdot$ are the standard norm and inner product in $\R^n$. 
Then, by considering the terms of order $O(1)$ and the ones of $O(\varepsilon)$ in $|\nabla d^\varepsilon|^2=1$, we obtain
\begin{equation}\label{d0d1}
	|\nabla d_0|^2=1, \qquad \quad \nabla d_0\cdot \nabla d_1=0. 
\end{equation}
Formally, we study the motion of the interface in the limit $\varepsilon\to0$.
To this end, let us define 
\begin{equation*}
	\Gamma^0(t):=\left\{\bm x\in\Omega : d_0(\bm x,t)=0\right\}, \qquad \Omega^0_{\pm}(t):=\left\{\bm x\in\Omega : \pm d_0(\bm x,t)>0\right\}.
\end{equation*}
We want to show (formally) that the motion of $\Gamma^0$ is governed by mean curvature flow. 
Hence, let us formally derive the equation for the function $d_0$ describing the motion of $\Gamma^0(t)$. 
Following \cite{Alf-Hil-Mat, Hil-Nara}, consider the following expansion for the solution of \eqref{eq:hypalca-multiD}
\begin{equation*}
	u^\varepsilon(\bm x,t)=\sum_{k=0}^\infty\varepsilon^kU_k(\bm x,t,z)=U_0(\bm x,t,z)+\varepsilon U_1(\bm x,t,z)+\varepsilon^2U_2(\bm x,t,z)+\cdots
\end{equation*}
near the interface $\Gamma^\varepsilon(t)$, where $z:=d^\varepsilon(\bm x,t)/\varepsilon$.
Since we are looking for an approximate solution $u^\varepsilon$ such that $u^\varepsilon\approx\pm1$ on $\Omega^\varepsilon_\pm(t)$, we assume 
\begin{equation*}
	u^\varepsilon(\bm x,t)=\pm1+\varepsilon\phi^\pm_1(\bm x,t)+\varepsilon^2\phi^\pm_2(\bm x,t)+\cdots, \qquad \mbox{ on } \Omega^\varepsilon_\pm(t).
\end{equation*}
To make the expansions near and away the interface consistent, we require the following matching conditions
\begin{equation*}
	U_0(\bm x,t,\pm\infty)=\pm1, \qquad \quad U_k(\bm x,t,\pm\infty)=\phi^\pm_k(\bm x,t,\pm\infty), \quad k\geq1.
\end{equation*}
We normalize $U_0$ in such a way that $U_0(\bm x,t,0)=0$.
By direct computations, near the interface $\Gamma^\varepsilon(t)$ we have
\begin{align*}
	u^\varepsilon_t & = U_{0,t}+U_{0,z} \frac{d_{0,t}}{\varepsilon}+U_{0,z}d_{1,t}+\varepsilon U_{0,z}d_{2,t}+\varepsilon U_{1,t}+U_{1,t}d_{0,t}+\varepsilon U_{1,t}d_{1,t}+\cdots,\\
	u^\varepsilon_{tt} & =U_{0,tt}+U_{0,tz}\frac{d_{0,t}}{\varepsilon}+U_{0,zt} \frac{d_{0,t}}{\varepsilon}+U_{0,zz} \frac{d_{0,t}^2}{\varepsilon^2}+U_{0,z} \frac{d_{0,tt}}{\varepsilon}+\cdots,\\
	\Delta u^\varepsilon & = \Delta U_0+\frac2\varepsilon\nabla d_0\cdot\nabla U_{0,z}+U_{0,z}\frac{\Delta d_0}\varepsilon+U_{0,zz}\frac{|\nabla d_0|^2}{\varepsilon^2}+\frac2\varepsilon U_{0,zz}\nabla d_0\cdot\nabla d_1+\\
					& \qquad	+\varepsilon \Delta U_1+2\nabla d_0\cdot\nabla U_{1,z}+U_{1,z}\Delta d_0+U_{1,zz}\frac{|\nabla d_0|^2}{\varepsilon}+2 U_{0,zz}\nabla d_0\cdot\nabla d_1+\cdots,\\
	f(u^\varepsilon) & =f(U_0)+\varepsilon f'(U_0)U_1+O(\varepsilon^2),\\
	g(u^\varepsilon) & =g(U_0)+\varepsilon g'(U_0)U_1+O(\varepsilon^2).
\end{align*}
We substitute these expansions in \eqref{eq:hypalca-multiD} and collect the $\varepsilon^{-2}$ and $\varepsilon^{-1}$ terms. 
Since we have the terms $\varepsilon^2\tau u^\varepsilon_{tt}$ and $\varepsilon^{-2}f(u^\varepsilon)$, the only terms with $\varepsilon^{-2}$ are 
$f(U_0)$ and $U_{0,zz}|\nabla d_0|^2$.
Then, from \eqref{d0d1} it follows that $U_{0,zz}+f(U_0)=0$.
Combining this equation with the matching and normalization conditions, we obtain that $U_0$ is the unique solution to the problem
\begin{equation}\label{eq:U0}
	U_{0,zz}+f(U_0)=0, \qquad U_0(\bm x,t,0)=0, \qquad U_0(\bm x,t,\pm\infty)=\pm1.
\end{equation}
Therefore, $U_0(\bm x,t,z)=\Phi(z)$ where $\Phi$ is the standing wave profile.
For example, in the case $f(u)=u(1-u^2)$ we have $U_0(z)=\tanh(z/\sqrt2)$.
The first approximation of the profile of a transition layer around the interface is the solution $U_0$.
Note that the first approximation is the same of the parabolic case and does not depend on the damping coefficient $g$.

Next, by collecting the $\varepsilon^{-1}$ terms, we deduce
\begin{equation*}
	g(U_0) U_{0,z}\,d_{0,t}=2\nabla d_0\cdot\nabla U_{0,z}+U_{0,z}\Delta d_0+2U_{0,zz}\nabla d_0\cdot\nabla d_1+U_{1,zz}|\nabla d_0|^2+f'(U_0)U_1.
\end{equation*}
Using \eqref{d0d1} and $\nabla U_{0,z}=0$, we get
\begin{equation}\label{eq:U1}
	U_{1,zz}+f'(U_0)U_1=U_{0,z}\left\{g(U_0)d_{0,t}-\Delta d_0\right\}.
\end{equation}
The solvability condition for the linear equation of $U_1$ \eqref{eq:U1} plays the key role in determining the equation of interface motion.
In order to obtain the solvability condition for \eqref{eq:U1}, we use the following lemma (see \cite[Lemma 2.2]{Alf-Hil-Mat}). 
\begin{lem}[\cite{Alf-Hil-Mat}]
Let $A(z)$ be a bounded function on $\R$.
Then the problem 
\begin{equation*}
	\begin{cases}
		\psi_{zz}+f'(U_0)\psi=A, \qquad \quad z\in\R,\\
		\psi(0)=0, \qquad \psi\in L^\infty(\R),
	\end{cases}
\end{equation*}
has a solution if and only if 
\begin{equation}\label{solvability}
	\int_\R A(z) U'_0(z)\,dz=0.
\end{equation}
Moreover, if the solution exists, it is unique and satisfies
\begin{equation*}
	|\psi(z)|\leq C\|A\|_\infty \qquad \quad \textrm{ for } \, z\in\R,
\end{equation*}
for some constant $C>0$.
\end{lem}
For the proof of this lemma see \cite{Alf-Hil-Mat}.
By applying the solvability condition \eqref{solvability} in equation \eqref{eq:U1}, we have
\begin{equation}\label{solvability2}
	\int_\R U'_0(z)\left\{g(U_0(z))d_{0,t}(\bm x,t)-\Delta d_0(\bm x,t)\right\} U'_0(z)\,dz=0.
\end{equation}
Since $U_0$ solves the problem \eqref{eq:U0}, it follows that $U'_0(z)=\sqrt{2F(z)}$.
Substituting this equality in \eqref{solvability2} and using the change of variable $U_0(z)=s$, we obtain
\begin{equation*}
	\sqrt2\int_{-1}^{+1}\sqrt{F(s)}\left\{g(s)d_{0,t}(\bm x,t)-\Delta d_0(\bm x,t)\right\}ds=0,
\end{equation*}
and as a consequence
\begin{equation*}
	\left(\int_{-1}^{+1}\sqrt{F(s)}\,g(s)\,ds\right)d_{0,t}(\bm x,t)=\left(\int_{-1}^{+1}\sqrt{F(s)}\,ds\right)\Delta d_0({\bm x,t}). 
\end{equation*}
Introducing the (weighted) average $\overline{g}$ of the continuous function $g$:
\begin{equation*}
	\overline{g}:=\frac{1}{\|\sqrt{F}\|_{{}_{L^1}}}\int_{-1}^{1} \sqrt{F(s)}\,g(s)\,ds,
\end{equation*}
we conclude that the function $d_0$ satisfies the heat equation
\begin{equation}\label{heat-d0}
	\overline{g}\,d_{0,t}=\Delta d_0.
\end{equation}
This generalizes the formal computation of \cite{Hil-Nara}, where the case $g\equiv\gamma\in\R$ is considered. 
In the latter case, equation \eqref{heat-d0} becomes $\gamma\,d_{0,t}=\Delta d_0$.
We can do the same remarks of \cite{Hil-Nara}: 
since $|\nabla d_0|=1$ near $\Gamma^0(t)$, the terms $-d_{0,t}$ and $\Delta d_0$ are equivalent to the outward normal velocity and the mean curvature of $\Gamma^0(t)$, respectively.
Hence, \eqref{heat-d0} means that the motion of $\Gamma^0(t)$ is governed by mean curvature.
In conclusion, we have formally shown that in the limit $\varepsilon\to0^+$ the interface $\Gamma^\varepsilon(t)$ moves by mean curvature.

\section{Motion of the interface in the radial case}\label{sec:radial}
In this section we study the evolution of radial solutions to \eqref{eq:hypalca-multiD} with damping coefficient $g\equiv1$
and with boundary conditions of Dirichlet type.
In particular, the aim of this section is to state and prove the main result of the paper, Theorem \ref{cor:main}.
To do this, consider the damped wave equation with bistable nonlinearity
\begin{equation}\label{eq:damped-multiD}
	\e^2\tau u^\e_{tt}+u^\e_t=\Delta u^\e-\e^{-2}F'(u^\e), \qquad \quad x\in B(0,1), \, \,t>0, 
\end{equation}
where $B(0,1)=\left\{x\in \R^n : |x|\leq1\right\}$, $n=2$ or $n=3$. 
We assume that the parameter $\tau$ depends on $\e$ and that there exists a positive number $\mu\ll1$ such that
\begin{equation}\label{eq:tau-ugly}
	\tau(\e)=o(\e^\mu).
\end{equation}
This assumption is instrumental in the proof of our main result,
nevertheless, for the numerical solutions in Figures \ref{fig:eps=0.02}-\ref{fig:eps=0.01_t=250} the result is valid without restrictions on $\tau >0$; 
for further discussions, see Remark \ref{rem:ass-tau}.

The function $F$ is required to be a double well potential with wells of equal depth;
precisely, we assume that $F\in C^2(\R)$ satisfies \eqref{eq:ass-F} plus a nondegenerate condition on $F''(\pm1)$, namely
\begin{equation}\label{eq:ass-F-radial}
	F(\pm1)=F'(\pm1)=0, \quad F''(\pm1)>0, \quad F(s)>0 \; \mbox{ for } s\neq\pm1.
\end{equation}
We restrict our attention on radially symmetric solutions and so on the equation
\begin{equation}\label{eq:damped-radial}
	\e^2\tau u^\e_{tt}+u^\e_t=u^\e_{rr}+\frac{n-1}{r}u^\e_r-\e^{-2}F'(u^\e), \qquad \quad r\in(0,1), \,\,t>0, 
\end{equation}
which is equation \eqref{eq:damped-multiD} in radial coordinates. 
We consider the case of Dirichlet boundary condition
\begin{equation}\label{eq:boundary}
	u^\e(1,t)=1, \qquad \quad \forall\,t\geq0;
\end{equation}
moreover, at $r=0$ $u$ must satisfy $u^\e_r(0,t)=0$ for any $t\geq0$.
We consider the boundary value problem \eqref{eq:damped-radial}, \eqref{eq:boundary} subject to initial data
\begin{equation}\label{eq:initialdata-radial}
	u^\e(r,0)=u^\e_0(r), \qquad \quad u^\e_t(r,0)=u_1^\e(r), \qquad r\in(0,1),
\end{equation}
Fix $\rho_0\in(0,1)$, and assume that $u_0^\e$ has a \emph{1-transition layer structure} with transition from $-1$ to $+1$ in $r=\rho_0$.
Precisely, we assume that $u_0^\e$ converges in $L^1$ as $\e\to0^+$ to the function
\begin{equation*}
	\bar u(r):=\begin{cases} -1, \qquad r<\rho_0, \\ +1, \qquad r>\rho_0,\end{cases}
\end{equation*}
that is
\begin{equation}\label{eq:u0-ass}
	\lim_{\e\to0}\int_0^1\left|u_0^\e(r)-\bar u(r)\right|r^{n-1}\,dr=0,
\end{equation}
and that $u_0^\e$ makes the transition in a way such that
\begin{equation}\label{eq:energy-u}
	\int_{0}^{1}\left[\frac{\e^3\tau}{2}(u_1^\e)^2+\frac{\e}{2}(u^\e_0)_r^2+\e^{-1}F(u^\e_0)\right]\theta(r)\,dr\leq c_0+z(\e),
\end{equation}
where $z:\mathbb{R}^+\rightarrow\mathbb{R}^+$ is a positive function with $z=o(1)$ as $\e\to0^+$ and
\begin{equation}\label{eq:c0}
	c_0:=\int_{-1}^{1}\sqrt{2F(s)}\,ds, \qquad \quad\theta(r):=\exp\left\{-(n-1)\left(\frac{r}{\rho_0}-1\right)\right\}\left(\frac{r}{\rho_0}\right)^{n-1}.
\end{equation}
Observe that \eqref{eq:energy-u}-\eqref{eq:c0} imply that the energy \eqref{eq:energy} remains bounded, namely
\begin{equation*}
	E_\e[u_0^\e,u_1^\e]:=\int_0^1\left[\frac{\e^3\tau}{2} \left(u^\e_1\right)^2+\frac{\varepsilon}2 \left(u_0^\e\right)_r^2+\e^{-1}F(u^\e_0)\right]r^{n-1}dr\leq M,
\end{equation*}
and so, the condition \eqref{eq:energialimitata} is satisfied.

Moreover, in order to apply Proposition \ref{prop:higher}, we assume that the initial data are such that \eqref{eq:R[u0,u1]} holds.
Finally, as in Section \ref{sec:higher}, in the following we shall consider uniformly bounded in $\e$ solutions and so we shall assume that
there exists $C>0$ (independently on $\e$) such that
\begin{equation}\label{eq:u-bounded}
	\sup_{r\in (0,1)}|u^\e(r,t)|\leq C, \qquad \quad \forall\, t\geq0.
\end{equation}
\begin{rem}\label{rem:initial}
Let us briefly show how to construct functions $u^\e_0,u^\e_1$ satisfying the assumptions \eqref{eq:R[u0,u1]}, \eqref{eq:u0-ass} and \eqref{eq:energy-u}.
The requirement \eqref{eq:energy-u} trivially holds if $\e^3\tau\|u_1^\e\|^2_{L^2}\to0$ as $\e\to0^+$ 
and the initial datum $u_0^\e$ satisfies \eqref{eq:energy-u} with $\tau=0$, that is the assumption on the initial datum in \cite{Bron-Kohn2}.
An example of function $u_0^\e$ satisfying the assumptions of \cite{Bron-Kohn2} can be constructed as in \cite{Sternberg}.
It is easy to check that, if $u_0^\e$ is constructed as in \cite{Sternberg} and $u_1^\e$ is sufficiently small as $\e\to0^+$,
then $(u^\e_0,u^\e_1)$ satisfy the assumptions \eqref{eq:R[u0,u1]}, \eqref{eq:u0-ass} and \eqref{eq:energy-u}.
\end{rem}

Now, we can state the main results of this paper. 
\begin{thm}\label{thm:main}
Fix $\rho_0\in(0,1)$.
Let $\tau$ be as in \eqref{eq:tau-ugly}, $F$ satisfying \eqref{eq:ass-F-radial} and let $u^\e$ be the solution to \eqref{eq:damped-radial} 
with Dirichlet boundary condition \eqref{eq:boundary} and initial data \eqref{eq:initialdata-radial}.
Assume that $u_0^\e$, $u^\e_1$ satisfy  \eqref{eq:R[u0,u1]}, \eqref{eq:u0-ass}, \eqref{eq:energy-u}
and that \eqref{eq:u-bounded} holds.
Then, for any $T\in(0,T_{\max})$ 
\begin{equation}\label{eq:main}
	\lim_{\e\to0}\int_0^T\int_0^1\left|u^\e(r,t)-\omega^\e(r,t)\right|r^{n-1}\,dr\,dt=0,
\end{equation}
where $T_{\max}:=\rho_0^2/2(n-1)$, and 
\begin{equation*}
	\omega^\e(r,t)=\left\{\begin{array}{ll} -1, \qquad r<\rho^\e(t),\\
	+1, \qquad r>\rho^\e(t),
	\end{array}\right.
\end{equation*}
with $\rho^\e=\rho^\e(t)$ satisfying
\begin{equation}\label{eq:rho}
	\e^2\tau (\rho^\e)''+(\rho^\e)'+\frac{n-1}{\rho^\e}=0, \qquad \quad \rho^\e(0)=\rho_0\in(0,1), \quad (\rho^\e)'(0)=\nu_0\in\left[-\frac{n-1}{\rho_0},0\right].	
\end{equation}
\end{thm}

\begin{thm}\label{cor:main}
Under the same assumptions of Theorem \ref{thm:main}, we have
\begin{equation}\label{eq:corollary}
	\lim_{\e\to0}\int_0^T\int_0^1\left|u^\e(r,t)-\omega^0(r,t)\right|r^{n-1}\,dr\,dt=0,
\end{equation}
where
\begin{equation*}
	\omega^0(r,t)=\left\{\begin{array}{ll} -1, \qquad r<\rho^o(t),\\
	+1, \qquad r>\rho^o(t),
	\end{array}\right.
\end{equation*}
with $\rho^o(t)=\sqrt{\rho_0^2-2(n-1)t}$.
\end{thm}
Theorem \ref{cor:main} 
shows that the formal computation given in Section \ref{sec:formal} 
is asymptotically correct in the radial case, for certain boundary conditions and initial data.
Indeed, as $\e$ goes to $0$, the motion of the ``transition sphere'' is governed by the mean curvature equation.
However, in order to prove the result in the hyperbolic setting, 
 we need to use the equation \eqref{eq:rho} which takes into account also the inertial term $\e^2\tau\rho''$ as shown in Theorem \ref{thm:main}.

The rest of the paper is devoted to prove the previous theorems.
To do this, we need some preliminary results.

\subsection{Study of the ODE}\label{subsec:rho}
First of all, let us study the behavior of the solutions to \eqref{eq:rho}.
From now on, to simplify notation we write $\rho$ instead of $\rho^\e$.
Formally, for $\e\tau=0$ we obtain 
\begin{equation*}
	(\rho^o)'+\frac{n-1}{\rho^o}=0, \qquad \quad \rho^o(0)=\rho_0.
\end{equation*}
and then we have $\rho^o(t)=\sqrt{\rho_0^2-2(n-1)t}$, which is defined for $t\in[0,T_{\max}]$, 
where 
\begin{equation}\label{eq:Tmax}
	T_{\max}:=\frac{\rho^2_0}{2(n-1)}.
\end{equation}
In particular, we can say that there exists a finite time $T_{\max}$ such that $\rho^o(\T)=0$ and 
\begin{equation*}
	\lim_{t\rightarrow\T}\rho^o(t)'=-\infty.
\end{equation*}
The following result collects some properties of the solutions to \eqref{eq:rho} that we will use later.
\begin{lem}\label{lem:prop-rho}
Let $(\rho,\rho')$ the solution to \eqref{eq:rho} and let $\T$ be the constant defined in \eqref{eq:Tmax}.
Then, there exists $T^\e_m\in[\T,\rho_0/|\nu_0|]$ such that $\rho(T^\e_m)=0$.
Moreover, we have
\begin{equation}\label{eq:inv-reg}
	\rho'(t)\leq0, \qquad \rho(t)\rho'(t)+n-1\geq0, \qquad \quad \forall\, t\in[0,T^\e_m),
\end{equation}
and for any (fixed) $T\in(0,\T)$,
\begin{equation}\label{eq:rho'-limitato}
	\rho'(t)^2\leq \frac{(n-1)^2}{\rho_0^2-2(n-1)T}=:M_T,
\end{equation}
for any $t\in[0,T]$.
\end{lem}
\begin{proof}
Rewrite equation \eqref{eq:rho} as the first order system
\begin{equation}\label{eq:rho-rho'}
	\begin{cases}
		\rho'=\nu\\
		\e^2\tau \nu'=-\nu-\frac{n-1}{\rho}
	\end{cases}, \qquad \quad 
	\rho(0)=\rho_0, \quad \nu(0)=\nu_0.
\end{equation}
Denote by $[0,T^\e_m)$ the maximal interval where the solution $(\rho,\nu)$ exists.
The region $$\Gamma:=\{(\rho,\nu) : \rho>0, \; -\frac{(n-1)}{\rho}\leq \nu\leq 0\}$$ is invariant for \eqref{eq:rho}, and in particular
\begin{equation*}
  	(\rho_0,\nu_0)\in \Gamma \qquad  \Longrightarrow \qquad  -\frac{(n-1)}{\rho(t)}\leq \nu(t) \leq \nu_0, \quad \forall\, t\in[0,T^\e_m).
\end{equation*}
It follows that if $\rho_0\in(0,1)$ and $-(n-1)/\rho_0< \nu_0< 0$, then $\rho'(t)\leq \nu_0$ for any $t\in[0,T^\e_m)$.
Therefore, 
\begin{equation*}
	\rho(t)\leq \rho_0+\nu_0t.
\end{equation*}
Hence, since $\nu_0<0$, we have $T^\e_m\leq-\rho_0/\nu_0$. 
On the other hand, in $\Gamma$ we have $\nu\geq-(n-1)/\rho$ and so, $\rho'(t)\geq -(n-1)/\rho(t)$ for any $t\in[0,T^\e_m)$.
This implies 
\begin{equation*}
	\frac12\rho(t)^2-\frac12\rho_0^2\geq-(n-1)t,
\end{equation*}
and, as a consequence
\begin{equation*}
	\rho(t)\geq \sqrt{\rho_0^2-2(n-1)t}.
\end{equation*}
Combining the two estimates for $\rho(t)$, we end up that there exists $T^\e_m\in[\T,\rho_0/|\nu_0|]$ such that $\rho(T^\e_m)=0$.

The properties \eqref{eq:inv-reg} and \eqref{eq:rho'-limitato} follow from the invariance of the region $\Gamma$.
In particular, for any (fixed) $T\in(0,T^\e_m)$, we deduce
\begin{equation*}
	\rho'(t)^2\leq\frac{(n-1)^2}{\rho(t)^2}\leq\frac{(n-1)^2}{\rho_0^2-2(n-1)t}\leq\frac{(n-1)^2}{\rho_0^2-2(n-1)T}=:M_T,
\end{equation*}
for any $t\in[0,T]$.
\end{proof}
Lemma \ref{lem:prop-rho} ensures that the radius $\rho$ vanishes in a finite time $T^\e_m$ and 
we will make use of properties \eqref{eq:inv-reg}, \eqref{eq:rho'-limitato} that hold for any $t\in[0,T]$ with $T<T^\e_m$. 
In Figure \ref{fig:rho=0.6eps=0.02} we show the solution to \eqref{eq:rho} for a particular choice of the parameters and of the initial data in the case $n=2$.
\begin{figure}[htbp]
\centering
\includegraphics[width=5.5cm,height=4.5cm]{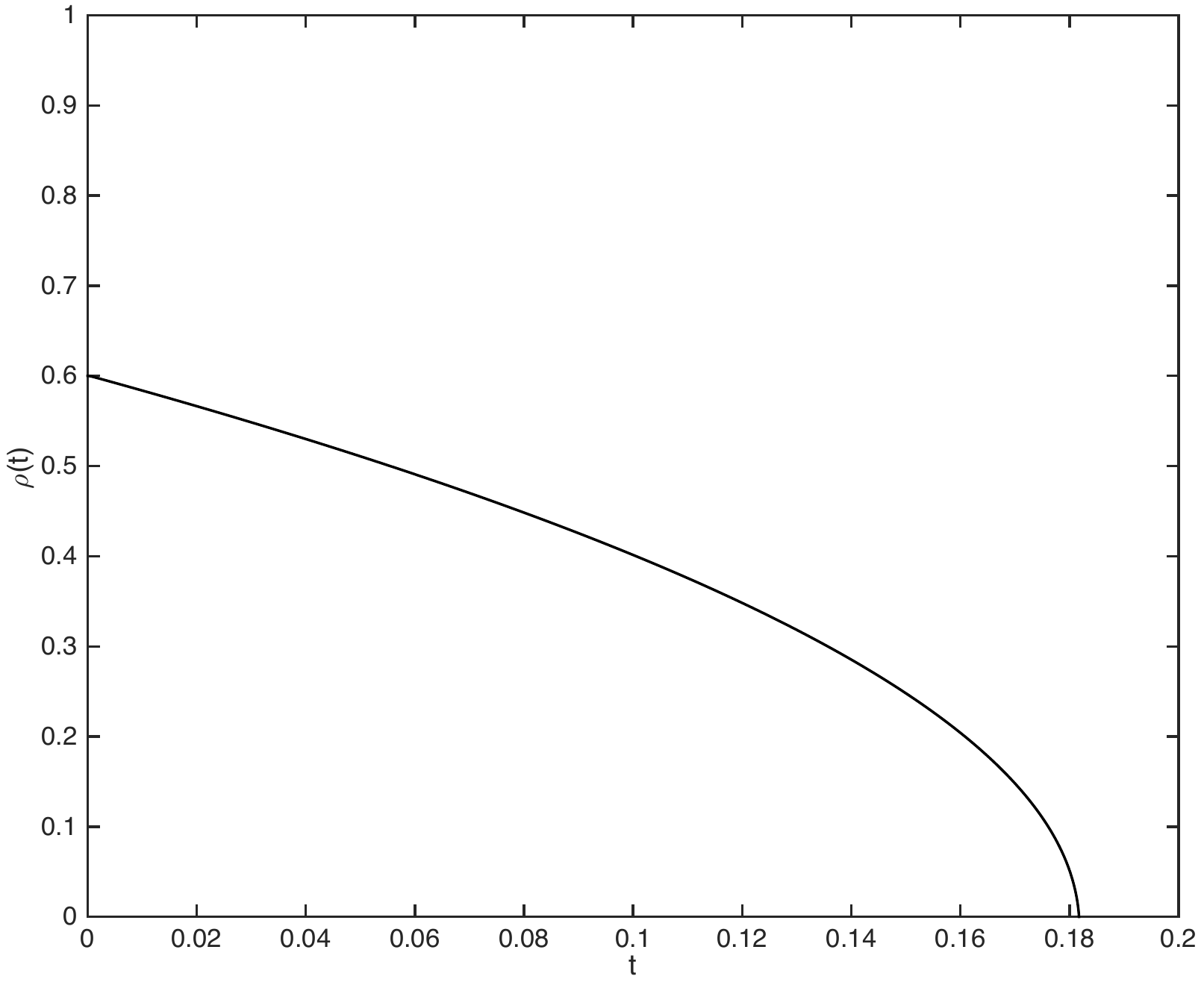}\qquad\quad
\includegraphics[width=5.4cm,height=4.55cm]{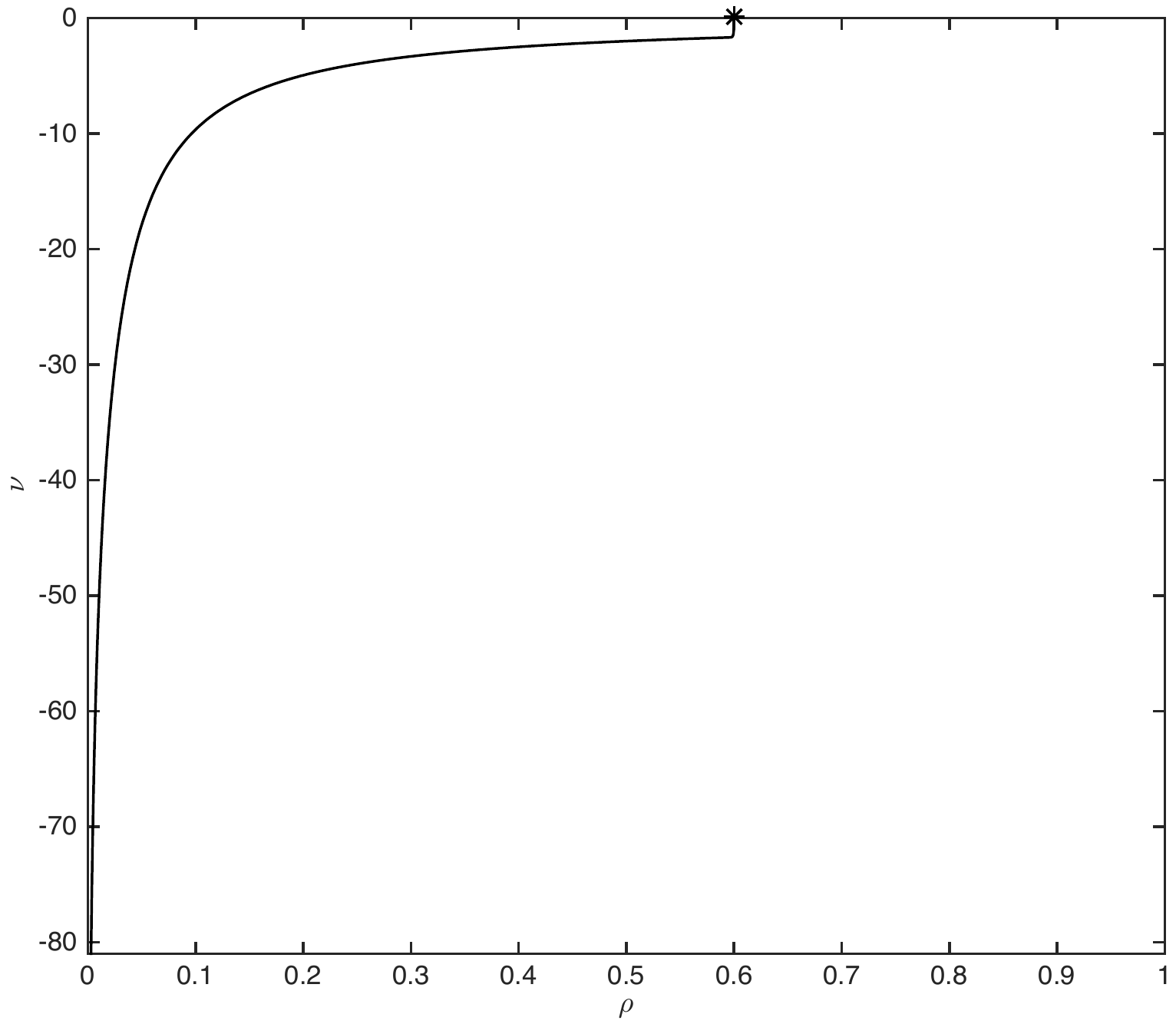}
\caption{Solution to \eqref{eq:rho} with $\tau=1$, $\e=0.02$ and initial data: $\rho_0=0.6$, $\nu_0=0$.}
\label{fig:rho=0.6eps=0.02}
\end{figure}

Let us underline that the behavior of the solutions to \eqref{eq:rho} is described by the ones to $\rho'=-(n-1)/\rho$ as $\eta:=\e^2\tau\to0$ 
and then, $T^\e_m$ tends to $\T$ as $\eta\to0$. 
In Figure \ref{fig:rho=0.6eps}, we show the solutions to \eqref{eq:rho} for $n=2$ with the same initial data, the same parameter $\tau$
and two different values of $\e$. 
Observe that for $\e=0.01$, $T^\e_m$ is very close to $\T$.
Precisely, we have the following result.
\begin{figure}[htbp]
\centering
\includegraphics[width=6cm,height=3cm]{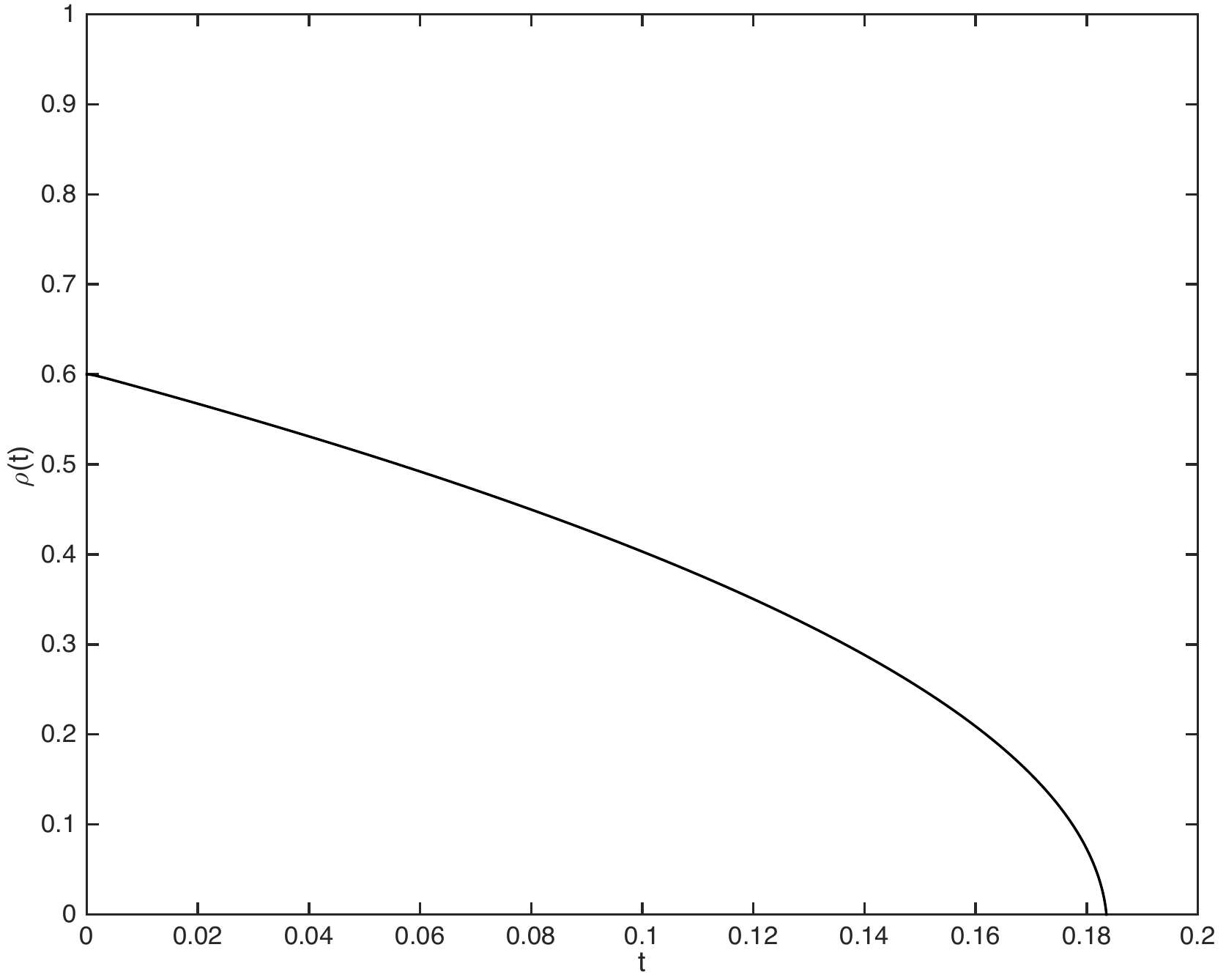}\qquad\quad
\includegraphics[width=6cm,height=3cm]{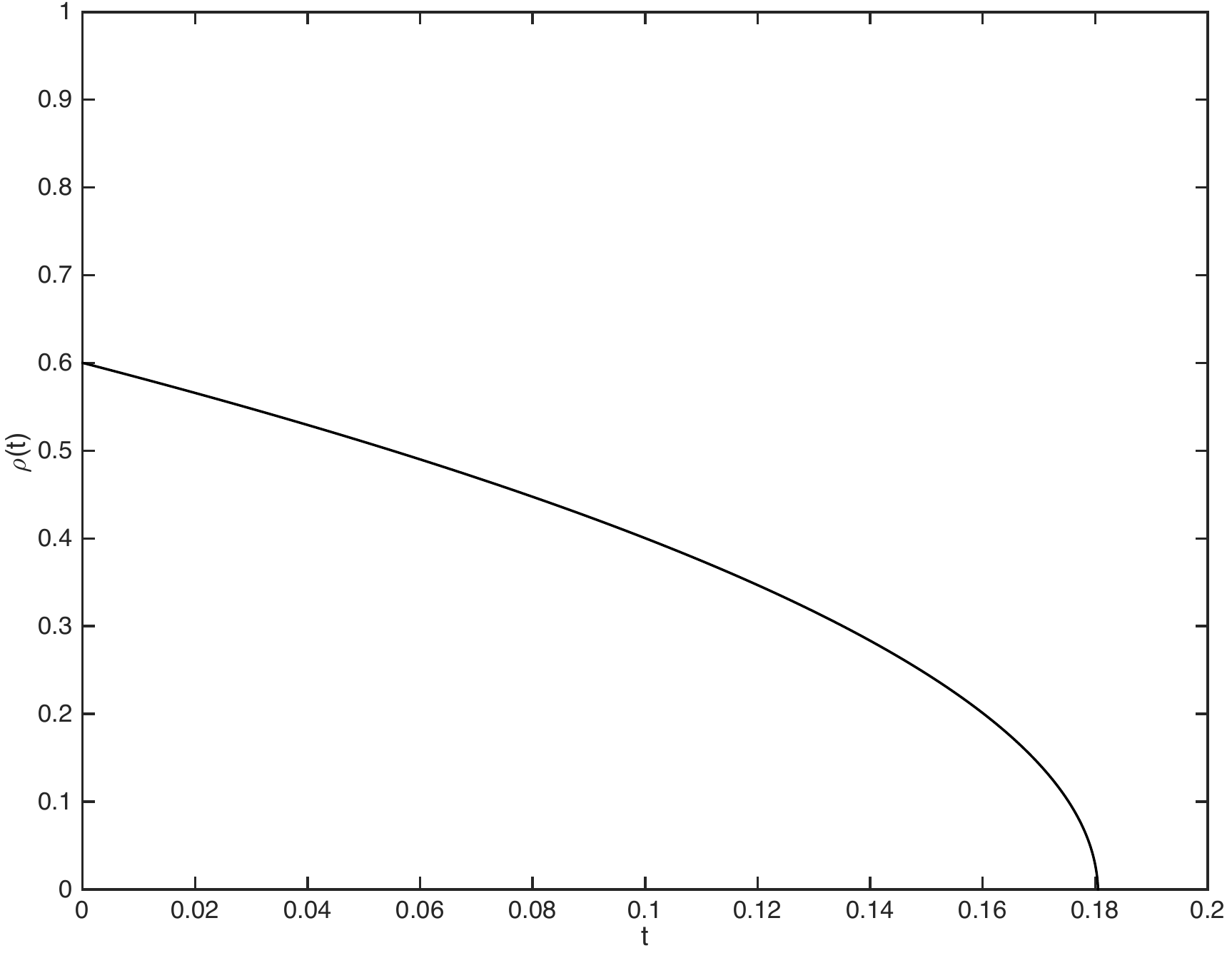}
\caption{Solution to \eqref{eq:rho} with initial data $\rho_0=0.6$, $\nu_0=0$ and $\tau=1$ for different values of $\e$:
		left $\e=0.03$, right $\e=0.01$.}
\label{fig:rho=0.6eps}
\end{figure}
\begin{lem}\label{lem:epstau-rho}
Fix $T\in(0,\T)$.
Let $(\rho,\nu)$ be the solution to \eqref{eq:rho} for $\eta := \e^2\tau$ on $(0,T)$ and $\rho^o(t)=\sqrt{\rho_0-2(n-1)t}$, that is the solution to \eqref{eq:rho} when $\eta=0$.
Then 
\begin{align}
	\lim_{\eta\rightarrow0}\sup_{t\in[0,T]}|\rho(t)-\rho^o(t)|&=0, \label{eq:rho-rho0}\\
	\lim_{\eta\rightarrow0}\sup_{t\in[t_1,T]}\left|\nu(t)+\frac{n-1}{\rho^o(t)}\right|&=0, \label{eq:nu-rho0'}
 \end{align}
for any $t_1\in(0,T)$.
\end{lem}

\begin{proof}
For $t\in[0,T]$, define 
\begin{equation*}
	\chi(t):=\chi_1(t)+\eta\chi_2(t):=\rho(t)-\rho^o(t)+\eta\left(\nu(t)+\frac{n-1}{\rho^o(t)}\right).
\end{equation*}
Recall that, from the assumptions on the initial data, we have $\chi_1(t)\geq0$, $\chi_2(t)\geq0$ and $\rho(t)\geq\rho^o(t)\geq\rho^o(T)$ for any $t\in[0,T]$.
By differentiating, we get
\begin{equation*}
	\chi_1'=\chi_2, \qquad \quad 
	\eta\chi_2'=-\chi_2+\frac{n-1}{\rho^o}-\frac{n-1}{\chi_1+\rho^o}+\eta\frac{(n-1)^2}{(\rho^o)^3}.
\end{equation*}
Using that
\begin{equation*}
	\eta\chi_2'=-\chi_2+\frac{(n-1)\chi_1}{\rho\,\rho^o}+\eta\frac{(n-1)^2}{(\rho^o)^3}\leq-\chi_2+\frac{n-1}{\rho^o(T)}\chi_1+\eta\frac{(n-1)^2}{\rho^o(T)^3},
\end{equation*}
for any $t\in(0,T)$, we deduce that there exists $C>0$ (depending on $T$ but not on $\e$ and $\tau$) such that
\begin{equation*}
	\chi'_1\leq\chi_2, \qquad 
	\eta \chi_2'\leq -\chi_2+C\chi_1+\eta C,
\end{equation*}
for any $t\in(0,T)$.
Summing, one has
\begin{equation*}
	\chi'_1+\eta\chi'_2\leq C\chi_1+\eta C,
\end{equation*}
and so
\begin{equation*}
	\chi'(t)\leq C\chi(t)+\eta C,  \qquad \quad \forall\, t\in[0,T].
\end{equation*}
Integrating and applying Gr\"onwall's Lemma, we obtain that there exists $C>0$ such that
\begin{equation*}
	\chi(t)\leq C(\chi(0)+\eta), \qquad \quad \forall\, t\in[0,T].
\end{equation*}
In particular, it follows that 
\begin{equation*}
	|\rho(t)-\rho^o(t)|\leq C(\chi(0)+\eta), \qquad \quad \forall\, t\in[0,T],
\end{equation*}
and by using that $\chi(0)=\eta|\nu_0+(n-1)/\rho_0|$ we end up with \eqref{eq:rho-rho0}.
Furthermore, we also have that
\begin{equation*}
	\eta\chi'_2\leq-\chi_2+C(\chi(0)+\eta).
\end{equation*}
Hence,
\begin{equation*}
	\eta\left(e^{t/\eta}\chi_2(t)\right)'\leq C(\chi(0)+\eta)e^{t/\eta},
\end{equation*}
and so
\begin{align*}
	\chi_2(t)&\leq C(\chi(0)+\eta)\bigl(1-e^{-t/\eta}\bigr)+\chi_2(0)e^{-t/\eta} \\
	& \leq C(\chi(0)+\eta)+\chi(0)\frac{e^{-t/\eta}}{\eta},
\end{align*}
for $t\in[0,T]$.
Therefore, for any fixed $t_1\in(0,T)$, we obtain \eqref{eq:nu-rho0'}.
\end{proof}

The previous result ensures that the behavior of the solutions to \eqref{eq:rho} is described
by the equation $\rho'=-(n-1)/\rho$ as $\e$ (or $\tau$) is small.
Recall that the latter equation describes the classic motion by mean curvature for radial solutions in the classic case.
Observe also that in Lemma \ref{lem:epstau-rho} we consider initial data as in \eqref{eq:rho}, and so
the properties \eqref{eq:rho-rho0}-\eqref{eq:nu-rho0'} hold for any initial data $(\rho_0,\nu_0)\in(0,1)\times[0,-(n-1)/\rho_0]$.
Let us stress that the scope of this section is to study the evolution of the solutions to \eqref{eq:damped-radial}-\eqref{eq:boundary},
when the initial datum has a transition from $+1$ to $-1$, then in Lemmas \ref{lem:prop-rho}-\ref{lem:epstau-rho} we use particular assumptions
on the initial data $\rho_0$, $\nu_0$.
However, in the case $(\rho_0,\nu_0)\notin\Gamma$, where $\Gamma$ is the invariant region, the solution enters to $\Gamma$ in a very short time and we have the same behavior
described in Lemma \ref{lem:prop-rho} (see an example in Figure \ref{fig:rho'=50}).
\begin{figure}[htbp]
\centering
\includegraphics[width=5cm,height=4cm]{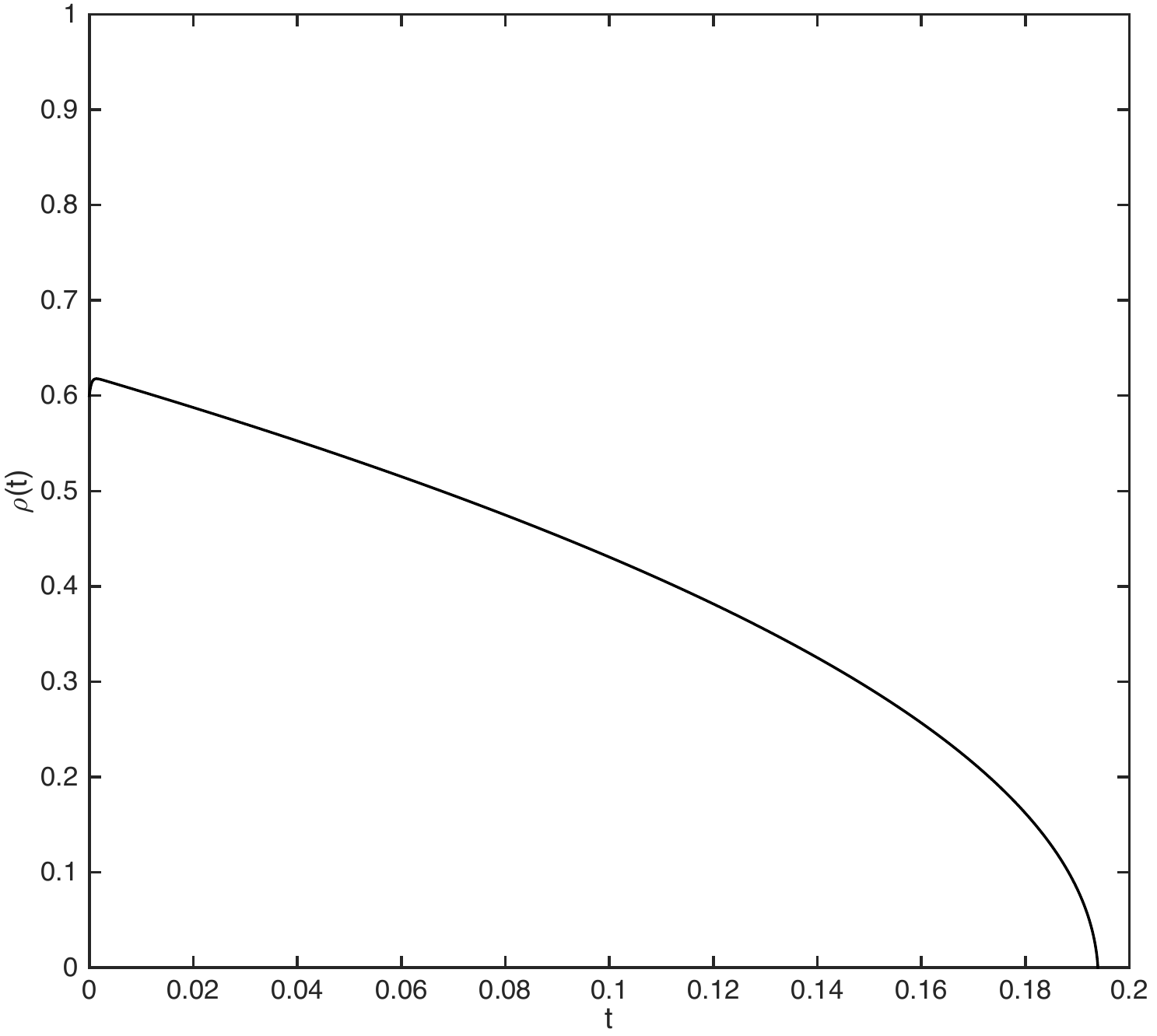}\qquad\quad
\includegraphics[width=5cm,height=4cm]{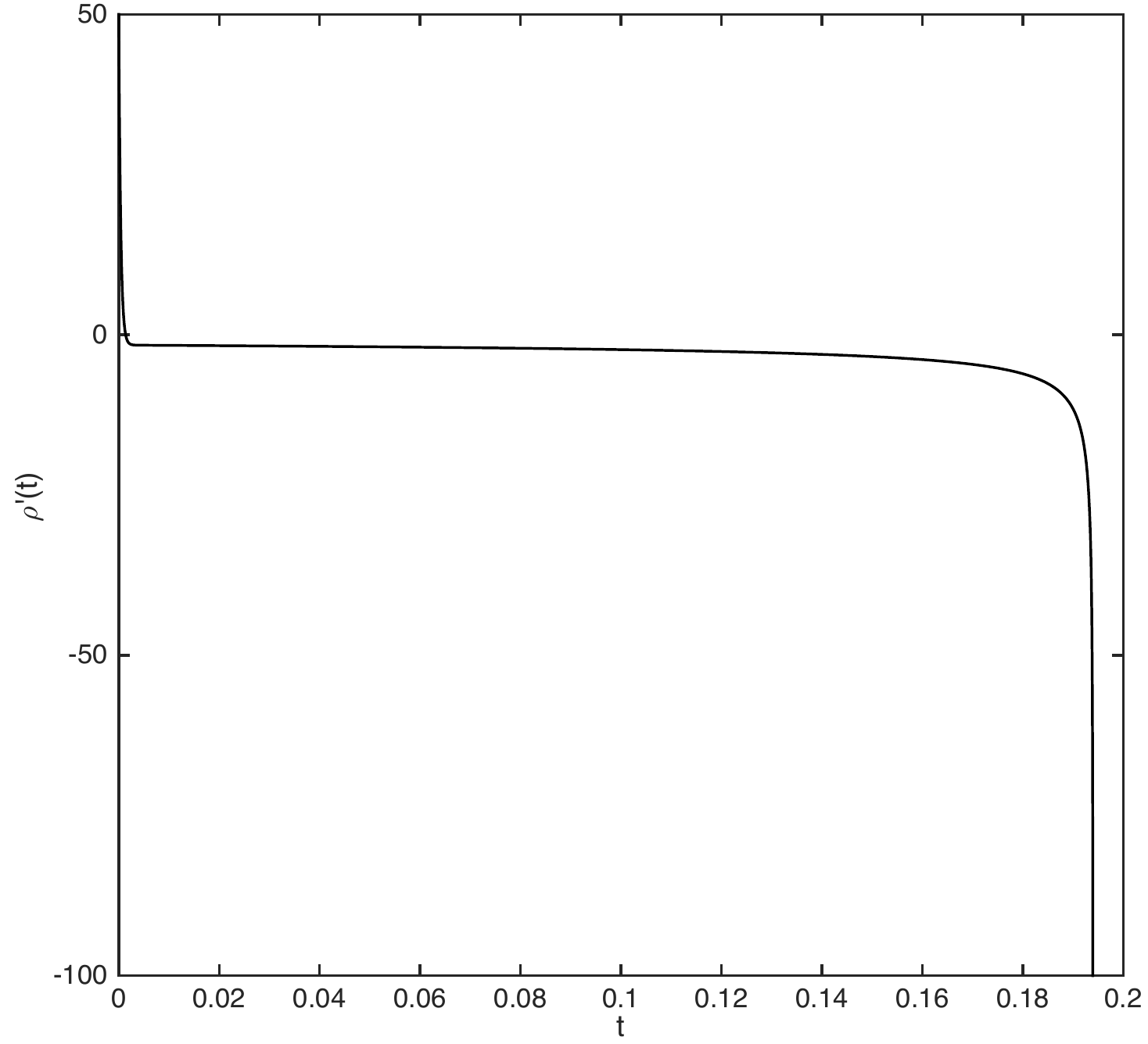}\\ \vspace{0.4cm}
\includegraphics[width=5.5cm,height=4.5cm]{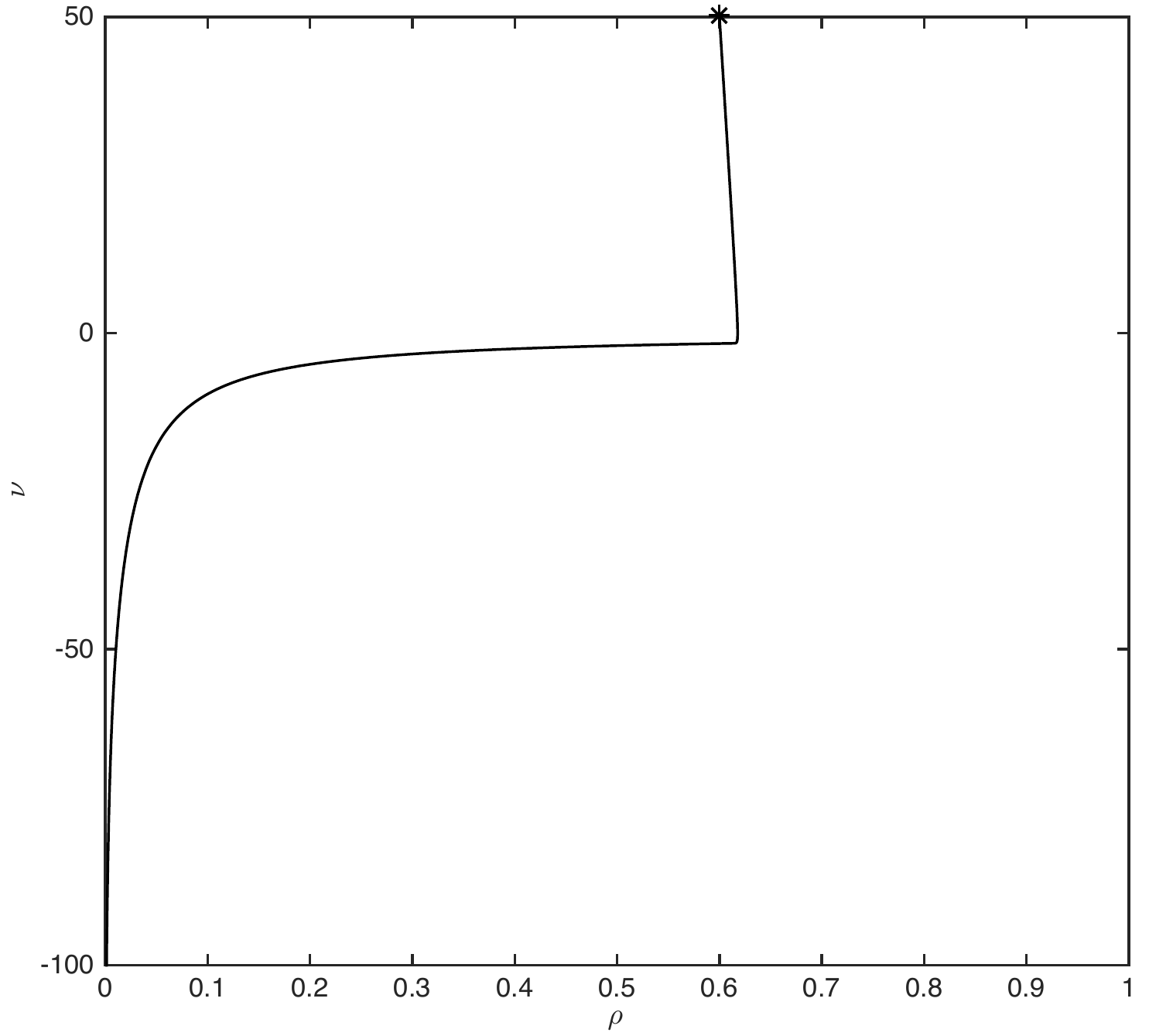}
\caption{Solution to \eqref{eq:rho} with $\tau=1$, $\e=0.02$ and initial data $\rho_0=0.6$, $\nu_0=50$.}
\label{fig:rho'=50}
\end{figure}

\subsection{Change of variables}
Following \cite{Bron-Kohn2}, we shall work in a moving coordinate system with respect to which $u^\e$ should be asymptotically stationary.
Then, we introduce the new variable $R=r-\rho(t)$ and define
\begin{equation}\label{eq:change-var}
	v^\e(R,t)=u^\e(R+\rho(t),t), \qquad \mbox{ or, equivalently } \qquad u^\e(r,t)=v^\e(r-\rho(t),t).
\end{equation}
The function $v^\e$ is defined for $R\in[-\rho(t),1-\rho(t)]$, with $t\in(0,T^\e_m)$ and 
$T^\e_m$, given in Lemma \ref{lem:prop-rho}, is the time when $\rho$ vanishes.
By differentiating \eqref{eq:change-var}, we infer
\begin{align*}
	u^\e_t(r,t)&=-\rho'(t)v^\e_R(R,t)+v^\e_t(R,t), \\
	u^\e_{tt}(r,t)&=-\rho''(t)v^\e_R(R,t)+\rho'(t)^2v^\e_{RR}(R,t)-2\rho'(t)v^\e_{tR}(R,t)+v^\e_{tt}(R,t),\\
	u_r(r,t)&=v_R(R,t).
\end{align*}
Fix $T\in(0,\T)$ where $\T$ is the constant defined in \eqref{eq:Tmax}, it follows that if $u^\e$ satisfies \eqref{eq:damped-radial}, then
\begin{align*}
	\e^2\tau v^\e_{tt}-2\e^2\tau\rho'v^\e_{tR}+v_t=&\left(1-\e^2\tau(\rho')^2\right)v^\e_{RR}\\
	&\quad+\left(\e^2\tau\rho''+\rho'+\frac{n-1}{R+\rho}\right)v^\e_R-\e^{-2}F'(v^\e),
\end{align*}
with $(R,t)\in(-\rho(t),1-\rho(t))\times(0,T)$.
Since $\rho$ satisfies \eqref{eq:rho}, we obtain
\begin{equation*}
	\e^2\tau\rho''+\rho'+\frac{n-1}{R+\rho}=-\frac{n-1}{\rho}+\frac{n-1}{R+\rho}=-\frac{(n-1)R}{\rho(R+\rho)},
\end{equation*}
and so,
\begin{equation*}
	\e^2\tau v^\e_{tt}-2\e^2\tau\rho'v^\e_{tR}+v^\e_t=\left(1-\e^2\tau(\rho')^2\right)v^\e_{RR}-\frac{(n-1)R}{\rho(R+\rho)}v^\e_R-\e^{-2}F'(v^\e).
\end{equation*}
We want the coefficient of $v^\e_{RR}$ to be strictly positive; using \eqref{eq:rho'-limitato}, we have
\begin{equation*}
	1-\e^2\tau\rho'(t)^2\geq1-\e^2\tau\frac{(n-1)^2}{\rho_0^2-2(n-1)T}=:1-\e^2\tau M_T, \qquad \quad \forall \, t\in(0,T).
\end{equation*}
Therefore we choose $\e_0=\e_0(T)$ sufficiently small so that $\e_0^2\tau M_T\leq1-\alpha$, where $\alpha\in(0,1)$.
Hence, we can state that for any $\e\in(0,\e_0)$
\begin{equation}\label{eq:alfa}
	1-\e^2\tau\rho'(t)^2\geq\alpha, \qquad \quad \forall \, t\in[0,T].
\end{equation}
Now, let us rewrite the equation for $v$ as
\begin{equation}\label{eq:v}
	\e^2\tau v^\e_{tt}-2\e^2\tau\rho'v^\e_{tR}+v^\e_t=\left(1-\e^2\tau(\rho')^2\right)\frac{(\phi^\e v^\e_R)_R}{\phi^\e}-\e^{-2}F'(v^\e),
\end{equation}
where the integrating factor $\phi^\e$ satisfies
\begin{equation}\label{eq:phi_R}
	\phi^\e_R=-\frac{(n-1)R}{\rho(R+\rho)\left(1-\e^2\tau(\rho')^2\right)}\phi^\e.
\end{equation}
Equation \eqref{eq:v} is complemented with the boundary conditions
\begin{equation}\label{eq:v-boundary}
	v^\e(1-\rho(t),t)=1, \qquad v^\e_R(-\rho(t),t)=0, \qquad \quad \forall \, t\in(0,T).
\end{equation}
The next step is to study the problem \eqref{eq:v}-\eqref{eq:v-boundary} in the domain $[-\rho(t),1-\rho(t)]\times[0,T]$,
where $T$ is a fixed constant strictly less than $\T$ and for $\e$ sufficiently small so that \eqref{eq:alfa} holds.
To start with, in the next subsection we collect some properties of the integrating factor $\phi^\e$.
From now on, we drop the superscript $\e$ and we use the notation $\phi=\phi^\e$.

\subsection{Properties of $\phi$}
Let us explicitly compute the solution of \eqref{eq:phi_R} satisfying $\phi(0)=1$.
Integrating \eqref{eq:phi_R}, we get
\begin{align*}
	\ln(\phi(R,t))&=-\int_0^R\frac{(n-1)s}{\rho(t)(s+\rho(t))\left(1-\e^2\tau\rho'(t)^2\right)}ds \\
	&=-\frac{(n-1)R}{\rho(t)\left(1-\e^2\tau\rho'(t)^2\right)}+\frac{n-1}{1-\e^2\tau\rho'(t)^2}\ln\left(\frac{R+\rho(t)}{\rho(t)}\right).
\end{align*}
Hence, we choose the integrating factor
\begin{equation}\label{eq:phi}
	\phi(R,t)=\exp\left(-\frac{(n-1)R}{\rho(t)\left(1-\e^2\tau\rho'(t)^2\right)}\right)\left(1+\frac{R}{\rho(t)}\right)^\frac{n-1}{1-\e^2\tau\rho'(t)^2}.
\end{equation}
From the smallness of $\e$ and \eqref{eq:alfa}, it follows that $\phi$ is well defined and positive in the domain
$[-\rho(t),1-\rho(t)]\times[0,T]$. 
Precisely, $\phi$ is zero if and only if $R=-\rho$ and we have
\begin{equation}\label{eq:prop-phi}
	0\leq\phi(R,t)\leq1, \qquad \phi(-\rho(t),t)=0, \quad \phi(0,t)=1.
\end{equation}
Furthermore, we will make use later of the following properties of $\phi$.
\begin{lem}
Let $\phi$ be the function defined in \eqref{eq:phi} in the domain $[-\rho(t),1-\rho(t)]\times[0,T]$ where $T<\T$,
$\rho$ satisfies \eqref{eq:rho} and $\e$ is sufficiently small that \eqref{eq:alfa} holds.
Then,
\begin{equation}\label{eq:prop-phi2}
	\phi(-R,t)\leq\phi(R,t), \qquad \forall \, (R,t)\in(0,\rho(t))\times(0,T),
\end{equation}
and, there exist positive constants $c,K_T$ such that
\begin{equation}\label{eq:prop-phi3}
	\phi(R,t)\geq1-K_TR^2, \qquad \forall \, (R,t)\in(-c,c)\times(0,T).
\end{equation}
\end{lem}
\begin{proof}
In order to prove \eqref{eq:prop-phi2}, we use the inequality
\begin{equation*}
	e^x(1-x)\leq e^{-x}(1+x), \qquad  \mbox{ for any }\,  x\in(0,1).
\end{equation*} 
Let us introduce $k:=(n-1)/(1-\e^2\tau(\rho')^2)$ and observe that $k\in[n-1,(n-1)/\alpha]$ for \eqref{eq:alfa}.
By elevating the above inequality to the power $k$ and for $x=R/\rho$ we obtain \eqref{eq:prop-phi2}.

Similarly, the property \eqref{eq:prop-phi3} follows from the inequality 
\begin{equation*}
	\exp(-kx)(1+x)^k\geq1-k^2x^2, 
\end{equation*} 
which holds for $x$ in a neighborhood of $0$ and for all $k\geq1$, because $x=0$ is a minimal point of the function $\exp(-kx)(1+x)^k-1+k^2x^2$.
Therefore, we deduce that there exists a constant $c>0$ such that for all $(R,t)\in(-c,c)\times(0,T)$, one has
\begin{equation*}
	\phi(R,t)\geq1-\frac{(n-1)^2}{\rho(t)^2(1-\e^2\tau\rho'(t)^2)^2}R^2. 
\end{equation*}
Using Lemma \ref{lem:prop-rho} and \eqref{eq:alfa}, we conclude that for $|R|$ sufficiently small and $t\in(0,T)$
\begin{equation*}
	\phi(R,t)\geq1-\frac{(n-1)^2}{\rho(T)^2\alpha^2}R^2,
\end{equation*}
that is \eqref{eq:prop-phi3} with $K_T:=\frac{(n-1)^2}{\alpha^2(\rho_0^2-2(n-1)T)}$.
\end{proof}
Now, let us consider the derivatives of $\phi$.
Regarding the derivative $\phi_R$, from the equation \eqref{eq:phi_R} and \eqref{eq:phi} it follows that 
$\phi_R$ is bounded and satisfies for any $t\in(0,T)$
\begin{equation}\label{eq:prop-phi_R}
	\begin{aligned}
		\phi_R(R,t)>0 \; & \mbox{ in} \, (-\rho(t),0), \qquad \phi_R(R,t)<0 \; \mbox{ in} \, (0,1-\rho(t)], \\
		& \phi_R(-\rho(t),t)=\phi_R(0,t)=0.
	\end{aligned}		
\end{equation}
For the time derivative $\phi_t$ we have the following result.
\begin{lem}
Let $\phi$ be the function defined in \eqref{eq:phi} with $\rho$ satisfying \eqref{eq:rho} and $\e$ sufficiently small that \eqref{eq:alfa} holds.
Then, for $(R,t)\in[-\rho(t),1-\rho(t)]\times[0,T]$ we have
\begin{equation}\label{eq:phi_t-phi_R}
	\phi_t(R,t)\leq -\frac{\rho'(t)}{\rho(t)}R\,\phi_R(R,t)\leq0.
\end{equation}
\end{lem}
\begin{proof}
Let us compute the time derivative of $\phi$ (we use the notation $\rho=\rho(t)$):
\begin{align*}
	\phi_t(R,t)=\phi(&R,t)\Biggl\{\frac{(n-1)R\rho'\left[\left(1-\e^2\tau(\rho')^2\right)-2\e^2\tau\rho\rho''\right]}{\rho^2\left(1-\e^2\tau(\rho')^2\right)^2}\\
	&+\frac{2\e^2\tau\rho'\rho''(n-1)}{\left(1-\e^2\tau(\rho')^2\right)^2}\ln\left(1+\frac{R}{\rho}\right)
	-\frac{R(n-1)\rho\rho'}{\left(1-\e^2\tau(\rho')^2\right)(R+\rho)\rho^2}\Biggr\}.
\end{align*}
Simplifying we get
\begin{align*}
	\phi_t(R,t)=\frac{\phi(R,t)(n-1)\rho'}{\rho^2(R+\rho)\left(1-\e^2\tau(\rho')^2\right)^2}\Bigl\{R^2&\left(1-\e^2\tau(\rho')^2\right)\\
	\qquad&-2\e^2\tau\rho''\rho(R+\rho)I(R,t)\Bigr\},
\end{align*}
where $I(R,t):=R-\rho\ln\left(1+\frac{R}{\rho}\right)$.
In order to determine the sign of $\phi_t$, we observe that, since $\rho$ satisfies \eqref{eq:rho},
$\e^2\tau\rho''\rho=-\rho'\rho-(n-1)$ and as a consequence, if $I(R,t)\geq0$ then \eqref{eq:inv-reg} and \eqref{eq:alfa}
imply $\phi_t(R,t)\leq0$ in the domain $[-\rho(t),1-\rho(t)]\times[0,T]$.
However,
\begin{equation*}
	I(R,t)=R\left(1-\frac{\rho}{R}\ln\left(1+\frac{R}{\rho}\right)\right)\geq0,
\end{equation*}
because the function $1-\ln(1+x)/x$ is positive for $x\geq0$ and negative for $x<0$.
Then, $\phi_t$ si negative in the domain.
Precisely, in $[-\rho(t),1-\rho(t)]\times[0,T]$ we have  
\begin{equation*}
	\phi_t(R,t)\leq\frac{\phi(R,t)(n-1)\rho'}{\rho^2(R+\rho)\left(1-\e^2\tau(\rho')^2\right)^2}\Bigl\{R^2\left(1-\e^2\tau(\rho')^2\right)\Bigr\},
\end{equation*}
and using \eqref{eq:phi_R} we get \eqref{eq:phi_t-phi_R}.
\end{proof}

\subsection{The energy functional}
Now, we introduce the functional
\begin{equation}\label{eq:energy-v}
	E_\phi[v,v_t](t):=\int_{-\rho(t)}^{1-\rho(t)}\left[\frac{\e^3\tau}{2}(v_t)^2+\e\left(1-\e^2\tau(\rho')^2\right)\frac{v_R^2}{2}+\e^{-1}F(v)\right]\phi\,dR,
\end{equation}
for $t\in[0,T]$ and $T<\T$, where $\rho$ satisfies \eqref{eq:rho}, $\e$ is so small that \eqref{eq:alfa} holds, 
$F$ satisfies \eqref{eq:ass-F-radial} and $\phi$ is defined in \eqref{eq:phi}.
In particular, the smallness of $\e$ (see \eqref{eq:alfa}) and the positivity of the integrating factor $\phi$ guarantee that $E_\phi[v,v_t](t)\geq0$, for all $t\in[0,T]$.
The goal of this subsection is to study the evolution of $E_\phi$ along the solutions $(v^\e,v^\e_t)$ to the problem \eqref{eq:v}-\eqref{eq:v-boundary}.
To simplify notation, we write $(v,v_t)$ instead of $(v^\e,v^\e_t)$. 

As we will see, the main problem is the presence of the term
\begin{equation*}
	\int_{\rho(t)}^{1-\rho(t)}v_t(R,t)^2\phi_R(R,t)\,dR,
\end{equation*}
because $\phi_R\sim \phi/(R+\rho)\sim(R+\rho)^{\frac{n-1}{1-\e^2\tau\rho'(t)^2}-1}$ for $R$ close to $-\rho$.
For any (fixed) $t>0$, since we are studying the behavior of the solutions when $\e\to0$, we need a control on
\begin{equation*}
	\int_{\rho(t)}^{1-\rho(t)}v_t(R,t)^2(R+\rho)^{n-2}\,dR,
\end{equation*}
which has a problem at $R=-\rho$.
As we will see, a possible way to overcome such problem is to use the higher order estimates introduced in Section \ref{sec:higher} 
and impose that the parameter $\tau$ depends on $\e$ in a way such that \eqref{eq:tau-ugly} holds.

\begin{prop}
Fix $T\in(0,\T)$, where $\T$ is defined in \eqref{eq:Tmax} and let $(v,v_t)$ be a sufficiently regular solution to the BVP \eqref{eq:v}-\eqref{eq:v-boundary} 
where $\tau$ satisfies \eqref{eq:tau-ugly}, $\e$ is so small that \eqref{eq:alfa} holds, $\rho$ satisfies \eqref{eq:rho}, and $\phi$ is defined in \eqref{eq:phi}.
Then, the functional \eqref{eq:energy-v} satisfies for any $t\in[0,T]$ and for any $\e$ sufficiently small,
\begin{equation}\label{eq:E'(t)}
	 \frac{d}{dt}E_\phi[v,v_t](t)\leq -\beta\e \int_{-\rho(t)}^{1-\rho(t)}v_t(R,t)^2\phi(R,t)\,dR+J^\e(t), 
\end{equation}
for some $\beta>0$ (independent on $\e,\tau,T$), where for $\mu$ as in \eqref{eq:tau-ugly}
\begin{align}
	J^\e(t):=& -\e^3\tau\rho'(t)\int_{-\rho(t)}^{-\rho(t)+\e^{2+\mu}}v_t(R,t)^2\phi_R(R,t)\,dR\nonumber\\
	 & -\frac{\rho'(t)}{\rho(t)}\int_{-\rho(t)}^{1-\rho(t)}\left[\frac{\e^3\tau}{2}(v_t)^2+\e\alpha\frac{v_R^2}{2}\right]R\,\phi_R\,dR.
	 \label{eq:J}
\end{align}
\end{prop}
\begin{proof}
Let us differentiate the functional $E_\phi$ defined in \eqref{eq:energy-v} with respect to $t$:
\begin{equation*}
	\frac{d}{dt}E_\phi[v,v_t]=I_1(t)+I_2(t)+I_3(t)+I_4(t),
\end{equation*}
where
\begin{align*}
	I_1(t)&:=\int_{-\rho(t)}^{1-\rho(t)}\left[\e^3\tau v_tv_{tt}+\e\left(1-\e^2\tau(\rho')^2\right)v_Rv_{Rt}+\e^{-1}F'(v)v_t\right]\phi\,dR,\\
	I_2(t)&:=-\int_{-\rho(t)}^{1-\rho(t)}\e^3\tau\rho'\rho''v_R^2\phi\,dR,\\
	I_3(t)&:=\int_{-\rho(t)}^{1-\rho(t)}\left[\frac{\e^3\tau}{2}(v_t)^2+\e\left(1-\e^2\tau(\rho')^2\right)\frac{v_R^2}{2}+\e^{-1}F(v)\right]\phi_t\,dR,\\
	I_4(t)&:=\Biggl[\left(\frac{\e^3\tau}{2}(v_t)^2+\e\left(1-\e^2\tau(\rho')^2\right)\frac{v_R^2}{2}+\e^{-1}F(v)\right)\phi\Biggr]^{1-\rho(t)}_{-\rho(t)}(-\rho'(t)).
\end{align*}
Integrating by parts, we get
\begin{equation*}
	\int_{-\rho(t)}^{1-\rho(t)}v_Rv_{Rt}\phi\,dR=\biggl[v_Rv_t\phi\biggr]_{-\rho(t)}^{1-\rho(t)}-\int_{-\rho(t)}^{1-\rho(t)}v_t{(\phi v_R)}_R\,dR,
\end{equation*}
By substituting, we infer
\begin{equation*}
	I_1=\e\int_{-\rho}^{1-\rho}\left[\e^2\tau v_{tt}\phi-\left(1-\e^2\tau(\rho')^2\right){(\phi v_R)}_R+\e^{-2}F'(v)\phi\right]v_t\,dR+I_5,
\end{equation*}
where
\begin{equation*}
	I_5(t):=\e\left(1-\e^2\tau(\rho')^2\right)\biggl[v_Rv_t\phi\biggr]_{-\rho(t)}^{1-\rho(t)}.
\end{equation*}
From the equation for $v$ \eqref{eq:v}, it follows that 
$$
	\e^2\tau v_{tt}\phi-\left(1-\e^2\tau(\rho')^2\right){(\phi v_R)}_R+\e^{-2}F'(v)\phi=-\phi v_t+2\e^2\tau\rho'v_{tR}\phi,
$$ 
and so
\begin{equation}\label{eq:I1}
	I_1(t)=-\e\int_{-\rho(t)}^{1-\rho(t)}\phi v_t^2\,dR +I_5(t)
	+2\e^3\tau\rho'\int_{-\rho(t)}^{1-\rho(t)}v_{tR}v_t\phi\,dR.
\end{equation}
Regarding $I_2$, since $\rho$ satisfies equation \eqref{eq:rho}, $\e^2\tau\rho'\rho''=-\frac{\rho'}{\rho}(\rho\rho'+n-1)$,
and from \eqref{eq:inv-reg} and the positivity of $\phi$ it follows that
\begin{equation}\label{eq:I2}
	I_2(t)\leq0, \qquad \quad \forall\, t\in[0,T].
\end{equation}
Moreover, using \eqref{eq:alfa}, \eqref{eq:phi_t-phi_R} and the positivity of $F$ we get
\begin{equation}\label{eq:I3}
	I_3(t)\leq-\frac{\rho'(t)}{\rho(t)}\int_{-\rho(t)}^{1-\rho(t)}\left[\frac{\e^3\tau}{2}(v_t)^2+\e\alpha\frac{v_R^2}{2}\right]R\,\phi_R\,dR,
\end{equation}
for all $t\in[0,T]$.
It remains to study $I_4$; first of all, notice that differentiating the first boundary condition of \eqref{eq:v-boundary} we infer
\begin{equation}\label{eq:vt-vR}
	v_t(1-\rho(t),t)=v_R(1-\rho(t),t)\rho'(t), \qquad \quad \forall\, t\in[0,T].
\end{equation}
Thus, using \eqref{eq:v-boundary}, \eqref{eq:vt-vR} and the fact the $\phi(-\rho(t),t)=0$ (see \eqref{eq:prop-phi}), we obtain
\begin{equation}\label{eq:I4}
	I_4(t)=-\frac\e2\rho'(t)v_R(1-\rho(t),t)^2\phi(1-\rho(t),t).
\end{equation}
At the same way, we deduce that
\begin{equation}\label{eq:I5}
	I_5(t):=\e\left(1-\e^2\tau\rho'(t)^2\right)\rho'(t)v_R(1-\rho(t),t)^2\phi(1-\rho(t),t).
\end{equation}
Combining \eqref{eq:I1}, \eqref{eq:I2}, \eqref{eq:I3}, \eqref{eq:I4} and \eqref{eq:I5} we end up with	
\begin{align*}
	\frac{d}{dt}E_\phi[v,v_t]\leq &-\e\int_{-\rho(t)}^{1-\rho(t)} v_t^2\phi\,dR+2\e^3\tau\rho'(t)\int_{-\rho(t)}^{1-\rho(t)}v_{tR}v_t\phi\,dR\\
	&-\frac{\rho'(t)}{\rho(t)}\int_{-\rho(t)}^{1-\rho(t)}\left[\frac{\e^3\tau}{2}(v_t)^2+\e\alpha\frac{v_R^2}{2}\right]R\,\phi_R\,dR\\
	&+\e\left(\frac12-\e^2\tau\rho'(t)^2\right)\rho'(t)v_R(1-\rho(t),t)^2\phi(1-\rho(t),t).
\end{align*}
Using that $2v_t(R,t)v_{tR}(R,t)=\frac{d}{dR}v_t(R,t)^2$ and integrating by parts we obtain
\begin{align*}
	\frac{d}{dt}E_\phi[v,v_t]\leq &-\e\int_{-\rho(t)}^{1-\rho(t)}v_t(R,t)^2\phi(R,t)\,dR\\
	&+\e^3\tau\rho'(t)v_t(1-\rho(t),t)^2\phi(1-\rho(t),t)\\
	&-\e^3\tau\rho'(t)\int_{-\rho(t)}^{1-\rho(t)}v_t(R,t)^2\phi_R(R,t)\,dR\\
	&-\frac{\rho'(t)}{\rho(t)}\int_{-\rho(t)}^{1-\rho(t)}\left[\frac{\e^3\tau}{2}(v_t)^2+\e\alpha\frac{v_R^2}{2}\right]R\,\phi_R\,dR\\
	&+\e\left(\frac12-\e^2\tau\rho'(t)^2\right)\rho'(t)v_R(1-\rho(t),t)^2\phi(1-\rho(t),t).
\end{align*}
Using \eqref{eq:vt-vR}, we conclude that
\begin{align*}
	\frac{d}{dt}E_\phi[v,v_t]\leq &-\e\int_{-\rho(t)}^{1-\rho(t)}v_t(R,t)^2\phi(R,t)\,dR+I_6(t)\\
	&+\frac\e2\rho'(t)v_R(1-\rho(t),t)^2\phi(1-\rho(t),t),
\end{align*}
where
\begin{align*}
	I_6(t):=& -\e^3\tau\rho'(t)\int_{-\rho(t)}^{1-\rho(t)}v_t(R,t)^2\phi_R(R,t)\,dR\\
	&-\frac{\rho'(t)}{\rho(t)}\int_{-\rho(t)}^{1-\rho(t)}\left[\frac{\e^3\tau}{2}(v_t)^2+\e\alpha\frac{v_R^2}{2}\right]R\,\phi_R\,dR.
\end{align*}
Since $\rho'$ is negative and $\phi$ is positive, we infer
\begin{equation}\label{eq:E'(t)-I6}
	\frac{d}{dt}E_\phi[v,v_t](t)\leq -\e\int_{-\rho(t)}^{1-\rho(t)}v_t(R,t)^2\phi(R,t)\,dR+I_6(t),
\end{equation}
for all $t\in[0,T]$.
The second integral in $I_6(t)$ is indeed non positive, but we keep it for later use. 
Concerning the first one, for which we are not able to determine its sign, we split it as follows:
\begin{align*}
	 \int_{-\rho(t)}^{1-\rho(t)}v_t(R,t)^2\phi_R(R,t)\,dR & =  \int_{-\rho(t)}^{-\rho(t) +\e^{2+\mu}}v_t(R,t)^2\phi_R(R,t)\,dR \\
	&\quad +  \int_{-\rho(t) +\e^{2+\mu}}^{1-\rho(t)}v_t(R,t)^2\phi_R(R,t)\,dR.
\end{align*} 
where $\mu$ is the same of \eqref{eq:tau-ugly}.
From \eqref{eq:phi_R} there exists a constant $C=C(\alpha,T)>0$ (independent of $\e$) such that 
\begin{equation*}
	|\phi_R(R,t)|<C\e^{-(2+\mu)} \phi(R,t),
\end{equation*}
for any $(R,t)\in [-\rho(t) + \e^{2+\mu}, 1 - \rho(t)]\times [0,T]$ and therefore \eqref{eq:E'(t)-I6} becomes
\begin{align*}
	\frac{d}{dt}E_\phi[v,v_t](t)\leq &-\e\int_{-\rho(t)}^{-\rho(t)+\e^{2+\mu}}v_t(R,t)^2\phi(R,t)\,dR\\
	&\qquad-(1+C\tau\e^{-\mu}\rho'(t))\e\int_{-\rho(t)+\e^{2+\mu}}^{1-\rho(t)}v_t(R,t)^2\phi(R,t)\,dR+J^\e(t),
\end{align*}
for all $t\in[0,T]$, where $J^\e$ is defined in \eqref{eq:J}.
Therefore, we obtain inequality \eqref{eq:E'(t)} choosing $\e$ so small that $1+C\tau\e^{-\mu}\rho'(t)\geq\beta$ for some $\beta\in[0,1]$ and for any $t\in[0,T]$.
Indeed, using that $\rho'(t)\geq-(n-1)/\rho(t)$ for any $t\in[0,T]$, we deduce
\begin{equation*}
	1+C\tau\e^{-\mu}\rho'(t)\geq 1-\frac{C(n-1)}{\rho(t)}\tau\e^{-\mu}\geq 1-\frac{C(n-1)}{\rho(T)}\tau\e^{-\mu},
\end{equation*} 
for any $t\in[0,T]$ and the proof is complete thanks to \eqref{eq:tau-ugly}.
\end{proof}
By integrating \eqref{eq:E'(t)} we obtain
\begin{equation}\label{eq:energy-variation}
	E_\phi[v,v_t](0)-E_\phi[v,v_t](\bar T)\geq\beta\e\int_0^{\bar T}\int_{-\rho(t)}^{1-\rho(t)}v_t(R,t)^2\phi(R,t)\,dtdR-h(\e),
\end{equation}
for all $\bar T\in[0,T]$, where 
\begin{equation*}
	h(\e):=\int_0^{T}J^\e(t)\,dt.
\end{equation*}
If $h$ is negative, in view of \eqref{eq:energy-variation}, we have that $E_\phi[v,v_t](0)\geq E_\phi[v,v_t](\bar T)$ for any $\bar T\in[0,T]$.
In order to use \eqref{eq:energy-variation} in the case of $h$ strictly positive we need the following result.
\begin{prop}
Assume that $T<\T$, $J^\e$ is defined in \eqref{eq:J}, $\rho$ satisfies \eqref{eq:rho} and $v$ is given by the change of variables \eqref{eq:change-var},
where $u$ is a solution to \eqref{eq:damped-radial} satisfying the same assumptions of Proposition \ref{prop:higher}.
Then,
\begin{equation}\label{eq:ass-h}
	\lim_{\e\to0}h(\e)=0.
\end{equation}
\end{prop}
\begin{proof}
We have that
\begin{align*}
	J^\e(t)& \leq J^\e_1(t)+J^\e_2(t)\\
	& :=-\frac{\e^3\tau\rho'(t)}{2\rho(t)}\int_{-\rho(t)}^{-\rho(t)+\e^{2+\mu}}(R+2\rho(t))v_t(R,t)^2\phi_R(R,t)\,dR\\
	& \;\quad-\frac{\e\alpha\rho'(t)}{2\rho(t)}\int_{-\rho(t)}^{1-\rho(t)}v_R(R,t)^2 R\,\phi_R(R,t)\,dR.
\end{align*}
Observe that $J^\e_2(t)\leq0$ for all $t\in[0,T]$ because $R\phi_R(R,t)<0$ for $(R,t)\in(-\rho(t),1-\rho(t))\times[0,T]$, 
whereas $\phi_R(R,t)>0$ for $(R,t)\in(-\rho(t),0)\times[0,T]$ and then $J^\e_1(t)\geq0$ for all $t\in[0,T]$.
Let us estimate the term $J_1^\e$. 
To do this, we recall that 
\begin{align*}
	v_t(R,t)^2& =\left(u_t(R+\rho(t),t)+\rho'(t)v_R(R,t)\right)^2\\
	& \leq2\left(u_t(R+\rho(t),t)^2+\rho'(t)^2v_R(R,t)^2\right),
\end{align*}
for all $(R,t)\in(-\rho(t),1-\rho(t))\times(0,T))$.
Hence,
\begin{align*}
	J^\e_1(t)& \leq -\frac{\e^3\tau\rho'(t)}{\rho(t)}\int_{-\rho(t)}^{-\rho(t)+\e^{2+\mu}}v_t(R,t)^2\phi_R(R,t)\,dR\\
	&\leq -\frac{2\e^3\tau\rho'(t)}{\rho(t)}\int_{-\rho(t)}^{-\rho(t)+\e^{2+\mu}}u_t(R+\rho(t),t)^2\phi_R(R,t)\,dR\\
	&\quad -\frac{2\e^3\tau\rho'(t)^3}{\rho(t)}\int_{-\rho(t)}^{-\rho(t)+\e^{2+\mu}}v_R(R,t)^2\phi_R(R,t)\,dR,
\end{align*}
and, as a trivial consequence
\begin{align*}
	J^\e(t)& \leq -\frac{2\e^3\tau\rho'(t)}{\rho(t)}\int_{-\rho(t)}^{-\rho(t)+\e^{2+\mu}}u_t(R+\rho(t),t)^2\phi_R(R,t)\,dR\\
	&\quad +\frac{\e\rho'(t)}{2\rho(t)}\int_{-\rho(t)}^{-\rho(t)+\e^{2+\mu}}\left(-4\e^2\tau\rho'(t)^2-\alpha R\right)v_R(R,t)^2\phi_R(R,t)\,dR\\
	& \quad-\frac{\e\alpha\rho'(t)}{2\rho(t)}\int_{-\rho(t)+\e^{2+\mu}}^{1-\rho(t)}v_R(R,t)^2 R\,\phi_R(R,t)\,dR.
\end{align*}
Choosing $\e$ so small that
\begin{equation*}
	\alpha(\rho(T)-\e^2)\geq4\e^2\tau\rho'(T)^2,
\end{equation*}
we end up with
\begin{equation*}
	J^\e(t)\leq -\frac{2\e^3\tau\rho'(t)}{\rho(t)}\int_{-\rho(t)}^{-\rho(t)+\e^{2+\mu}}u_t(R+\rho(t),t)^2\phi_R(R,t)\,dR=:J^\e_3(t).
\end{equation*}
Now, let us estimate the integral $J^\e_3$.
Using that
\begin{align*}
	\left|\phi_R(R,t)\right| & \leq \frac{(n-1)|R|}{\rho(T)\alpha(R+\rho)}|\phi(R,t)|\leq \frac{C}{R+\rho}(R+\rho)^\frac{n-1}{1-\e^2\tau\rho'(t)^2}\\
	& \leq C(R+\rho)^{n-2}, 
\end{align*}
where $C>0$ depends on $T$, we deduce that there exists $C>0$ (depending on $T$) such that 
\begin{align*}
	J^\e_3(t)&\leq C\e^3\tau\int_{-\rho(t)}^{-\rho(t)+\e^{2+\mu}}u_t(R+\rho(t),t)^2(R+\rho)^{n-2}\,dR\\
	&=C\e^3\tau\int_{0}^{\e^{2+\mu}}u_t(r,t)^2r^{n-2}\,dr.
\end{align*}
Coming back to cartesian coordinates, we obtain
\begin{equation*}
	J^\e_3(t)\leq C\e^3\tau\int_{B(0,\e^{2+\mu})}\frac{u_t(x,t)^2}{|x|}\,dx,
\end{equation*}
where $B(0,\e^{2+\mu})$ is the ball of center $0$ and of radius $\e^{2+\mu}$.
From H\"older's inequality, it follows that
\begin{equation*}
	\int_{B(0,\e^{2+\mu})}\frac{u_t(x,t)^2}{|x|}\,dx\leq\left(\int_{B(0,\e^{2+\mu})}u_t(x,t)^{2q}\,dx\right)^{\frac1q}\left(\int_{B(0,\e^{2+\mu})}\frac{1}{|x|^{q'}}\,dx\right)^\frac{1}{q'},
\end{equation*}
where $\frac1q+\frac1{q'}=1$ and $q'= \frac{q}{q-1} <n$.
For such $q'$, the second integral is bounded as follows:
\begin{equation*}
	\left ( \int_{B(0,\e^{2+\mu})}\frac{1}{|x|^{q'}}\,dx\right)^\frac{1}{q'} \leq C \e^{(2+\mu)\left( \frac{n}{q'} - 1\right )},
\end{equation*}
and therefore it is convenient to choose $q$ big such that $q'$ is as close as possible to $1$. 
In particular, in view of Sobolev inequalities, we choose $q = \frac{2^*}{2} = 3$ for $n=3$, namely, $q' = \frac{3}{2}$, and  $q$ large   so that 
$q'=1+\mu/5$ for $n=2$. 
With these choices, we can say that there exists a constant $C>0$ (depending on $T$) such that
\begin{equation*}
	J^\e_3(t)\leq C\e^\sigma \tau\|u_t(\cdot,t)\|_{L^{2q}(B(0,\e^{2+\mu}))}^2\leq C\e^\sigma\tau\| u_t(\cdot,t)\|_{H^{1}(\Omega)}^2, 
\end{equation*}
for all $t\in(0,T)$, where
\begin{equation*}
	\sigma = \begin{cases} 
			5+\mu\left(\frac{1-\mu}{5+\mu}\right), & \hbox{if}\ n=2, \\
			5+\mu, & \hbox{if}\ n=3.
			\end{cases}
\end{equation*}
Using Proposition \ref{prop:higher} and \eqref{eq:grad-ut}, we conclude that
\begin{equation*}
	h(\e) = \int_0^T J^\e(t)\, dt  \leq C\e^{-5+\sigma},
\end{equation*}
where $\sigma$ is defined above, and then the proof is complete.
\end{proof}
\begin{rem}\label{rem:ass-tau}
If the term $I_6$ of inequality \eqref{eq:E'(t)-I6} is negative, then the functional $E_\phi$ decreases in time
along the solutions $(v^\e,v^\e_t)$ to the problem \eqref{eq:v}-\eqref{eq:v-boundary}, and we need no assumptions on the parameter $\tau>0$.
Also if $I_6$ is positive with $I_6=o(1)$ as $\e\to0$, we must not impose a smallness condition on the parameter $\tau$.
Since we are not able to establish a priori the sign of $I_6$ and we do not have an estimate of $u_t$ near $x=0$, 
we introduce the function $J^\e$ and use the estimate \eqref{eq:grad-ut}.
In this way, we need to impose the condition \eqref{eq:tau-ugly} on the parameter $\tau$ to obtain \eqref{eq:ass-h}.
However, we believe that such condition on the smallness of $\tau$ is indeed technical, as confirmed by numerical evidence; 
for instance in the numerical examples of Section \ref{sec:intro}, Figures \ref{fig:eps=0.02}-\ref{fig:eps=0.01_t=250}, we choose $\tau=1$.
\end{rem}

\subsection{Dynamics of $v^\e$}
Denote by $v^\e$ the solution of \eqref{eq:v} with boundary conditions \eqref{eq:v-boundary} and initial data
\begin{equation}\label{eq:v0}
	v^\e(R,0)=v_0^\e(R), \qquad v_t^\e(R,0)=v_1^\e(R), \quad R\in[-\rho_0,1-\rho_0].
\end{equation}
Similarly to \eqref{eq:u0-ass}, we assume that $v_0^\e$ converges in $L^1$ to $\bar v(R)$ as $\e\to0$, where
\begin{equation}\label{eq:ass-v0}
	 \bar v(R):=\begin{cases} -1, \qquad R<0, \\ +1, \qquad R>0,\end{cases}
\end{equation}
and that  the energy $E_\phi$ at the time $t=0$ satisfies
\begin{align}\label{eq:ass-energy}
	E_\phi[v^\e_0,v_1^\e]&:=\int_{-\rho(0)}^{1-\rho(0)}\left[\frac{\e^3\tau}{2}(v_1^\e)^2+
			\e\left(1-\e^2\tau\rho'(0)^2\right)\frac{\left(v_0^\e\right)_r^2}{2}+\e^{-1}F(v^\e_0)\right]\phi\,dR\notag\\
	&\leq c_0+z(\e),
\end{align}
where $c_0, z$ are the same of \eqref{eq:energy-u}.
Using \eqref{eq:energy-variation}, assumption \eqref{eq:ass-energy} and the positivity of $\phi$, we obtain the following estimate for the energy
\begin{equation}\label{eq:energy-Tbar}
	E_\phi[v^\e,v^\e_t](\bar T)\leq c_0+y(\e),
\end{equation}
for all $\bar T\in[0,T]$, where $y=\max\{z,h\}$. 
In particular, since $z=o(1)$ as $\e\to0^+$, in view of \eqref{eq:ass-h}, we deduce that 
the energy is uniformly bounded for any $\bar T\in[0,T]$.

The function $v^\e$ is defined in the region $[-\rho(t),1-\rho(t)]\times[0,T^\e_m)$; 
however, in the following we shall work in a region 
\begin{equation}\label{eq:region-a}
	[-a, a]\times(0,T)\subset(-\rho(t),1-\rho(t))\times(0,T^\e_m).
\end{equation}
Since $\rho$ is a decreasing function of $t$ and $\rho(t) \geq \rho^o(t)$ for any $t\in [0,\T] \subset [0,T^\e_m)$, 
for any $T\in(0,\T)$ it is possible to choice $a>0$ (depending on $T$) such that \eqref{eq:region-a} is satisfied; 
i.e.\ $a< \min\{\rho^o(T), 1-\rho_0\}$.
The function $\phi$ vanishes only at $R=-\rho$ and so, with this choice of $a$ and $T$, we can say that 
\begin{equation}\label{eq:phi_m}
	\phi(R,t)\geq\phi_m>0 \quad \mbox{ for } (R,t)\in(-a, a)\times(0,T),
\end{equation}
where $\phi_m$ is a constant depending only on $T$ (to be explicitly obtained).
For $0\leq t_1<t_2\leq T$, define
\begin{equation}\label{eq:d^eps}
	d^\e(t_1,t_2):=\int_{-a}^a\big|\Psi(v^\e(R,t_1))-\Psi(v^\e(R,t_2))\big|\,dR,
\end{equation}
where the function $\Psi$ is defined in \eqref{eq:Psi}.
\begin{prop}
Let $(v^\e,v^\e_t)$ be a sufficiently regular solution to the IBVP \eqref{eq:v}-\eqref{eq:v-boundary}-\eqref{eq:v0},
where $\tau$ satisfies \eqref{eq:tau-ugly}, $\e$ is so small that \eqref{eq:alfa} holds, $\rho$ satisfies \eqref{eq:rho}, 
$\phi$ is defined in \eqref{eq:phi}, and the initial data satisfy \eqref{eq:ass-energy}.
Moreover, fix $T\in(0,\T)$ and $a>0$ such that \eqref{eq:region-a} holds. 
Then, there exists a constant $C>0$ (depending on $T$, but not on $\e$) such that 
\begin{equation}\label{eq:d-stima1}
	d^\e(t_1,t_2)\leq C(t_2-t_1)^{1/2}\bigl(E_\phi[v^\e,v^\e_t](t_1)-E_\phi[v^\e,v^\e_t](t_2)+h(\e)\bigr)^{1/2},
\end{equation}
whenever $0\leq t_1<t_2\leq T$.
\end{prop}
\begin{proof}
Fix $t_1$, $t_2$ satisfying $0\leq t_1<t_2\leq T$.
Thanks to Cauchy--Schwarz inequality we obtain
\begin{align*}
	\int_{t_1}^{t_2}&\int_{-\rho(t)}^{1-\rho(t)}\left|\frac{d}{dt}\Psi(v^\e)\right|\phi\,dtdR =\int_{t_1}^{t_2}\int_{-\rho(t)}^{1-\rho(t)}\left|\Psi'(v^\e)v^\e_t\right|\phi\,dtdR\\
	&\;\;\leq\left(\int_{t_1}^{t_2}\int_{-\rho(t)}^{1-\rho(t)}\Psi'(v^\e)^2\phi\,dtdR\right)^{1/2}\left(\int_{t_1}^{t_2}\int_{-\rho(t)}^{1-\rho(t)}(v^\e_t)^2\phi\,dtdR\right)^{1/2}.
\end{align*}
Using that $\Psi'(v^\e)^2\phi=2F(v^\e)\phi\leq2\e E_\phi[v^\e,v^\e_t]$ and \eqref{eq:energy-Tbar}, we deduce that the energy is uniformly bounded in time and therefore
\begin{equation*}
	\int_{t_1}^{t_2}\int_{-\rho(t)}^{1-\rho(t)}\Psi'(v^\e)^2\phi\,dtdR\leq2\e\int_{t_1}^{t_2}E_\phi[v^\e,v^\e_t](t)\,dt\leq C\e(t_2-t_1),
\end{equation*}
where $C$ is a positive constant depending on $T$.
Moreover, integrating \eqref{eq:E'(t)} we infer
\begin{equation*}
	\int_{t_1}^{t_2}\int_{-\rho(t)}^{1-\rho(t)}(v^\e_t)^2\phi\,dtdR\leq (\beta\e)^{-1}\left(E_\phi[v^\e,v^\e_t](t_1)-E_\phi[v^\e,v^\e_t](t_2)+h(\e)\right).
\end{equation*}
Then, we get
\begin{align}
	\int_{t_1}^{t_2}\int_{-\rho(t)}^{1-\rho(t)}\left|\frac{d}{dt}\Psi(v^\e)\right|\phi\,dtdR\leq C(t_2-t_1)^{1/2}&\biggl(E_\phi[v^\e,v^\e_t](t_1) \nonumber\\
	& -E_\phi[v^\e,v^\e_t](t_2)+h(\e)\biggr)^{1/2}, \label{eq:ineq1}
\end{align}
for some $C>0$ depending on $T$.
On the other hand, we have
\begin{equation*}
	d^\e(t_1,t_2)\leq\int_{-a}^a\int_{t_1}^{t_2}\left|\frac{d}{dt}\Psi(v^\e)\right|\,dtdR,
\end{equation*}
and so, using \eqref{eq:phi_m} we infer
\begin{align}
	d^\e(t_1,t_2)&\leq(\phi_m)^{-1}\int_{-a}^a\int_{t_1}^{t_2}\left|\frac{d}{dt}\Psi(v^\e)\right|\phi\,dtdR\nonumber\\
	&\leq(\phi_m)^{-1}\int_{t_1}^{t_2}\int_{-\rho(t)}^{1-\rho(t)}\left|\frac{d}{dt}\Psi(v^\e)\right|\phi\,dtdR. \label{eq:ineq2}
\end{align}
Combining \eqref{eq:ineq1} and \eqref{eq:ineq2}, we end up with \eqref{eq:d-stima1} and the proof is complete.
\end{proof}
\begin{rem}\label{rem:Holder}
 Denoting by $d^\e(t)$ the function $d^\e(0,t)$ for $t\in[0,T]$, then
\begin{equation*}
	|d^\e(t_1)-d^\e(t_2)|\leq d^\e(t_1,t_2).
\end{equation*}
Thus, estimate \eqref{eq:d-stima1} implies $d^\e$ is an H\"older continuous function in $t$, uniformly in $\e$ for \eqref{eq:energy-Tbar}.
\end{rem}
The next step is to establish a lower bound for $E_\phi[v^\e,v^\e_t]$.
From Young inequality it follows that
\begin{equation}\label{eq:Young}
	\e\frac{v_R^2}{2}+\e^{-1}F(v)\geq\sqrt{2F(v)}|v_R|=\left|\frac{d}{dR}\Psi(v)\right|.
\end{equation}
The lower bound \eqref{eq:Young} is fundamental in the proof of the following result.
We use the notation
\begin{equation*}
	P_\phi[v^\e](t):=\int_{-\rho(t)}^{1-\rho(t)}\left[\e\frac{v^\e_R(R,t)^2}{2}+\e^{-1}F(v^\e(R,t))\right]\phi(R,t)\,dR.
\end{equation*}
\begin{prop}\label{prop:lowerbound}
Let $P_\phi[v^\e]$ defined above with $\rho$ satisfying \eqref{eq:rho}, $F$ satisfying \eqref{eq:ass-F-radial} and $\phi$ defined in \eqref{eq:phi}.
Fix $T\in(0,\T)$ and $a>0$ such that \eqref{eq:region-a} holds.
Then, there exist positive constants $\e_0, C_1, C_2$ (independent on $\e$) such that
\begin{equation}\label{eq:lowerbound}
		P_\phi[v^\e](t)\geq\phi(-C_1d^\e(t)-\e^{1/2},t)\cdot(c_0-C_2\e^{1/2}),
\end{equation}
for any $\e\in(0,\e_0)$ and $t\in[0,T]$ such that
\begin{equation}\label{eq:cond-lower}
	C_1d^\e(t)+\e^{1/2}\leq a.
\end{equation}
Here, the positive constant $c_0$ is the same of \eqref{eq:ass-energy}, $\e_0$ and $C_2$ depend on $T$, 
whereas the constant $C_1$ can be chosen independent on $\e$ and $T$.
\end{prop}

\begin{proof}
The first step of the proof is to prove the existence of two points $R_1$, $R_2$ in a neighborhood of $0$ such that
$v^\e$ is \emph{close} to $-1$ in $R_1$ and $v^\e$ is \emph{close} to $1$ in $R_2$.
To do this, let us define $$A:=(-C_1d^\e(t)-\e^{1/2},C_1d^\e(t)+\e^{1/2}),$$
and fix $t\in[0,T]$ such that assumption \eqref{eq:cond-lower} is satisfied, namely such that $A\subset(-a,a)$.
We claim that there exist $R_1,R_2\in A$ such that 
\begin{equation}\label{eq:claim-Ri}
	v^\e(R_1,t)\leq-1+C\e^{1/4}, \qquad \quad v^\e(R_2,t)\geq1-C\e^{1/4},
\end{equation}
for some constant $C>0$.
To start with, we prove the existence of $R_1$ such that the first inequality of \eqref{eq:claim-Ri} holds.
Let us introduce
\begin{align*}
	I^- & :=(-C_1d^\e(t)-\e^{1/2},0)\cap\left\{R: v^\e(R,t)<\tfrac14\right\},\\
	I^+ &:=(-C_1d^\e(t)-\e^{1/2},0)\cap\left\{R: v^\e(R,t)\geq\tfrac14\right\}.
\end{align*}
From the assumption \eqref{eq:cond-lower} and recalling the definition of $d^\e(t):=d^\e(0,t)$ 
where $d^\e(t_1,t_2)$ is defined in \eqref{eq:d^eps}, we deduce
\begin{align*}
	d^\e(t)&=\int_{-a}^a\big|\Psi(v^\e_0(R))-\Psi(v^\e(R,t))\big|\,dR\geq\int_{I^+}\big|\Psi(v^\e_0(R))-\Psi(v^\e(R,t))\big|\,dR\\
	&\geq\{\Psi(1/4)-\Psi(0)\}\,m(I^+),
\end{align*}
where we used that $v^\e(R,t)\geq\frac14$ in $I^+$ and $v^\e_0(R)<0$ if $\e$ is sufficiently small and $R\in I^{+}$ for \eqref{eq:ass-v0}.
Since $m(I^+)=C_1d^\e(t)+\e^{1/2}-m(I^-)$, we get
\begin{equation*}
	\left(\{\Psi(1/4)-\Psi(0)\}^{-1}-C_1\right)d^\e(t)\geq \e^{1/2}-m(I^-).
\end{equation*}
Taking $C_1\geq\{\Psi(1/4)-\Psi(0)\}^{-1}$, we obtain 
\begin{equation}\label{eq:m(I)}
	m(I^-)\geq\e^{1/2}.
\end{equation}
Moreover, we have
\begin{equation}\label{eq:C3}
	\min_{I^-}\phi\int_{I^-}\e^{-1}F(v^\e)\,dR\leq\int_A\e^{-1}F(v^\e)\phi\,dR\leq P_\phi[v^\e]\leq C_3,
\end{equation}
where in the last estimate we assumed without loss of generality that $P_\phi[v^\e]\leq C_3$ for some $C_3>0$.
Since $I^-\subset A\subset(-a,a)$, we can use \eqref{eq:phi_m} and the estimates \eqref{eq:m(I)}, \eqref{eq:C3} imply
the existence of $R_1\in I^-$ such that
\begin{equation}\label{eq:F-R1}
	F(v^\e(R_1,t))\leq C_3 (\phi_m)^{-1}\e^{1/2}.
\end{equation}
Using the assumptions on $F$ \eqref{eq:ass-F-radial}, we infer that there exists $\beta_0>0$ such that
\begin{equation}\label{eq:F-quadratic}
	\frac{F''(-1)}{4}(v^\e+1)^2\leq F(v^\e)\leq F''(-1)(v^\e+1)^2,
\end{equation}
for any $v^\e\in[-1-\beta_0,-1+\beta_0]$.
Assume $\e$ sufficiently small so that \eqref{eq:F-R1} and the fact that $v^\e(R_1,t)<\frac14$ in $I^-$ imply $v^\e\in[-1-\beta_0,-1+\beta_0]$.
Thus, from \eqref{eq:F-R1} and \eqref{eq:F-quadratic}, it follows that $v^\e(R_1,t)\leq-1+C\e^{1/4}$,
that is the first inequality of \eqref{eq:claim-Ri}.
The second one can be proved similarly. 
Now, we shall use \eqref{eq:claim-Ri} and \eqref{eq:Young} to complete the proof of \eqref{eq:lowerbound}. 
Precisely, we have 
\begin{align}
	P_\phi[v^\e](t)&\geq\int_{R_1}^{R_2}\left[\e\frac{v^\e_R(R,t)^2}{2}+\e^{-1}F(v^\e(R,t))\right]\phi(R,t)\,dR \notag \\
	&\geq\int_{R_1}^{R_2}\left|\frac{d}{dR}\Psi(v^\e(R,t))\right|\phi(R,t)\,dR \notag \\
	&\geq\min_A\phi\,\big|\Psi(v^\e(R_2,t))-\Psi(v^\e(R_1,t))\big|, \label{eq:Plower1}
\end{align}
using \eqref{eq:Young} in the second step.
For the last term, using \eqref{eq:claim-Ri} we infer
\begin{align*}
	\big|\Psi(v^\e(R_2,t))-\Psi(v^\e(R_1,t))\big|&\geq\int_{-1+C\e^{1/4}}^{1-C\e^{1/4}}\sqrt{2F(s)}\,ds\\
	&=c_0-\int_{-1}^{-1+C\e^{1/4}}\sqrt{2F(s)}\,ds\\
	&\qquad\quad-\int_{1-C\e^{1/4}}^1\sqrt{2F(s)}\,ds. 
\end{align*}
Using again the assumptions on $F$ \eqref{eq:ass-F-radial} and the upper bound for $F$ in \eqref{eq:F-quadratic}, 
we deduce that there exists $C>0$ depending on $T$ such that for any $\e$ sufficiently small
\begin{equation*}
	\int_{-1}^{-1+C\e^{1/4}}\sqrt{2F(s)}\,ds\leq\sqrt{2F''(-1)}\int_{-1}^{-1+C\e^{1/4}}|s+1|\,ds\leq C\e^{1/2}.
\end{equation*}
A similar results holds true for the last integral, and therefore we obtain that there exists $C>0$ depending on $T$ such that
\begin{equation}\label{eq:Plower2}
	\big|\Psi(v^\e(R_2,t))-\Psi(v^\e(R_1,t))\big|\geq c_0-C\e^{1/2}.
\end{equation}
It remains to study the term $\displaystyle\min_A\phi$ in \eqref{eq:Plower1}.
Since $\phi$ satisfies \eqref{eq:prop-phi2} and \eqref{eq:prop-phi_R}, we conclude that
\begin{equation}\label{eq:min_A}
	\min_A\phi=\phi(-C_1d^\e(t)-\e^{1/2},t).
\end{equation}
Substitute \eqref{eq:Plower2} and \eqref{eq:min_A} in \eqref{eq:Plower1} and the proof is complete.
\end{proof}
\begin{rem}
From \eqref{eq:lowerbound} and \eqref{eq:prop-phi3}, it follows that if $C_1d^\e(t)+\e^{1/2}$ is sufficiently small then
\begin{equation*}
	P_\phi[v^\e](t)\geq\left(1-K_T(C_1d^\e(t)+\e^{1/2})^2\right)\cdot(c_0-C_2\e^{1/2}).
\end{equation*}
Using the definition of the energy \eqref{eq:energy-v}, we deduce that there exists $C>0$ (depending on $T$ but not on $\e$ and $\tau$) such that
\begin{align*}
	E_\phi[v^\e,v^\e_t](t)&\geq P_\phi[v^\e](t)-\frac12\e^3\tau\rho'(t)^2\int_{-\rho(t)}^{1-\rho(t)}v^\e_R(R,t)^2\phi(R,t)\,dR\\
	&\geq P_\phi[v^\e](t)-C\alpha^{-1}\e^2\tau,
\end{align*}
where in the last passage we used \eqref{eq:rho'-limitato}, \eqref{eq:alfa} and \eqref{eq:energy-Tbar}.
Hence, we end up with the following lower bound
\begin{equation}\label{eq:lower}
	E_\phi[v^\e,v^\e_t](t)\geq c_0-C\e^{1/2}-Cd^\e(t)^2,
\end{equation}
which holds for any sufficiently small $\e$.
Actually, we have proved a property stronger than \eqref{eq:lower};
we have proved the following ``local'' bound
\begin{equation}\label{eq:lower2}
	\int_A \left[\e\frac{v^\e_R(R,t)^2}{2}+\e^{-1}F(v^\e(R,t))\right]\phi(R,t)\,dR\geq c_0-C\e^{1/2}-Cd^\e(t)^2,
\end{equation}
which we will use later.
\end{rem}
The next step is to prove the following fundamental result.
\begin{prop}\label{prop:main-v}
Let $(v^\e,v^\e_t)$ be a sufficiently regular solution to the IBVP \eqref{eq:v}-\eqref{eq:v-boundary}-\eqref{eq:v0},
where  $\tau$ satisfies \eqref{eq:tau-ugly}, $\rho$ satisfies \eqref{eq:rho},
$\phi$ is defined in \eqref{eq:phi}, and the initial data satisfy \eqref{eq:ass-energy}.
Moreover, fix $T\in(0,\T)$ and $a>0$ such that \eqref{eq:region-a} holds. 
Then, there exists a constant $C>0$ (depending on $T$, but not on $\e$)  such that 
\begin{equation}\label{eq:d^eps-estimate}
	d^\e(t)\leq C\max\left\{\e^{1/4},\sqrt{y(\e)}\right\}, \qquad \qquad \mbox{ for } \quad t\in[0,T],
\end{equation}
provided $\e\in(0,\e_0)$, where $\e_0$ is a small constant depending on $T$ such that in particular \eqref{eq:alfa} holds. 
\end{prop}
\begin{proof}
We shall prove inductively that if $\e\in(0,\e_0)$ then
\begin{equation}\label{eq:induzione}
	d^\e(t)\leq 2^N\max\left\{\e^{1/4},\sqrt{y(\e)}\right\}, \qquad \qquad \mbox{ for } \quad t\in[0,\min(Nt_0,T)],
\end{equation}
where $N$ is a positive integer, for a suitable choice of the constant $t_0>0$.
The constant $\e_0$ will depend on $N$, but the value of $t_0$ will be independent on $N$ and $\e$.
The estimate \eqref{eq:d^eps-estimate} clearly follows from \eqref{eq:induzione}, by taking $N=[T/t_0]+1$.

Let us start with the first step of the induction, namely let us show \eqref{eq:induzione} for $N=1$.
From \eqref{eq:d-stima1}, it follows that
\begin{equation*}
	d^\e(t)\leq Ct^{1/2}\bigl(E_\phi[v^\e_0,v^\e_1]-E_\phi[v^\e,v^\e_t](t)+h(\e)\bigr)^{1/2}.
\end{equation*}
In order to estimate the latter quantity, we use assumption \eqref{eq:ass-energy} and Proposition \ref{prop:lowerbound}.
As observed in Remark \ref{rem:Holder}, the function $d^\e$ is H\"older continuous in $t$, uniformly in $\e$;
hence, we can choose $t_0>0$ sufficiently small such that assumption \eqref{eq:cond-lower} is satisfied on $[0,t_0]$ for small values of $\e$. 
Combining \eqref{eq:ass-energy} and \eqref{eq:lower}, we obtain for all $t\in[0,t_0]$
\begin{equation*}
	d^\e(t)^2 \leq Ct\biggl(c_0+z(\e)-c_0+C\e^{1/2}+Cd^\e(t)^2+h(\e)\biggr),
\end{equation*}
where the positive constant $C$ depends on $T$, but not on $\e$.
Hence, we have
\begin{equation*}
	d^\e(t)^2\leq Ct\bigl(d^\e(t)^2+\e^{1/2}+y(\e)\bigr),
\end{equation*}
and as a consequence
\begin{equation*}
	(1-Ct_0)d^\e(t)^2\leq Ct_0\bigl( \e^{1/2}+y(\e)\bigr) ,
\end{equation*}
for all $t\in[0,t_0]$.
Choosing $t_0$ small enough so that $0<Ct_0/(1-Ct_0)\leq2$, we end up with \eqref{eq:induzione} in the case $N=1$.

Now, let us proceed with the inductive step.
Assume that \eqref{eq:induzione} holds for $N=1,2,\dots,k$ and that $kt_0\leq T$ (otherwise \eqref{eq:induzione} is trivial).
As in the initial step we use \eqref{eq:ass-energy} and Proposition \ref{prop:lowerbound}.
The condition \eqref{eq:cond-lower} is satisfied because 
\begin{equation*}
	d^\e(t)\leq|d^\e(t)-d^\e(kt_0)|+d^\e(kt_0), \qquad \quad \forall\,t\in[kt_0,\min((k+1)t_0,T)],
\end{equation*}
and we can use the H\"older continuity of $d^\e$ and the inductive hypothesis, by choosing $t_0$ and $\e$ sufficiently small.
Moreover, we have 
\begin{equation*}
	d^\e(t)\leq d^\e(kt_0,t)+d^\e(kt_0), \qquad \quad \forall\,t\in[kt_0,\min((k+1)t_0,T)],
\end{equation*}
and so
\begin{equation*}
	d^\e(t)^2\leq 2d^\e(kt_0,t)^2+2d^\e(kt_0)^2\leq2d^\e(kt_0,t)^2+2^{2k+1}\max\left\{\e^{1/2},y(\e)\right\}, 
\end{equation*}
for all $t\in[kt_0,\min((k+1)t_0,T)]$, where in the last inequality we used the inductive hypothesis.
Let us estimate the remaining term.
From \eqref{eq:d-stima1}, it follows that
\begin{equation*}
	d^\e(kt_0,t)^2\leq C(t-kt_0)\bigl(E_\phi[v^\e,v^\e_t](kt_0)-E_\phi[v^\e,v^\e_t](t)+h(\e)\bigr),
\end{equation*}
for all $t\in[kt_0,\min((k+1)t_0,T)]$.
Using \eqref{eq:energy-Tbar} and \eqref{eq:lower}, we obtain
\begin{equation*}
	d^\e(kt_0,t)^2\leq Ct_0\left(d^\e(t)^2+\e^{1/2} +y(\e)\right), \quad \quad \forall\,t\in[kt_0,\min((k+1)t_0,T)],
\end{equation*}
where the positive constant $C$ depends on $T$ but not on $\e$ and $t_0$.
Hence, we get 
\begin{equation*}
	d^\e(t)^2\leq Ct_0\left(d^\e(t)^2+\e^{1/2} + y(\e) \right)+2^{2k+1}\max\left\{\e^{1/2},y(\e)\right\}, 
\end{equation*}
for all $t\in[kt_0,\min((k+1)t_0,T)]$, and, as an easy consequence 
\begin{equation*}
	(1-Ct_0)d^\e(t)^2\leq2(Ct_0+2^{2k})\max\left\{\e^{1/2},y(\e)\right\}, 
\end{equation*}
for all $t\in[kt_0,\min((k+1)t_0,T)]$.
By choosing $t_0$ sufficiently small so that $Ct_0 \leq1/4$, we conclude that 
\begin{equation*}
	d^\e(t)^2\leq \left(\frac23+\frac432^{2k+1}\right)\max\left\{\e^{1/2},y(\e)\right\}\leq2^{2k+2}\max\left\{\e^{1/2},y(\e)\right\},
\end{equation*}
for all $t\in[kt_0,\min((k+1)t_0,T)]$, that is \eqref{eq:induzione} with $N=k+1$, and the proof is complete.
\end{proof}

\subsection{Proof of the main results}
In this subsection, we conclude the proof of Theorems \ref{thm:main} and \ref{cor:main}.
Before proceeding with the proof, let us make some comments.
Firstly, we remark that the condition \eqref{eq:energy-u} on the initial data allows us to make use of Proposition \ref{prop:main-v}.
Indeed, using the change of variables \eqref{eq:change-var} 
\begin{equation*}
	v_0^\e(R)=u_0^\e(R+\rho_0), \qquad \quad v_1^\e(R)=u_1^\e(R+\rho_0)+\nu_0\partial_ru_0^\e(R+\rho_0),
\end{equation*}
we obtain that \eqref{eq:u0-ass} and \eqref{eq:energy-u} are equivalent to \eqref{eq:ass-v0} and \eqref{eq:ass-energy}, in the case $\nu_0=0$.

Secondly, as we have already observed, the assumption \eqref{eq:energy-u} ensures that
the condition \eqref{eq:energialimitata} holds.
Moreover, the condition \eqref{eq:u-bounded} permits to remove the assumption \eqref{eq:ass-F-inf} on $F$ as pointed out 
in Remark \ref{rem:ass-F-inf} and then we can apply Theorem \ref{thm:compactness} to the solution 
of the problem \eqref{eq:ass-F-radial}-\eqref{eq:u-bounded} introduced in Section \ref{sec:radial}.

Observe also that since the function $h$ satisfies \eqref{eq:ass-h}, we have that $y=o(1)$ as $\e\to0$.
Substituting in \eqref{eq:d^eps-estimate} we obtain 
\begin{equation}\label{eq:d^eps->0}
	\lim_{\e\to0}d^\e(t)=0, \qquad \quad \mbox{ for any } t\in[0,T].
\end{equation}
Finally, let us recall the definition
\begin{equation*}
	\omega^0(r,t)=\left\{\begin{array}{ll} -1, \qquad r<\rho^o(t),\\
	+1, \qquad r>\rho^o(t),
	\end{array}\right.
\end{equation*}
with $\rho^o(t)=\sqrt{\rho_0^2-2(n-1)t}$. 
We have
\begin{equation}\label{eq:omega^e-omega0}
	|\omega^\e(r,t)-\omega^0(r,t)|=\left\{\begin{array}{ll} 2, \qquad \rho^o(t)\leq r\leq\rho(t),\\
	0, \qquad \quad \mbox{ otherwise},
	\end{array}\right.
\end{equation}
Therefore, from Lemma \ref{lem:epstau-rho} it follows that $\omega^\e\to\omega^0$ as $\e\to0$.
Now, we have all the tools to prove Theorem \ref{thm:main}.
\begin{proof}[Proof of Thereom \ref{thm:main}]
Fix $T\in(0,\T)$.
We will prove the property \eqref{eq:main} by contradiction.
If \eqref{eq:main} is not true, then there exists a sequence $\e_j$ and a constant $\delta>0$ such that 
\begin{equation}\label{eq:contradiction}
	\int_0^T\int_0^1\left|u^\e(r,t)-\omega^\e(r,t)\right|r^{n-1}\,dr\,dt\geq\delta.
\end{equation}
The assumptions of Theorem \ref{thm:compactness} are satisfied, then we can apply it to the solution $u^{\e_j}$ and we can state that
there exists a subsequence (still denoted $u^{\e_j}$) such that
\begin{equation}\label{eq:subsequence}
	\lim_{\e_j\to0} u^{\e_j}(r,t)=u^*(r,t) \qquad \quad \mbox{ for a.e. }\, (r,t)\in(0,1)\times(0,T),
\end{equation}
where $u^*$ takes only the values $\pm1$.
Regarding $\omega^\e$, from \eqref{eq:rho-rho0} and \eqref{eq:omega^e-omega0} it follows that 
\begin{equation}\label{eq:omega^e}
	\lim_{\e_j\to0} \omega^{\e_j}(r,t)=\omega^0(r,t) \qquad \quad \mbox{ for any }\, (r,t)\in(0,1)\times(0,T).
\end{equation}
Using the assumption \eqref{eq:u-bounded} and \eqref{eq:subsequence}-\eqref{eq:omega^e}, 
we may pass to the limit as $\e_j\to0$ in \eqref{eq:contradiction} and conclude that
\begin{equation}\label{eq:contradiction2}
	\int_0^T\int_0^1\left|u^*(r,t)-\omega^0(r,t)\right|r^{n-1}\,dr\,dt\geq\delta.
\end{equation}
We will show that $u^*\equiv\omega^0$ and so that \eqref{eq:contradiction2} can not be true.
Consider the functions $v^{\e_j}$ and $v^*$ corresponding to $u^{\e_j}$ and $u^*$ through the change of variables \eqref{eq:change-var}:
\begin{align*}
	&v^\e(R,t)=u^\e(R+\rho(t),t), & &(R,t)\in(-\rho(t),1-\rho(t))\times(0,T);  \\
	&v^*(R,t)=u^*(R+\rho^o(t),t),  & &(R,t)\in(-\rho^o(t),1-\rho^o(t))\times(0,T).
\end{align*}
From the assumptions on the initial data, it follows that the function $v^{\e_j}_0=v^{\e_j}_0(R)=v^{\e_j}(R,0)$ satisfies 
\begin{equation*}
	\lim_{\e_j\to0}v^{\e_j}_0(R)=\begin{cases} -1, \qquad R<0, \\ +1, \qquad R>0.\end{cases}
\end{equation*}
On the other hand, thanks to Proposition \ref{prop:main-v} (see \eqref{eq:d^eps->0}) we can state that
\begin{equation*}
	\lim_{\e_j\to0}\int_{-a}^a\big|\Psi(v^{\e_j}_0(R))-\Psi(v^{\e_j}(R,t))\big|\,dR=0.
\end{equation*}
Applying the dominated convergence theorem, we conclude that
\begin{equation}\label{eq:v*}
	v^*(R,t)=\begin{cases} -1, \qquad R<0, \\ +1. \qquad R>0,\end{cases} \qquad \quad \forall\, (R,t)\in(-a,a)\times(0,T).
\end{equation}
This implies that $u^*=\omega^0$ in $(r,t)\in(\rho^o(t)-a,\rho^o(t)+a)\times(0,T)$.
In order to handle values of $R$ outside of $(-a,a)$ we use \eqref{eq:energy-Tbar} and the ``local'' lower bound \eqref{eq:lower2};
setting $A=(-d^\e(t)-\e^{1/2},d^\e(t)+\e^{1/2})$, we have
\begin{equation}\label{eq:ineq-promain1}
	\int_{(-\rho(t),1-\rho(t))\backslash A}\left[\e\frac{v^\e_R(R,t)^2}{2}+\e^{-1}F(v^\e(R,t))\right]\phi(R,t)\,dR\leq C\left(y(\e)+\e^{1/2}+(d^\e)^2\right).
\end{equation}
In particular, we deduce
\begin{equation*}
	\int_{-\rho(t)}^{-d^\e(t)-\e^{1/2}}\left[\e\frac{v^\e_R(R,t)^2}{2}+\e^{-1}F(v^\e(R,t))\right]\phi(R,t)\,dR\leq C\left(y(\e)+\e^{1/2}+(d^\e)^2\right).
\end{equation*}
Using \eqref{eq:Young} and \eqref{eq:d^eps-estimate}, we infer
\begin{equation}\label{eq:ineq-promain2}
	\int_{-\rho(t)}^{-d^\e(t)-\e^{1/2}}\left|\frac{d}{dR}\Psi(v^\e(R,t))\right|\phi(R,t)\,dR\leq C\left(y(\e)+\e^{1/2}\right).
\end{equation}
Observe that the function $\phi$ defined in \eqref{eq:phi} is strictly positive and vanishes only at $-\rho$;
then, for any $\eta\in(0,\rho(T))$ we have
\begin{equation*}
	\phi(R,t)\geq \phi_m>0 \qquad \quad \forall\, (R,t)\in[-\rho(t)+\eta,0]\times[0,T],
\end{equation*}
where the constant $\phi_m$ can be chosen only depending on $\eta$.
For example, in view of \eqref{eq:alfa}, for $\e$ sufficiently small we can choose 
\begin{equation*}
	\phi_m=\eta^{\frac{n-1}{\alpha}}.
\end{equation*}
Therefore, using \eqref{eq:d^eps->0} and \eqref{eq:ineq-promain2}, we can say that for any fixed $t\in[0,T]$ 
and for any two points $R_1, R_2\in(-\rho^o(t),0)$, one has 
\begin{equation*}
	\phi_m\left|\Psi(v^\e(R_2,t))-\Psi(v^\e(R_1,t))\right|\leq C\left(y(\e)+\e^{1/2}\right),
\end{equation*}
whenever $\e$ is sufficiently small.
Since $\phi_m$ is strictly positive, by passing to the limit $\e_j\to0$, we obtain that $\Psi(v^*)$ is constant on $(-\rho^o(t),0)$.
From the definition of $\Psi$ \eqref{eq:Psi} and \eqref{eq:v*}, we conclude that $v^*(R,t)=-1$ for $(R,t)\in(-\rho^o(t),0)\times(0,T)$.
Similarly, we can prove that $v^*(R,t)=+1$ for $(R,t)\in(0,1-\rho^o(t))\times(0,T)$.
Indeed, using \eqref{eq:ineq-promain1} we can also say that
\begin{equation*}
	\int_{d^\e(t)+\e^{1/2}}^{1-\rho(t)}\left|\frac{d}{dR}\Psi(v^\e(R,t))\right|\phi(R,t)\,dR\leq C\left(y(\e)+\e^{1/2}\right).
\end{equation*}
Since $\phi$ is strictly positive in $[0,1-\rho(t)]$, $\rho$ satisfies \eqref{eq:rho-rho0} and $d^\e$ satisfies \eqref{eq:d^eps->0}, 
for any fixed $t\in[0,T]$ and for any two points $R_1, R_2\in(0,1-\rho^o(t))$, we get 
\begin{equation*}
	\left|\Psi(v^\e(R_2,t))-\Psi(v^\e(R_1,t))\right|\leq C\left(y(\e)+\e^{1/2}\right),
\end{equation*}
whenever $\e$ is sufficiently small.
Therefore, $\Psi(v^*)$ is constant on $(0,1-\rho^o(t))$ and from \eqref{eq:v*},
we have that $v^*(R,t)=+1$ for $(R,t)\in(0,1-\rho^o(t))\times(0,T)$.
In conclusion, returning to the original variables, we have shown that $u^*=\omega^0$ in $(0,1)\times(0,T)$.
This contradicts \eqref{eq:contradiction2} and the proof is complete.
\end{proof}
Now, we proceed with the proof of Theorem \ref{cor:main}.
\begin{proof}[Proof of Theorem \ref{cor:main}]
Fix $T\in(0,\T)$. 
From triangle inequality, it follows that
\begin{align*}
	\int_0^T\int_0^1\left|u^\e(r,t)-\omega^0(r,t)\right|r^{n-1}\,dr\,dt\leq&\int_0^T\int_0^1\left|u^\e(r,t)-\omega^\e(r,t)\right|r^{n-1}\,dr\,dt\\
	&+\int_0^T\int_0^1\left|\omega^\e(r,t)-\omega^0(r,t)\right|r^{n-1}\,dr\,dt.
\end{align*}
The first term of the right hand side of the previous inequality tends to $0$ as $\e\to0$ for \eqref{eq:main}.
For the other one, we use \eqref{eq:omega^e-omega0}:
\begin{equation*}
	\int_0^T\int_0^1\left|\omega^\e(r,t)-\omega^0(r,t)\right|r^{n-1}\,dr\,dt=\frac2n\int_0^T\!\!\left[\rho(t)^n-\rho^o(t)^n\right]\,dt
	\leq 2T\sup_{t\in[0,T]}\left|\rho(t)-\rho^o(t)\right|.
\end{equation*}
Therefore, using \eqref{eq:main} and \eqref{eq:rho-rho0} we obtain \eqref{eq:corollary}.
\end{proof}

\end{document}